\documentclass[preprint,12pt,superscriptaddress]{revtex4}
\usepackage{ae} 
\usepackage[T1]{fontenc}
\usepackage{psfrag}

\usepackage[ansinew]{inputenc}
\usepackage{amsmath}
\usepackage{graphicx}
\usepackage{color}
\usepackage[english]{babel}
\usepackage{amsfonts}
\usepackage{amssymb}
\usepackage{amsthm}
\usepackage{newlfont}
\usepackage{eurosym}
\DeclareInputText{128}{\euro} 
\usepackage{lscape} 
\usepackage{subfigure}

\theoremstyle{plain}

 \newtheorem{thm}{Theorem}[section]
 
 \newtheorem{lem}[thm]{Lemma}
 \newtheorem{prop}[thm]{Proposition}

\theoremstyle{definition}

\theoremstyle{remark}

\newcommand{\interior}{\operatorname{int}}

\newcommand{\dist}{\operatorname{dist}}
\newcommand{\C}{\mathbb{C}}

\newcommand{\D}{\mathbb{D}}
\newcommand{\R}{\mathbb{R}}
\newcommand{\T}{\mathbb{T}}

\newcommand{\Z}{\mathbb{Z}}

\newcommand{\XX}{\mathcal{X}}

\newcommand{\eps}{\varepsilon}

\newcommand{\mR}{\mathbb{R}}

\newcommand{\figref}[1]{Figure~\ref{fig:#1}}
\newcommand{\secref}[1]{Section~\ref{sec:#1}}
\newcommand{\appref}[1]{Appendix~\ref{app:#1}}
\newcommand{\lemref}[1]{Lemma~\ref{lem:#1}}
\newcommand{\thmref}[1]{Theorem~\ref{thm:#1}}
\renewcommand{\eqref}[1]{(\ref{eq:#1})}

\begin{document}

\title{Extreme value laws in dynamical systems under physical observables}

\author{Mark P. Holland}
\author{Renato Vitolo}
\email{r.vitolo@exeter.ac.uk}
\affiliation{
College of Engineering, Mathematics and Physical Sciences, University of
Exeter,
Exeter, UK
}%
\author{Pau Rabassa}
\affiliation{
Johann  Bernoulli Institute for the Mathematical Sciences,
Rijksuniversiteit Groningen,
Groningen,
The Netherlands
}%
\author{Alef E. Sterk}
\affiliation{
College of Engineering, Mathematics and Physical Sciences, University of
Exeter,
Exeter, UK
}%
\author{Henk Broer}
\affiliation{
Johann  Bernoulli Institute for the Mathematical Sciences,
Rijksuniversiteit Groningen,
Groningen,
The Netherlands
}%

\begin{abstract}
  Extreme value theory for chaotic deterministic dynamical systems is a
  rapidly expanding area of research. Given a system and a real function
  (observable) defined on its phase space, extreme value theory studies the
  limit probabilistic laws obeyed by large values attained by the observable
  along orbits of the system. Based on this theory, the so-called block
  maximum method is often used in applications for statistical prediction of
  large value occurrences. In this method, one performs statistical inference
  for the parameters of the Generalised Extreme Value (GEV) distribution,
  using maxima over blocks of regularly sampled observations along an orbit of
  the system.  The observables studied so far in the theory are expressed as
  functions of the distance with respect to a point, which is assumed to be a
  density point of the system's invariant measure. However, this is not the
  structure of the observables typically encountered in physical applications,
  such as windspeed or vorticity in atmospheric models. In this paper we
  consider extreme value limit laws for observables which are not functions of
  the distance from a density point of the dynamical system.  In such cases,
  the limit laws are no longer determined by the functional form of the
  observable and the dimension of the invariant measure: they also depend on
  the specific geometry of the underlying attractor and of the observable's
  level sets.  We present a collection of analytical and numerical results,
  starting with a toral hyperbolic automorphism as a simple template to
  illustrate the main ideas.  We then formulate our main results for a
  uniformly hyperbolic system, the solenoid map.  We also discuss
  non-uniformly hyperbolic examples of maps (H\'enon and Lozi maps) and of
  flows (the Lorenz63 and Lorenz84 models). Our purpose is to outline the main
  ideas and to highlight several serious problems found in the numerical
  estimation of the limit laws.
\end{abstract}

\maketitle

\tableofcontents

\section{Introduction} 
\label{sec:intro}

\subsubsection*{Background on extreme value theory}
Classic extreme value theory concerns the probability distribution of unlikely
(large) events,
see~\cite{Gal78,LLR83,Res87,Castillo,Embrechts,Coles2001,Beirlantetal2004}. Given
a stochastic process $X_1,X_2, \dots$ governed by independent identically
distributed random variables, let $M_n$ be the random variable defined as the
maximum over the first $n$ occurrences:
\begin{equation*}
  \label{eq:maxn}
  M_n = \max(X_1,\dots,X_n). 
\end{equation*}
This variable has a degenerate limit as $n\to\infty$, and therefore it is
necessary to consider a rescaling. Suppose that there exist sequences $a_n\geq
0$ and $b_n\in \R$ such that the rescaled variable $a_n(M_n - b_n)$ converges
to a non-degenerate distribution. That is
\begin{equation}
  \label{eq:EVlimit} 
  \lim_{n\rightarrow \infty} P\left(a_n\left(M_n-b_n\right) \leq x \right) = G(x)
\end{equation}
Then extreme values theory asserts that the limit $G(x)$ can only be one of
three different types: the Gumbel, Weibull and Fr\'echet parametric families
of probability distributions.  These three families can be combined into a
single three-parameter family having distribution function
\begin{equation}
  \label{eq:GEV}
  G(x) = \exp\left(- \left[1 + \xi \left(-\frac{x-\mu}{\sigma}\right)
    \right]^{-1/\xi} \right), 
\end{equation}
defined on the set $\left\{x |\thinspace 1 + \xi
  \left(-\frac{x-\mu}{\sigma}\right) >0\right\}$, where the parameters satisfy
$-\infty < \mu <\infty$, $\sigma>0$ and $-\infty<\xi <
\infty$. Eq.~\eqref{GEV} is called the generalised extreme value (GEV) family
of distributions. The subset of the GEV family with $\xi=0$ is interpreted as
the limit of~\eqref{GEV} as $\xi\rightarrow 0$, leading to the Gumbel family
(with parameters $\mu$ and $\sigma$).

In the applications the GEV family is particularly useful to predict the
probability of occurrence of future large values of a quantity, given a sample
of past experimental measurements of that quantity.  The so-called \emph{block
  maximum} method is frequently used in this setting. Here one extracts a
sub-sample of maxima over data blocks: in environmental and climate contexts
one often uses blocks of length one year, hence the name of \emph{annual
  maximum} method. One then estimates the parameters $(\mu,\sigma,\xi)$,
assuming that that the block maxima form a random sample drawn from a GEV
distribution with unknown parameters.  Maximum likelihood is a common
estimation method~\cite{Coles2001}: in this case, standard asymptotic theory
also provides confidence intervals (uncertainties) for the point
estimates. The estimated GEV parameters and associated uncertainties can then
be used to derive other quantities of interest, such as return periods for
given return levels of the variable of interest, see the above references
and~\cite{Felicietal2007a,Felicietal2007b,Vitoloetal2008,VHF:09} for examples.

\subsubsection*{Extremes in deterministic systems}
Recent work has extended the domain of extreme value theory to the setting of
chaotic deterministic dynamical systems
~\cite{Collet,Hai03,Fre09,FF08,Chazottes:10,FFT10,FFT10b,FFT10c,Gupta,HNT}.
We briefly outline the difference of our problem setting as opposed to the
above results.  Suppose that we have a dynamical system $(\XX,\nu, f)$, where
$\XX$ is a $d$-dimensional Riemannian manifold, $f:\XX\rightarrow \XX$ a
measurable map and $\nu$ an $f$-invariant probability measure. Assume that
there is a compact invariant set $\Lambda\subset\XX$ which supports the
measure $\nu$. Specifically, our main interest is the situation where
$\Lambda$ is a strange attractor and $\nu$ is a Sinai-Ruelle-Bowen (SRB)
measure~\cite{Young2002}.  Given an observable $\phi:\XX \rightarrow \R \cup
\{+\infty\}$ we study extreme value limit laws for the stationary stochastic
process $X_1, X_2, \dots$ defined by
\begin{equation}
\label{equation-stocastic-process-ds} 
X_i =\phi \circ f^{i-1}, \quad i \geq 1. 
\end{equation}
The theoretical work cited above focused on the special case where $\phi$
has the form
\begin{equation}
  \label{eq:obs-dist} 
  \phi(p) = g(\dist(p,p_M)),\quad p\in\XX,
\end{equation}
where $g:[0,+\infty)\to\R$ is a measurable function of the distance
$\dist(\cdot,\cdot)$ in $\XX$ and $p_M$ is a density point of $\nu$. However,
typical observation functions used in applications \emph{are not of this
  form}. Consider for example the quasi-geostrophic model
of~\cite{Felicietal2007a,Felicietal2007b}: this model was conjectured
in~\cite{Lucarinietal2007} to possess a compact (bounded) strange attractor in
its (unbounded) phase space. The observables used
in~\cite{Felicietal2007a,Felicietal2007b,VHF:09,VitoloandSperanza2011} are the
system's total energy, the wind speed and vorticity at a gridpoint in the
lower level. These observables can be written as
\begin{equation}
  \label{eq:energy-vorticity}
  \phi_E(p)=p^TEp, \quad  \phi_W(p)=||Wp||, \quad \phi_V(p)=Vp,
  \quad\text{respectively,}
\end{equation}
where $p$ is a point in the phase space $\XX=\R^d$, $||\cdot||$ denotes the
Euclidean norm and $E\in\R^{d\times d}$, $W\in\R^{2\times d}$,
$V\in\R^{1\times d}$ are matrices. None of the observables
in~\eqref{energy-vorticity} has the form~\eqref{obs-dist}. In fact this
situation is to be expected in many, if not in most, observables found in
applications, including the atmospheric and oceanic models
of~\cite{Sterk:10,BDSSV2010}.  Although observables such
as~\eqref{energy-vorticity} are usually unbounded in the system's phase space,
the system's attractor $\Lambda$ is usually bounded due to the presence of
dissipative processes in the models. Therefore, time series of such
observables should be expected to have an upper bound and, hence, large values
typically obey Weibull limit distributions, see~\cite{Felicietal2007a,VHF:09}
for a more detailed discussion.

\subsubsection*{Sketch of the results}
In this paper we will consider observables $\phi$ which are not a function of
the distance from a point $p_M$ as in~\eqref{obs-dist}. Such observables
include cases, like those just mentioned, where $\phi$ has no upper bound in
the phase space, although time series of $\phi$ on the system's attractor are
bounded. Hence we will restrict to the Weibull case in our numerical
examples. For comparison with the already available theory, we also consider
cases when $\phi$ is maximised at a point $p_M$, where however $p_M$ may or
may not be a density point for the invariant measure $\nu.$

In such situations, to determine the form of the limiting GEV distribution
$G(x)$ becomes a much more delicate problem: $G(x)$ is no longer determined by
the functional form of $\phi(p)$ and by the dimension of the SRB measure
$\nu$, but also critically depends on the geometry of the attractor
$\Lambda\subset\XX$ and of the level sets of $\phi$. A careful analysis is
required even if we assume that $\phi$ takes the form of~\eqref{obs-dist}, but
allowing instead that $p_M\not\in\Lambda.$ Without attempting an exhaustive
analysis of all possible cases, we focus on selected examples to illustrate
the key ideas of our approach, in view of applications to a given system and
observable.

To whet the Reader's appetite, we here anticipate one of the results of this
paper.  For $f$ we consider the solenoid map~\cite{BroerTakens2011} embedded
in $\R^3=\{(x,y,z)\}$: this system possesses a strange attractor $\Lambda$
which is locally the product of a Cantor set with an
interval~\cite{Hasselblatt:04}, where the interval represent the leaves of a
one-dimensional unstable manifold $W^u$. For the observable we take
$\phi(x,y,z)=ax+by+c+d$, which is clearly not of the form~\eqref{obs-dist}:
rather, $\phi$ resembles the vorticity observable defined
in~\eqref{energy-vorticity}. For this pair of system and observable we obtain
the formula
\begin{equation}\label{eq:general-tail-anticipated}
-\frac{1}{\xi}=\frac{d_u}{2}+d_s
\end{equation}
for the tail index $\xi$ of the limiting GEV distribution. Here
$d_u=\dim(W^u)=1$ and $d_s$ is the dimension in the stable direction, which in
this case is given by $\dim_H(\Lambda)-1$, where $\dim_H$ denotes the
Hausdorff dimension.  Loosely speaking, the factor $1/2$
in~\eqref{general-tail-anticipated} is obtained under a generic condition of
quadratic tangency between a local unstable manifold within the attractor
$\Lambda$ and the level sets of the observable $\phi$.  As we shall argue, we
believe formula~\eqref{general-tail-anticipated} to be valid, or at least
sufficiently informative, for a large class of pairs $(f,\phi)$ of
systems-observables. However, we also discuss examples where
formula~\eqref{general-tail-anticipated} has to be modified to take into
account the local geometry or the local scaling of the invariant measure of
the attractor, or the local behaviour of the level sets of the observable.  We
will restrict our discussion to the tail index $\xi$, which is usually the
most delicate parameter to estimate in the analysis of extreme values: see
e.g.~\cite{Coles2001,Beirlantetal2004} and~\cite{Faranda2011} for the link
between the normalising constants $a_n,b_n$ and the other two GEV parameters
$\mu,\sigma$. We note, however, that our numerical procedure also provides
estimates for the latter two parameters, see \appref{numerical}.


\subsubsection*{Outline of the paper}
This paper is organised as follows. The general framework of extremes in
dynamical systems is presented in \secref{mainresult}. Our main theoretical
results are formulated in Sections~\ref{sec:thom} to~\ref{sec:lorenz} for
specific dynamical systems.  As a particularly simple example, we consider
Thom's map in \secref{thom}, to illustrate the main ideas in our
approach. Then in \secref{solenoid} we discuss the solenoid attractor, which
displays several features which are found in many concrete physical systems.
\secref{henon} presents results for two non-uniformly hyperbolic systems, the
H\'enon and Lozi maps. In \secref{lorenz} we examine two prototypical flows
with chaotic dynamics due to Lorenz~\cite{Lorenz:63,Lorenz:84}.  In all of the
sections, analytical calculations precede numerical simulations, where the
latter aim to show the typical behaviour and estimation problems which can be
expected to occur. We return in \secref{conclusions} to the general relevance
of our approach.

\section{Extremes in dynamical systems: the general problem setting}
\label{sec:mainresult}
We consider a measure preserving system $(\XX,\nu, f)$ with a compact
attracting set $\Lambda\subset \XX$. Given an observable
$\phi:\XX\to\mathbb{R}$ and a threshold $u\in\R$, we define the level regions
$L^+(u)$ (resp. level sets $L(u)$) as follows:
\begin{equation}
  \label{eq:levelsets}
  L^{+}(u)=\{p\in\XX:\phi(p)\geq u\},\qquad
  L(u)=\{p\in\XX:\phi(p)=u\}.
\end{equation}
For the reasons discussed in the Introduction, we consider observables which
achieve a finite maximum within $\Lambda$, although the observable themselves
could be unbounded in $\XX$. We define
\begin{equation}
  \label{eq:suprem}
  \tilde{u}=\sup_{p\in\Lambda}\phi(p).
\end{equation}
Since $\Lambda$ is compact there exists (at least) one point
$\tilde{p}\in\Lambda$ for which $\phi(\tilde{p})=\tilde{u}$. We will assume
that such an \emph{extremal point} $\tilde{p}$ is unique. Given our focus on
the Weibull case (again, see the Introduction) we consider sequences
$u_n=u/a_n+b_n$ for which the limit
\begin{equation}
  \label{eq:tau}
  \lim_{n\to\infty}n\nu(L^{+}(u_n)):=\tau(u)
\end{equation}
exists.  From the theory of~\cite{Leadbetter1983}, we can choose
$b_n=\tilde{u}$, and we take $a_n\to\infty$. The precise form of $a_n$ depends
on the regularity of $\phi$, and the regularity of the density of $\nu$ in the
vicinity of the extremal point $\tilde{p}$. In general $a_n$ will be a power
law in $n$, and $\tau(u)$ will be regularly varying in $u$.  If we now
consider the process $M_n = \max(X_1,\dots,X_n)$ with $X_n= \phi \circ f^n$,
then we investigate to what extent the following statement is true:
\begin{equation}\label{extreme.dist}
    n\nu\{p:\phi(p)\geq u_n\}\to\tau(u)\quad\Leftrightarrow\quad
    \nu\{M_{n}\leq u_n\}\to e^{-\tau(u)}.
\end{equation}
If $\tau(u)=u^{\alpha}$, then the process $M_n$ is described by a GEV
distribution with tail index $\xi=-1/\alpha$. The statement
(\ref{extreme.dist}) is shown to hold for a wide class of dynamical systems,
such as those governed by non-uniformly expanding maps, and systems with
(non)-uniformly hyperbolic attractors,
\cite{Chazottes:10,Guptaetal2008}. However the current theory for analysing
extremes in dynamical systems assumes that the level regions $L^{+}(u)$
introduced in~\eqref{levelsets} are described by balls, and moreover that
these balls are centred on points in $\Lambda$ that are generic for
$\nu$. These assumptions allow for the tail index to be expressed in terms of
local dimension formulae for measures.

In this article we do not assume that the level sets are balls: for example we
consider observables of the form
$\phi(p)=\phi(x_1,\ldots,x_d)=\sum_{i}|x_i|^{a_i}$, where the level sets have
cusps or are non-conformal. We also consider observables
$\phi(p)=\sum_{i}c_ix_i$, for which the level sets are hyperplanes. For
observables of these types (also compare with~\eqref{energy-vorticity}) the
standard machinery does not immediately apply. The first problem is to
determine the sequence $u_n$ and the limit $\tau(u)$ defined
in~\eqref{tau}. Even if the measure $\nu$ is sufficiently regular then the
sequence $u_n$ will depend on the geometry of the attractor close to where
$\phi(p)$ achieves its maximum value on $\Lambda$, in addition to depending on
the form of $\phi$. In \secref{thom} we illustrate the various geometrical
scenarios that can arise using an hyperbolic toral automorphism as a simple
example. When $\nu$ is a more general SRB measure, then even for uniformly
hyperbolic systems (such as the solenoid map) it becomes a non-trivial problem
to determine $u_n$ and $\tau(u)$. We discuss this scenario in
\secref{solenoid}.  The second problem is to verify statement
(\ref{extreme.dist}) for the class of observables under consideration. This
relies upon checking two conditions $D_2(u_n), D'(u_n)$, see
\cite{FF08,Guptaetal2008}. We summarise these conditions as follows.  For
integers $t,l$ let $M_{t,l}=\max\{X_{t+1},X_{t+2},\ldots, X_{t+l}\}$, with
$M_{0,l}:=M_l$. Then:
\begin{description}
\item[$(D_2(u_n))$] We say condition $D_2(u_n)$ holds for the process
  $X_0,X_1,\ldots, $ if for any integers $l$,$t$ and $n$ we have
  \begin{equation*}
    |\nu ( X_1 >u_n, M_{t,l} \le u_n )-\nu (X_1 >u_n)\nu ( M_l \le u_n)| \le
    \gamma(n,t),
  \end{equation*}
  where $\gamma (n,t)$ is non-increasing in $t$ for each $n$ and
  $n\gamma(n,t_n)\to 0$ as $n\to \infty$ for some sequence $t_n=o(n)$,
  $t_n\rightarrow \infty$.
\item[$(D'(u_n))$] We say condition $D^{'}(u_n)$ holds for the process
  $X_1,X_2,\ldots, $ if
\begin{equation}\label{cond:dprime}
  \lim_{k\to \infty}\limsup_{n\to\infty} n\sum_{l=2}^{[n/k]}\nu(X_1>u_n,X_j>u_n)=0.
\end{equation}
\end{description}
If the level sets have complicated geometry, or if the measure $\nu$ is
supported on a fractal set then these conditions must be carefully
checked. For uniformly hyperbolic systems, and for observations that are
functions of balls these conditions are checked in \cite{Guptaetal2008}.  In
this article we consider the computation of the GEV tail index $\xi$ for more
general observations and contrast to results known for observations that are
functions of balls.  We focus on particular examples to highlight how the
geometrical features of the level sets and the attractor feed into the
computation of the tail index $\xi$, without attempting an exhaustive analysis
of all possible cases. We also discuss the computation of the tail index for
non-uniformly hyperbolic systems such as the Lozi map and H\'enon map, and also
for Lorenz flows, again for general observations.

\section{A prototypical example: Thom's map}
\label{sec:thom}

Let $\T^2=\R^2\mod 1$ be the $2$ dimensional torus. Thom's map $f: \T^2
\rightarrow \T^2$, also known as \emph{Arnold's cat map}, is the hyperbolic
toral automorphism defined by
\begin{equation}
   \label{eq:thom} 
f(x,y) = (2x+y, x+y) \mod 1.
\end{equation}
This system is Anosov and it has Lebesgue measure $\nu$ on the torus $\T^2$ as
the (unique) invariant measure.  With this example we want to study the role
of the observable in determining extreme value laws. For this purpose we will
consider $f$ as a map of $\mR^2$ having the square $[0,1)^2$ as the invariant
set. In other words, $\XX=\R^2$ and $\Lambda=[0,1)^2$, hence $\Lambda$ is not
an \emph{attractor}, strictly speaking.  The advantage is that this allows us
to take functions of $\mathbb{R}^2$ as observables, rather than functions of
$\T^2$. In this way, we can construct observables which are maximised at
points in the interior or in the complement of $\Lambda$ and whose level sets
have different shapes.

The main point of this section is that the value of the tail index is
determined by the interaction between the shape of level
sets~\eqref{levelsets} of the observable and the shape of the support of the
invariant measure (colloquially, the geometry of the attractor).  To
illustrate our ideas, and without attempting to cover all possible cases, we
consider the following two observables $\phi_{\alpha},\phi_{a,b}:\mR^2\to\mR$
\begin{align}
  \label{eq:obsalpha}
  \phi_{\alpha}(x,y)&=1-\dist(p,p_M)^\alpha,\quad\text{with $p=(x,y)\in\mR^2$.}\\
  \label{eq:obsab}
  \phi_{a,b}(x,y)&=1-|x-x_M|^a-|x-y_M|^b,
\end{align}
where, given our focus on the Weibull case, we require $a,b,\alpha>0$.  Both
observables are maximised at a point $p_M=(x_M,y_M)\in\mR^2$.  When $p_M$ is
in the interior of $\Lambda$, observable~\eqref{obsalpha} has the form so far
analysed in the mathematical literature about extremes in dynamical systems,
but we will also consider the case $p_M\not\in\Lambda.$
Observable~\eqref{obsab} has been chosen to illustrate the effect of the shape
of the level sets: the level regions of~\eqref{obsab} are not balls unless
$a=b=2$, in which case~\eqref{obsab} can be written as~\eqref{obsalpha} for
$\alpha=2$.  The next subsection contains our analytical results, numerical
simulations are postponed to \secref{thom-numerics}.

\subsection{Analytical calculations}
\label{sec:thom-theory}

The level regions $L^+(u)$ as defined in~\eqref{levelsets} are always balls
for observable~\eqref{obsalpha}. However three (main) different situations
occur, depending on the location of the point $p_M$ relative to the support of
the invariant measure, see the sketch in \figref{thom-illustrationalpha}.
\begin{thm}
  \label{thm:thom-theoryalpha}
  Let $\xi$ be the tail index of the GEV limit distribution associated to the
  process $M_n = \max(X_1,\dots,X_n)$ with $X_n= \phi \circ f^n$, where $f$ is
  the map~\eqref{thom} and $\phi_\alpha:\R^2\to\mR$ is the observable
  in~\eqref{obsalpha}.  Then for $\nu$-a.e. $p_M=(x_M,y_M)\in\mathbb{R}^2$ we
  have:
  \begin{align}
    \label{eq:tail-thom-alphain}
    -\frac1\xi&=\frac 2\alpha &&\text{for $p_M\in\Lambda$;}\\
    \label{eq:tail-thom-alphaout1}
    -\frac1\xi&=\frac 32 &&\text{for $p_M\not\in\Lambda$,
      with either $y_M\in(0,1)$ or $x_M\in(0,1)$;}\\
    \label{eq:tail-thom-alphaout2}
    -\frac1\xi&=2 &&\text{for $p_M\not\in\Lambda$,
      with both $x_M,y_M\not\in[0,1]$;}
  \end{align}
\end{thm}
For observable~\eqref{obsab} the shape of the level sets $L(u)$ depends on $a$
and $b$. For example, $L(u)$ has a convex elliptic-like shape when both
$a,b>1$ (see the sketch in \figref{thom-illustrationab}~(A)), or an
asteroid-like shape when both $a,b<1$, (see \figref{thom-illustrationab}~(B)).
Clearly various possibilities arise, depending on the geometry of the level
sets, on whether the point $p_M$ is in the interior of $\Lambda$ and on the
local geometry of $\Lambda$ near the extremal point
$\tilde{p}=(\tilde{x},\tilde{y})$ with minimum distance from $p_M$.
\begin{thm}
  \label{thm:thom-theoryab}
  Let $\xi$ be the tail index of the GEV limit distribution associated to the
  process $M_n = \max(X_1,\dots,X_n)$ with $X_n= \phi \circ f^n$, where $f$ is
  the map~\eqref{thom} and $\phi_{a,b}:\R^2\to\mR$ is the observable
  in~\eqref{obsab}.  Then for $\nu$-a.e. $p_M=(x_M,y_M)\in\mathbb{R}^2$ we have:
  \begin{align}
    \label{eq:tail-thom-abin}
    -\frac1\xi&=\frac 1a+\frac 1b &&\text{for $p_M\in\Lambda$.}
  \end{align}
\end{thm}
To prove Theorems~\ref{thm:thom-theoryalpha}-\ref{thm:thom-theoryab} the main
step is to determine the explicit sequence $u_n$ and functional form of
$\tau(u)$ as defined in~\eqref{tau}. We will not give the verification of
$D_2(u_n)$, $D'(u_n)$ as this follows step by step from \cite{Guptaetal2008}
for this class of observables.  The main proof of \thmref{thom-theoryalpha} is
contained in Lemmas~\ref{lem:obsalphain},\ref{lem:obsalphaout1}
and~\ref{lem:obsalphaout2}.  The proof of \thmref{thom-theoryab} is given in
\lemref{ab}.  We do not further discuss case~(C) of
\figref{thom-illustrationab}, or the other configurations not covered by
\figref{thom-illustrationab}.
\begin{lem}
  \label{lem:obsalphain}
  Suppose $p_M\in\interior(\Lambda)=(0,1)^2$ and $\phi$ takes the form of
  (\ref{eq:obsalpha}), then $\xi=-\alpha/2$.
\end{lem}
\begin{proof}
  If $p_M\in\interior(\Lambda)$ (see \figref{thom-illustrationalpha}~(A)),
  then we see that
  \begin{equation}
    \begin{split}
      n\nu\{\phi(x,y)\geq u_n\} &=n\nu\{x\leq (1-u_n)^{1/\alpha}\}\\
      &=n(1-u_n)^{2/\alpha}.
    \end{split}
  \end{equation}
  Thus the correct scaling laws are $a_n=n^{\alpha/2}, b_n=1$ and
  $\tau(u)=(-u)^{2/\alpha}.$
\end{proof}
\begin{lem}
  \label{lem:obsalphaout1}
  Suppose $p_M\not\in\overline{\Lambda}=[0,1]^2$ and $\phi$ takes the form of
  (\ref{eq:obsalpha}).  If $y_M\in(0,1)$ or $x_M\in(0,1)$ then $\xi=-2/3$.
\end{lem}
\begin{proof}
  If $p_M\not\in \overline{\Lambda}$ then there will exist a unique extremal
  point $\tilde{p}=(\tilde{x},\tilde{y})\in \overline{\Lambda}$ where
  $\phi(p)$ achieves its supremum with value $\tilde{u}$ as
  in~\eqref{suprem}. Since $y_M\in(0,1)$ or $x_M\in(0,1)$ then this point
  $\tilde{p}$ will not be a vertex of $\partial \Lambda$, see
  \figref{thom-illustrationalpha}~(B). The scaling $u_n$ will be chosen to so
  that
  \begin{equation}
    n\nu\{\phi(x,y)\geq u_n\}=n\nu\{p=(x,y)\in \Lambda:d(p,p_M)\leq (1-u_n)^{1/\alpha}\}\to\tau(u).
  \end{equation}
  The middle term is no longer $n(1-u_n)^{1/\alpha}$ since the level region
  that intersects $\Lambda$ is not a ball.  We first of all set
  $u_n=u/a_n+\tilde{u}$ so that
  \begin{equation}
    \nu\{\phi(x,y)\geq u_n\}=\nu\left\{p=(x,y)\in
    \Lambda:(1-\tilde{u})^{1/\alpha}\leq d(p,p_M)\leq 
    \left(1-\tilde{u}-\frac{u}{a_n}\right)^{1/\alpha}\right\}.
  \end{equation}
  To choose $a_n$ we first note that the level set $L(\tilde{u}^{1/\alpha})$
  as defined in~\eqref{levelsets} is a circle that is tangent to $\partial
  \Lambda$ (since the extremal point $\tilde{p}$ is not a vertex).  However
  the level set $L((\tilde{u}-\frac{u}{a_n})^{1/\alpha}))$ crosses $\partial
  \Lambda$ transversally (and is concentric to $L(\tilde{u}^{1/\alpha})$).
  Thus a basic geometrical calculation gives:
  \begin{equation}\label{basic-geom}
    \begin{split}
      \nu\{\phi(x,y)\geq u_n\} &=\left\{(1-\tilde{u}-u/a_n)^{1/\alpha}-
        (1-\tilde{u})^{1/\alpha}\right\}^{3/2}\\
      &=(1-\tilde{u})^{3/2\alpha}\left\{\left(1-\frac{u}{a_n(1-\tilde{u})}\right)^{1/\alpha}-1\right\}^{3/2}\\
      &=(1-\tilde{u})^{3/2\alpha}\left\{-\frac{u}{\alpha a_n(1-\tilde{u})}+O\left(\frac{1}{a^2_n}\right)
      \right\}^{3/2}.
    \end{split}
  \end{equation}
  Setting $a_n=n^{2/3}$ implies that $\tau(u)=(-u)^{3/2}$.
\end{proof}
\begin{lem}
  \label{lem:obsalphaout2}
  Suppose $p_M\not\in\Lambda$ and $\phi$ takes the form of
  (\ref{eq:obsalpha}).  If both $x_M,y_M\not\in[0,1]$ then $\xi=-1/2$.
\end{lem}
\begin{proof}
  Without loss of generality we consider $p_M=(x_M,y_M)$ in the upper right
  hand quadrant as in \figref{thom-illustrationalpha}~(C). Also, we assume
  that $x_M=1+\lambda\cos\theta$, $y_M=1+\lambda\sin\theta$ for $\lambda>0$
  and $\theta\in(0,\pi/2)$. For such values of $(x_M,y_M)$, the corner point
  $(1,1)\in\partial\Lambda$ will always maximise $\phi$. The proof is
  identical to \lemref{obsalphaout1} except that the level sets are not
  tangent to $\partial \Lambda$ at $(1,1)$, as illustrated in
  \figref{thom-illustrationalpha}~(C).  A simple geometric argument shows that
  for a disk $D(\epsilon)$ of radius $\lambda+\epsilon$, centred at
  $(x_M,y_M)$ we have
  $$\mathrm{Leb}\{p=(x,y)\in M\cap D(\epsilon)\}=\mathcal{O}(\epsilon^2),$$
  where $\mathrm{Leb}$ denotes the Lebesgue measure.
  Setting $u_n=u/a_n+\tilde{u}$ and comparing to (\ref{basic-geom}), we
  obtain:
  \begin{equation}\label{basic-geom2}
    \begin{split}
      \nu\{\phi(x,y)\geq u_n\} &=\left\{(1-\tilde{u}-u/a_n)^{1/\alpha}-
        (1-\tilde{u})^{1/\alpha}\right\}^{2}\\
      &=(1-\tilde{u})^{2/\alpha}\left\{\left(1-\frac{u}{a_n(1-\tilde{u})}\right)^{1/\alpha}-1\right\}^2\\
      &=(1-\tilde{u})^{2/\alpha}\left\{-\frac{u}{\alpha a_n(1-\tilde{u})}+O\left(\frac{1}{a^2_n}\right)
      \right\}^2.
    \end{split}
  \end{equation}
  Setting $a_n=n^{1/2}$ implies that $\tau(u)=(-u)^{2}$.
\end{proof}
This concludes the proof of \thmref{thom-theoryalpha}.  For the proof of
\thmref{thom-theoryab} we have the following lemma.
\begin{lem}
  \label{lem:ab}
  Suppose that $p_M\in\interior(\Lambda)=(0,1)^2$ and $\phi$ takes the form of
  (\ref{eq:obsab}). Then for $u\lesssim 1$ we have that $\mathrm{Leb}(L(u)) =
  C(1-u)^{\frac{1}{a} + \frac{1}{b}}$ for some $C_0\leq C\leq 4$ where
  $C_0>0$.
\end{lem}
\begin{proof}
  Let $u=1-\eps$.  For $\eps$ sufficiently small the level region can
  be written as
  \begin{equation}
    \label{eq:ellipsoid}
    L^+(u)=\{(x,y) \in \interior(\Lambda): \thinspace |x|^a + |y|^b \leq \eps\}.
  \end{equation}
  The area of this set is bounded from above by the area of a rectangle of
  sides $2\eps^{1/a}$ and $2\eps^{1/b}$. Also, for any $q\in(0,1)$, the area
  of the set is bounded from below by that of a rectangle of sides
  $2q^{1/a}\eps^{1/a}$ and $2(1-q)^{1/b}\eps^{1/b}$, so we can choose
  $C_0=4q^{1/a}(1-q)^{1/b}$.
\end{proof} 
Hence if $(x_M,y_M)\in \interior(\Lambda)$, we see that
\begin{equation}
  n\nu\{p=(x,y):\phi(p)\geq u_n\}\to (-u)^{\frac{1}{a}+\frac{1}{b}}
  \quad\textrm{with}\quad a_n=n^{\frac{ab}{a+b}}, b_n=1.
\end{equation}

\subsection{Numerical results}
\label{sec:thom-numerics}

As formula~\eqref{tail-thom-abin} shows, one of the main ingredients in
determining the tail index is the shape of the level sets of the
observable. We here fix $a=2$ and consider two values of $b$, namely $b=1$ and
$b=3.5$. In both cases, the value of $\xi$ expected according
to~\eqref{tail-thom-abin} is less than $-0.5$. Since the maximum likelihood
estimator is not regular for $\xi<-0.5$~\cite{Coles2001}, we resort to the
method of L-moments for the numerical estimation of the GEV parameters.  See
\appref{numerical} for details on our procedure for parameter estimation and
associated uncertainties.

In \figref{thom-blocklens-p} we examine the sensitivity of the numerical
estimates of $\xi$ with respect to the block length used to compute the
maxima. Essentially no significant variations are found for block lengths
larger than 1000. Hence, we fix $N_{blocklen}=10000$ and conduct a study of
the dependence of the tail index on the parameter $b$ of the
observable~\eqref{tail-thom-abin}. \figref{xis-thom-p} shows a good agreement
with the theoretical predictions of~\eqref{tail-thom-abin} for a range of
values of $b$. In this example the level regions of the observable of the form
$L^+(u)$ as in~\eqref{levelsets} are fully contained in the interior of
$\Lambda=[0,1)^2$, at least for sufficiently large values of the threshold
$u$. The only peculiarity is the non-circular shape of $L^+(u)$, see
\figref{thom-illustrationab} and compare with \lemref{ab}.

We now consider a case where the level regions $L^+(u)$ are not fully
contained in $\Lambda$. We take observable~\eqref{obsalpha} and vary the
location of the point $p_M=(x_M,y_M)$.  Starting from values $x_M,y_M$ such
that $p_M$ is in the interior of $\Lambda$ we increase $x_M$ across $1$,
bringing $p_M$ in the region where $y_M\in(0,1)$ but $x_M\not\in(0,1)$.  This
transition is illustrated in panels (A) and (B) of
\figref{thom-illustrationalpha} and the two situations correspond
to~\eqref{tail-thom-alphain} and~\eqref{tail-thom-alphaout1} respectively.

\figref{thom-blocklens-b} shows the sensitivity of the numerical estimates of
$\xi$ with respect to the block length used to compute the maxima for four
values of $x_M$. Convergence to the theoretical
value~\eqref{tail-thom-alphain} is achieved already with block lengths of a
few hundred iterates when the point $p_M$ is in the interior of $\Lambda$
(panel~(A), $x_M=0.9$) and on the boundary of $\Lambda$ (panel~(B),
$x_M=1.0$). When $p_M$ is in the complement of $\overline{\Lambda}=[0,1]^2$
but close to its boundary, then very large block lengths ($N_{blocklen}>10^5$)
are required to achieve convergence to the theoretical value
of~\eqref{tail-thom-alphaout1} (panel~(C), $x_M=1.01$). When $p_M$ is further
away from $[0,1]^2$ shorter block lengths of about $10^4$ iterates already
guarantee convergence to the theoretical value of~\eqref{tail-thom-alphaout1}
(panel~(D), $x_M=1.1$).

\figref{xis-thom-b}~(A) shows the discontinuity of $\xi$ in the transition
between the situations of panels (A) and (B) in
\figref{thom-illustrationalpha}.  The figure shows the estimated value of
$\xi$ as a function of $x_M$ where $y_M$ is kept constant and a fixed block
length is used. As the point $p_M$ exits $\Lambda$, the value of $\xi$ has a
jump from the value of~\eqref{tail-thom-alphain} to that
of~\eqref{tail-thom-alphaout1}. However, the numerical estimation does not
resolve this jump unless large block lengths ($N_{blocklen}>10^5$ iterates)
are used.

Lastly, \figref{xis-thom-b} shows the discontinuity of $\xi$
from~\eqref{tail-thom-alphaout1} to~\eqref{tail-thom-alphaout2}, at the
transition between the situations of panels (B) and (C) in
\figref{thom-illustrationalpha}. This transition is not resolved accurately even
with block lengths of $10^5$. From the numerical point of view, very large
block lengths are required near the transition to detect the change of scaling
between~(\ref{basic-geom}) and~(\ref{basic-geom2}).

The example discussed in this section is admittedly somewhat artificial.  It
has been chosen to clearly illustrate the main ideas and the problems which
are found in the numerical estimation, without the additional complications
due to higher dimensionality of phase space and fractal nature of the
attractors.  In the next section, we consider a situation which is closer to
what one can expect in concrete physical systems.

\section{Uniformly hyperbolic attractors: the solenoid map} 
\label{sec:solenoid} 

Consider the solid torus as the product of $\T=\R/\Z$ times the unit disc in
the complex plane $\D_R=\{ z\in \C | \thinspace |z|<1 \}$.  Then the solenoid
map is defined as follows:
\begin{equation}
\label{eq:solenoid-nonembedded}
\begin{array}{rccc}
  f_{\lambda}: & \T \times \D & \rightarrow & \T \times \D  \\
  \displaystyle \rule{0ex}{4ex} 
  & \left( \psi, w \right) &
  \mapsto & \displaystyle   \left( 2 \psi ,  \lambda w + K e^{i 2\pi\psi}  \right). 
\end{array}
\end{equation}
In order to have the map well defined we need $K + \lambda R < R$ and
$\lambda R < K$. For our purposes it is convenient to have the torus embedded
in $\R^3$.  Consider Cartesian coordinates $(x,y,z)\in\R^3$ and define
corresponding cylindrical coordinates $r,\psi, z$ by $x=r \cos(\psi)$ and $y=r
\sin(\psi)$.  Then the torus of width $R$ can be identified with the set $D=\{
(r-1)^2 + z^2 \leq R ^2\}$ for $R< 1$. The torus $\T \times \D_{R}$ (with
coordinates $(\psi, u+ i v)$) can be identified with $D$ taking $r=1+u$ and
$z=v$. We thus obtain an embedded solenoid map
\begin{align}
  \label{eq:solenoid}
  g_\lambda:D\to D,\quad g_\lambda(\psi,r,z) = (2\psi, 1+ K \cos(\psi) +
  \lambda(r-1), K \sin(\psi) + \lambda z).
\end{align}
The solenoid attractor is defined as the attracting set of the map
$g_\lambda$:
\[
\Lambda = \bigcap_{j\geq 1} g^{j}_{\lambda}(D).
\]
For $\lambda < \frac{1}{2}$ we have
\begin{equation}
  \label{eq:dimsolenoid}
  \dim_H(\Lambda) = 1+ \frac{\log 2}{\log \lambda^{-1}},
\end{equation}
where $\dim_H$ denotes the Hausdorff dimension~\cite{Sim97}. We consider the
following observables $\phi_\alpha,\phi_{abcd}:\mR^3\to\mR$:
\begin{align}
  \label{eq:obsalpha3}
  \phi_{\alpha}(x,y,z)&=1-\dist(p,p_M)^\alpha,\quad\text{with $p=(x,y,z)\in\mR^3$,} \\ 
  \label{eq:obsabcd}
  \phi_{abcd}(x,y,z)&=ax+by+cz+d,
\end{align}
Observable~\eqref{obsalpha3} is maximised at a point $p_M\in\mR^3$,
whereas~\eqref{obsabcd} is unbounded in the phase space $\R^3$ (except for the
trivial choice $a=b=c=0$). Note that the vorticity observable $\phi_V$
of~\eqref{energy-vorticity} has the same general form as~\eqref{obsabcd}.
Our theoretical expectations are first discussed in \secref{solenoid-theory},
followed by numerical results in \secref{solenoid-numerics}.

\subsection{Analytical calculations}
\label{sec:solenoid-theory}
For observables which are functions of distance we have the following result.
It is not explicitly stated in the literature but the proof follows
straightforwardly from \cite{Guptaetal2008}.
\begin{thm}
  Let $\xi$ be the tail index of the GEV limit distribution associated to the
  process $M_n = \max(X_1,\dots,X_n)$ with $X_n= \phi \circ g_\lambda^n$,
  where $g_\lambda$ is the map~\eqref{solenoid} and $\phi_\alpha:\mR^3\to\mR$
  the observable of~\eqref{obsalpha3}, where $p_M\in\Lambda$.  Then we have:
  \begin{equation}
    \label{eq:tail-solenoid-alphain}
    -\frac1\xi=\frac{\dim_H(\Lambda)}{\alpha}.
  \end{equation}
\end{thm}
More interesting considerations arise for the observable \eqref{obsabcd}.  As
a simple case, consider first the degenerate solenoid with $\lambda=0$ and
take a planar observable $\phi:=ax+by+d$, thus reducing the problem to the
$(x,y)$-plane. In this case we have the trivial dimension formula
$\dim_H(\Lambda)=1$ since $\Lambda$ is a circle. However, for computing the
tail index we lose a factor of $1/2$ due to the geometry of the level set.
Indeed, level sets are straight lines within the $(x,y)$-plane, and at the
extremal point $\tilde{p}=(\tilde{x},\tilde{y})$ the critical level set
$L(\tilde{u})$ is tangent to $\Lambda$. Since the tangency is quadratic, we
find that
\begin{equation}
  \label{eq:uplussolenoid}
  \nu(L^{+}(\tilde{u}-\epsilon))=m_{\gamma^{u}}\{\gamma^{u}(\tilde{p})\cap
  L^{+}(\tilde{u}-\epsilon)\}=\mathcal{O}(\sqrt{\epsilon}).  
\end{equation}
Here $\gamma^{u}(\tilde{p})$ is the unstable manifold through $\tilde{p}$
(i.e. it is the unit circle), and $m_{\gamma^{u}}$ is the one-dimensional
conditional (Lebesgue) measure on $\gamma^{u}(\tilde{p})$. Hence
\begin{equation*}
  -\frac1\xi=\dim_H(\Lambda)-\frac12=\frac12.
\end{equation*}
The mechanism described above is similar to that illustrated for Thom's map in
\figref{thom-illustrationalpha}~(B), leading to
formula~\eqref{tail-thom-alphaout1}: indeed, there we have
$\dim_H(\Lambda)=2$, yielding the value $3/2$ for the tail $-1/\xi$.

For $\lambda>0$, the attractor has more complicated geometry and is locally
the product of a Cantor set with an interval~\cite{Hasselblatt:04}. Planar
cross sections that intersect $\Lambda$ transversely form a Cantor set of
dimension $\dim_H(\Lambda)-1=-\log2/\log\lambda$. To calculate
$\nu(L^{+}(\tilde{u}-\epsilon))$ we would like to repeat the calculation above
using equation \eqref{uplussolenoid}, but now the set of unstable leaves that
intersect $L^{+}(\tilde{u}-\epsilon)$ form a Cantor set (for each
$\epsilon>0$). The extremal point $\tilde{p}$ where $\phi(p)$ attains its
maximum on $\Lambda$ forms a \emph{tip} of $\Lambda$ relative to
$L(\tilde{u})$.  Such a tip corresponds to a point on $\tilde{p}\in\Lambda$
whose unstable segment $\gamma^u(\tilde{p})$ is tangent to $L(\tilde{u})$ at
$\tilde{p}$, and moreover normal to $\nabla\phi(\tilde{p})$ at $\tilde{p}$.
Given $\epsilon>0$, we (typically) expect to find a Cantor set of values
$t\in[0,\epsilon]$ for which the level sets $L(\tilde{u}-t)$ are tangent to
some unstable segment $\gamma^{u}\subset\Lambda$.  For other values of $t$
these level sets cross the attractor transversally. Given (fixed)
$\epsilon_0>0$ we can define the \emph{tip set}
$\Gamma\equiv\Gamma(\epsilon_0)\subset\Lambda$ as follows: let
$T_{p}\gamma^u(p)$ be the tangent space to $\gamma^u$ at $p$. Then we define
\begin{equation}\label{eq:tipset}
\Gamma=\{p\in L^{+}(\tilde{u}-\epsilon_0)\cap\Lambda:
T_{p}\gamma^{u}(p)\cdot\nabla\phi(p)=0\}.
\end{equation}
This tip set plays a role in proving the following result, which in turn 
provides us with information on the form of the tail index $\xi$.
\begin{prop}\label{prop:solenoid} 
  Suppose that $g_{\lambda}$ is the map \eqref{solenoid} and
  $\phi=\phi_{abcd}$. Define $\tau(\epsilon)=\nu\{\phi(p)\geq
  \tilde{u}-\epsilon\}$.  If $\mathrm{dim}_{H}(\Gamma)<1$, then modulo a zero
  measure set of values $(a,b,c,d)$, $\tau(\epsilon)$ is regularly varying
  with index $1/2+\mathrm{dim}_{H}(\Gamma)$ as $\epsilon\to 0$.
\end{prop}
We give a proof below. Based on this proposition we conjecture that
\begin{equation}
  \label{eq:tail-solenoid-abcout}
  -\frac{1}{\xi}=\mathrm{dim}_{H}(\Lambda)-\frac{1}{2}=\frac{1}{2}
  +\frac{\log 2}{\log\lambda^{-1}}.
\end{equation}
We outline the main technical challenges that would need to be overcome to
prove this conjecture.  Firstly, conditions $D(u_n)$ and $D'(u_n)$ should be
checked. We believe that this should follow from \cite{Guptaetal2008}, however
the proof would be non-standard due to the level set geometry. Secondly, we
would claim that $\mathrm{dim}_{H}(\Gamma)=\mathrm{dim}_{H}(\Lambda)-1$. The
proof of this would utilise the techniques used in \cite{Hasselblatt:04} to
analyse the regularity of the holonomy map between stable disks. In
particular, the Authors of \cite{Hasselblatt:04} show that the holonomy map is
Lipschitz on a set of full dimension. However, it does not automatically
follow that the holonomy map between $\Gamma$ and $\Lambda\cap D$ is Lipschitz
(for a disk $D$ transverse of $\Lambda$), but we believe that it is for
general planar observations.

\begin{proof}[Proof of Proposition \ref{prop:solenoid}]
For each $\epsilon<\epsilon_0$, consider the set 
$\Gamma(\epsilon)\subset\Gamma\cap L^{+}(\tilde{u}-\epsilon)$. Then
for each $p\in\Gamma(\epsilon)$, there exists $t<\epsilon$ such that 
$\gamma^{u}(p)$ is tangent to $L(\tilde{u}-t)$.
If the observable $\phi$ takes the form of \eqref{obsabcd},
then by the same calculation as~\eqref{uplussolenoid} we obtain
\begin{equation}\label{eq:leaf-measure}
m_{\gamma^{u}}\{\gamma^{u}(p)\cap
  L^{+}(\tilde{u}-\epsilon)\}=\mathcal{O}(\sqrt{\epsilon-t}).
\end{equation}
Thus to compute $\nu(L^{+}(\tilde{u}-\epsilon))$, we
integrate~\eqref{leaf-measure} over all relevant $t<\epsilon$ using the
measure $\nu_{\Gamma}$, which is the measure $\nu$ conditioned on $\Gamma$.
Provided $\mathrm{dim}_{H}(\Gamma)<1$, the projection of $\Gamma$ onto the
line in the direction of $\nabla\phi$ is also a Cantor set of the same
dimension for typical (full volume measure) $(a,b,c,d)$, see \cite{Falconer}.
Thus the set of values $t$ corresponding to when $L(\tilde{u}-t)$ is tangent
to $\Gamma$ form a Cantor set of dimension $\mathrm{dim}_{H}(\Gamma)$.  If
$\pi$ is the projection from $\Gamma$ onto a line in the direction of
$\nabla\phi$, then the projected measure $\pi_{*}\nu_{\Gamma}$ has local
dimension $\mathrm{dim}_{H}(\Gamma)$ for typical $(a,b,c,d)$.  We have
\begin{equation}
\nu(L^{+}(\tilde{u}-\epsilon))=\int_{0}^{\epsilon}\int_{L^{+}(\tilde{u}-\epsilon)}
dm_{\gamma^{u}}d\nu_{\Gamma}.
\end{equation}
To estimate this integral we bound it above via the inequality
$m_{\gamma^u}(\gamma^u\cap L^{+}(\tilde{u}-\epsilon)) \leq C\sqrt{\epsilon}$,
and bound it below using the fact that for $t>\epsilon/2$,
$m_{\gamma^u}(\gamma^u\cap L^{+}(\tilde{u}-\epsilon)) \geq
C\sqrt{\epsilon}$. Here $C>0$ is a uniform constant.  Putting this together we
obtain for typical $(a,b,c,d)$
\begin{equation}
\nu(L^{+}(\tilde{u}-\epsilon))=\int_{0}^{\epsilon}\int_{L^{+}(\tilde{u}-\epsilon)}
dm_{\gamma^{u}}d\nu_{\Gamma}\approx \sqrt\epsilon\cdot
\epsilon^{\mathrm{dim}_{H}(\Gamma)}
=\epsilon^{1/2+\mathrm{dim}_{H}(\Gamma)}.
\end{equation}
\end{proof}

\subsection{Numerical results}
\label{sec:solenoid-numerics}
We now examine the convergence of the numerically estimated values to the
theoretically expected ones. For the numerical simulations, we rewrite
observable~\eqref{obsabcd} in two forms
$\phi_{\theta,1},\phi_{\theta,2}:\mR^3\to\mR$ (form~\eqref{obsabcd} is
recovered for suitable values of $a,b,c,d$):
\begin{align}
  \label{eq:obsthetaxy}
  \phi_{\theta,1}(x,y,z)&=\cos(2\pi \theta) (x -x_0) +\sin(2 \pi \theta) (y-y_0) ,\\
  \label{eq:obsthetaxz}
  \phi_{\theta,2}(x,y,z)&=\cos(2\pi \theta) (x -x_0) +\sin(2 \pi \theta) (z-z_0) .
\end{align}
The level sets associated to~\eqref{obsthetaxy} and~\eqref{obsthetaxz} are
planes orthogonal to $(\cos(2\pi \theta), \sin(2 \pi \theta), 0)$ and
$(\cos(2\pi \theta),0, \sin(2 \pi \theta))$, respectively.
\figref{blocklens-soli-theta} shows the dependence of the estimates of $\xi$
with respect to the block length $N_{blocklen}$ for both~\eqref{obsthetaxy}
and~\eqref{obsthetaxz} at $\theta=0.5$. \figref{blocklens-soli-theta} suggests
that the block length $N_{blocklen}=10^4$ is sufficient to get an estimate
coherent with the theoretical value~\eqref{tail-solenoid-abcout} for
observable~\eqref{obsthetaxy}, whereas the same value is not sufficient for
observable~\eqref{obsthetaxz}.  This is illustrated in
\figref{xis-solenoid-theta1} for a range of values of $\theta$: reliable
estimation is obtained with block length $N_{blocklen}=10^4$ for
observable~\eqref{obsthetaxy}~(panel~(A)) but not for
observable~\eqref{obsthetaxz}~(panel~(B)), for which $N_{blocklen}=10^6$ seems
to suffice~(panel~(C)).

In summary, the minimum block length required for (approximate) convergence to
the theoretical value may vary strongly with the location within the attractor
of the extremal point $\tilde{p}$ in~\eqref{uplussolenoid}, that is with the
relative position of the attractor and the level sets. Also, the minimum block
length may depend on the dimensionality of the attractor. Numerical
experiments suggest that reliable estimation is more difficult when the
dimensionality of the attractor is smaller. \figref{xis-solenoid-lambda}
indeed shows better agreement with the prediction
of~\eqref{tail-solenoid-abcout} for the larger values of $\lambda$ which also
correspond to a larger dimension according to~\eqref{dimsolenoid}.

Lastly, we consider observable~\eqref{obsalpha3}. As we did in
\figref{thom-blocklens-b}, we illustrate the effect of the point $p_M$
``dropping out'' of the attractor $\Lambda$. To achieve this, we iterate the
solenoid map starting from an arbitrarily chosen initial condition.  After
discarding a transient of $10^5$ iterates, we regard the final point $p_M^0$
of the orbit as being generic with respect to the Sinai-Ruelle-Bowen measure
on $\Lambda$.  \figref{solenoid-blocklens-alpha}~(A) shows the sensitivity of
the estimates of $\xi$ with respect to block length for
observable~\eqref{obsalpha3} when $p_M$ is equal to $p_M^0$ as obtained
above. The estimates display strong oscillations around the theoretical value
and barely seem to settle for very large block lengths $N_{blocklen}>10^7$.
We return to this problem in \secref{henon}.

We then choose $p_M$ as a perturbation of point $p_M^0$ in the radial
direction in $\mR^3$: namely we set $p_M=p_M^t=(1+t)p_M^0$. By dissipativity
of the solenoid map, we expect $p_M^t\not\in\Lambda$ with probability 1 when
$t\neq0$.  We find out that when $t$ is sufficiently large
(\figref{solenoid-blocklens-alpha}~B)), the estimates of $\xi$ converge to the
theoretically expected value~\eqref{tail-solenoid-abcout} already for block
lengths of 1000.  However, when $t$ is small
(\figref{solenoid-blocklens-alpha}~C)) the estimates are closer to the value
attained within the attractor~\eqref{tail-solenoid-alphain} for small block
lengths, whereas convergence to the theoretically expected
value~\eqref{tail-solenoid-abcout} takes place for $N_{blocklen}$ larger than
about $10^5$.

\section{Non-uniformly hyperbolic examples: the H\'enon and Lozi maps}
\label{sec:henon}

We here consider the H\'enon map~\cite{benedicks-young,Chazottes:10}
\begin{equation}
  \label{eq:henon}
  h_{a,b}:\mR^2\to\mR^2,\quad h_{a,b}(x,y)=(1-ax^2+y,bx),
\end{equation}
for the classical parameter values $(a,b)=(1.4,0.3)$ and the Lozi
map~\cite{Collet:84,Young:85}
\begin{equation}
  \label{eq:lozi}
  l_{a,b}:\mR^2\to\mR^2,\quad l_{a,b}(x,y)=(1-a|x|+y,bx),
\end{equation}
for $(a,b)=(1.7,0.1)$, under the observables
\begin{align}
  \label{eq:obsalpha4}
  \phi_{\alpha}(x,y)&=-\dist(p,p_M)^\alpha,\quad\text{with $p=(x,y)\in\mR^2$,}\\
  \label{eq:obstheta}
  \phi_{\theta}(x,y)&=x\cos(2\pi\theta)+y\sin(2\pi\theta),
\end{align}
where $\alpha>0$ and $\theta\in[0,2\pi]$ are parameters and $p_M$ is a point
in $\mR^2$. Following the discussion for the solenoid map, we could conjecture
that
\begin{align}
  \label{eq:tail-henon-alphain}
  -\frac1\xi&=\frac{\dim_H(\Lambda)}{\alpha}\quad\text{for $\phi=\phi_{\alpha}$
    and $p_M\in\Lambda$;}\\
  \label{eq:tail-henon-thetaout}
  -\frac1\xi&=\dim_H(\Lambda)-\frac12\quad\text{for $\phi=\phi_{\theta}$.}
\end{align}
The numerical verification of these conjectures turns out to be rather
problematic. First of all for a given system it may be very hard or even
unfeasible to compute an estimate of the Hausdorff dimension.  For this
reason, we will use the Lyapunov (Kaplan-Yorke) dimension of the H\'enon or
Lozi attractor instead of the Hausdorff dimension appearing
in~\eqref{tail-henon-thetaout}-\eqref{tail-henon-alphain}. The Lyapunov
dimension of an attractor $\Lambda\subset\R^n$ of a system with $\R^n$ as
phase space is defined as
\begin{equation}
  \label{eq:dimlyap}
  \dim_L(\Lambda) = k+ \frac{\sum_{j=1}^k\chi_{j}}{-\chi_{k+1}},
\end{equation}
where $\chi_1\ge\chi_2\ge\ldots\ge\chi_n$ are the Lyapunov exponents and $k$
is the maximum index for which $ \sum_{j=1}^k\chi_{j}\ge0.$ It is believed
that the Lyapunov dimension forms an upper bound for the Hausdorff dimension
under general conditions~\cite{GP1983,Kap1984}.

For the H\'enon map under observable~\eqref{obsalpha}, and in view of the
results of a recent paper~\cite{Chazottes:10}, it is expected that
formula~\eqref{tail-henon-alphain} holds for so-called
\emph{Benedicks-Carleson parameter values}~\cite{FreitasandFreitas2008a}. Such
parameter values, however, are obtained by a perturbative argument near
$(a,b)=(2,0)$, where the bound on the smallness of $b$ is not
explicit. Moreover, the parameter exclusion methods used to define the
Benedicks-Carleson parameter values are not constructive. For these reasons,
it is not possible to say whether Benedicks-Carleson behaviour is also
attained at the ``classical'' parameter values $(a,b)=(1.4,0.3)$. Despite
this,~\eqref{tail-henon-alphain} forms our best guess for the value of $\xi$.

For planar observables, we again study the \emph{tip set}
$\Gamma\subset\Lambda$ as defined for the Solenoid map, namely, for fixed
$\epsilon_0>0$ and $p=(x,y)$, let
\begin{equation}
\Gamma=\{p\in L^{+}(\tilde{u}-\epsilon_0)\cap\Lambda:
T_{p}\gamma^{u}(p)\cdot\nabla\phi(p)=0\},
\end{equation}
and for each $\epsilon<\epsilon_0$, consider the set 
$\Gamma(\epsilon)\subset\Gamma\cap L^{+}(\tilde{u}-\epsilon)$. Then
for each $p\in\Gamma(\epsilon)$, there exists $t<\epsilon$ such that 
$\gamma^{u}(p)$ is tangent to $L^{+}(\tilde{u}-t)$. For the planar observable
$\phi$ we would expect to obtain (as with the solenoid):
\begin{equation}\label{eq:leaf-measure-henon}
m_{\gamma^u}\{\gamma^{u}(p)\cap
  L^{+}(\tilde{u}-\epsilon)\}=\mathcal{O}(\sqrt{\epsilon-t}),
\end{equation}
where $m_{\gamma^{u}}$ is the conditional (Lebesgue) measure on the
one-dimensional unstable manifold.  However in this calculation we have
assumed that the tangency between $\gamma^{u}(p)$ and $L(\tilde{u}-t)$ is
quadratic, and that the unstable segment is sufficiently long so as to cross
$L(\tilde{u}-\epsilon)$ from end to end.  For the H\'enon map both of these
conditions can fail. In particular, the H\'enon attractor admits a critical
set of folds that correspond to points where the attractor curvature is
large. More precisely the critical set is formed by homoclinic tangency points
between stable and unstable manifolds.  This set has zero measure, but it is
dense in the attractor. Furthermore the attractor has complicated geometry,
where local stable/unstable manifolds can fold back and forth upon
themselves. However, the regions that correspond to these folds (of high
curvature) occupy a set of small measure. See~\cite{WangYoung} and references
therein for a more detailed discussion.

To compute the tail index, we conjecture to have the following formula:
\begin{equation}\label{tail:henon}
-\frac{1}{\xi}=\mathrm{dim}_H(\nu)-\frac{1}{2}
\end{equation}
where $\nu$ is the SRB measure for the H\'enon map (at Benedicks-Carleson
parameters).  This would follow from the estimate:
\begin{equation}
\nu(L^{+}(\tilde{u}-\epsilon))=\int_{0}^{\epsilon}\int_{L^{+}(\tilde{u}-\epsilon)}
dm_{\gamma^{u}}d\nu_{\Gamma}\approx\sqrt\epsilon\cdot\epsilon^{\mathrm{dim}_H(\Gamma)},
\end{equation}
where the factor of $\sqrt\epsilon$ comes from equation
(\ref{eq:leaf-measure}). To obtain equation (\ref{tail:henon}), we would need
to show that $\mathrm{dim}_{H}(\Gamma) =\mathrm{dim}_{H}(\nu)-1$.  This is
perhaps harder to verify and it will depend on the regularity of the holonomy
map taken along unstable leaves.  Finally we would project this set onto a
line in the direction of $\nabla\phi(p)$, and typically the projection would
preserve the dimension.

\figref{henon-blocklens-alpha} shows the dependence of the estimates of $\xi$
with respect to the block length $N_{blocklen}$ for the H\'enon map under the
observable~\eqref{obsalpha}. We see that the estimates exhibit strong
oscillations around the value predicted by~\eqref{tail-henon-alphain} even for
fairly large block lengths.  \figref{henon-blocklens-theta} shows the
dependence of the estimates of $\xi$ with respect to the block length
$N_{blocklen}$ for the H\'enon map under the observables~\eqref{obstheta} at
$\theta=0,0.5$.  The horizontal lines represent the values predicted
by~\eqref{tail-henon-thetaout}, where, as above, we have used the of Lyapunov
(Kaplan-Yorke) dimension of the H\'enon attractor instead of the Hausdorff
dimension. We see that block lengths of at least $10^4$ are required for the
estimation to reach the neighbourhood of the value predicted
by~\eqref{tail-henon-thetaout}. However, the estimates still exhibit
substantial oscillations around the predicted values for block lengths as
large as $10^7$, although both the variability of the individual point
estimates and the estimation uncertainty are here much less pronounced than in
\figref{henon-blocklens-alpha}.

We had already seen the above behaviour in the solenoid map: namely, the
estimates in panel~(A) of \figref{solenoid-blocklens-alpha} also exhibit
larger variance and variability than those in panel~(C). In that case,
however, the theoretical value of panel~(A) is not conjectural, since it
follows from the theory discussed in \secref{solenoid} for observables such
as~\eqref{obsalpha3} when the point $p_M$ belongs to the attractor.

Hence we do not interpret the variability in
Figures~\ref{fig:henon-blocklens-alpha} and~\ref{fig:henon-blocklens-theta} as
a dismissal of~\eqref{tail-henon-thetaout}-\eqref{tail-henon-alphain}.
Rather, we claim that this behaviour is due to a problematic aspect of the
numerical estimation. To illustrate our claim, we more carefully examine the
estimates of the GEV distribution obtained for block lengths of $5000$ and
$10000$, for observable~\eqref{obstheta} with $\theta=0$. In this case, the
observable simply coincides with the projection on the $x$-axis: this is very
useful for the visualisation.

The kernel-smoothed density of the block maxima show various peaks (panel~(A1)
in \figref{attractor}). A particularly pronounced peak occurs nearby
$x=1.2727$. Examination of the points on the time series of the block maxima
(panel~(B1)) and of the points on the H\'enon attractor corresponding to the
block maxima (panel~(C1)) reveals that this peak is associated to a pair of
branches of the attractor that exhibit a turning point slightly above 1.2727.
This peak corresponds to a ``corner'' in the quantile-quantile plot
(panel~(D1)) comparing the empirical distribution of the block maxima to
fitted GEV distribution. For values of $x$ at the left of the peak, the
empirical distribution of the block maxima displays a strong deviation from
the fitted GEV distribution.

When the block length is increased to $10^4$ (right column of
\figref{attractor}), the kernel-smoothed density of the block maxima drops to
almost zero at the left of the peak (panel~(A2)). Indeed, the portion of the
H\'enon attractor corresponding to the block maxima (panel~(C2)) does \emph{no
  longer} include the two leftmost branches which were found in
panel~(C1). Moreover, a much smaller fraction of points now belongs to the
branch of the attractor having a turning point at 1.2727. This also
corresponds to the peak in the density being lower in panel~(A2) than in
panel~(A1). More importantly, this correspond to a much better overall fit to
the GEV distribution: as illustrated by the quantile-quantile plot in
panel~(D2), there still is some deviation at the lower tail, but it is orders
of magnitude smaller than in panel~(D1). 

We believe that this is the explanation for the poor convergence to the
theoretical estimates which we have found in \figref{henon-blocklens-theta},
also see~\cite{NBN06} for a related discussion.
The fractal structure portrayed in panels~(D1-2) of \figref{attractor} is
indeed present at all spatial scales near the extremal point
$\tilde{p}=(\tilde{x},\tilde{y})$ on the H\'enon attractor for which
observable~\eqref{obstheta} with $\theta=0$ is maximised. As blocks of
increasing lengths are used, increasingly many attractor branches are
discarded. Near the block length values for which one major branch is
discarded, a better agreement is obtained between the sample of block maxima
and the limiting GEV distribution. These are the block length values for which
we expect the estimated value of $\xi$ to lie closer to the theoretical
prediction in panel~(A) of \figref{henon-blocklens-theta}.

The effect of the variability in the estimates is illustrated in
\figref{xis-henon-theta}, where we show estimates of $\xi$ for
observable~\eqref{tail-henon-thetaout} with several values of $\theta$ and
with four block lengths.  For $N_{blocklen}=10^3$, the estimates vary
substantially across the range of values of $\theta$
(\figref{xis-henon-theta}~(A)).  Varying $\theta$ from 0 to 1 amounts to
rotate the level sets of the observable~\eqref{tail-henon-thetaout}, which are
straight lines. Hence, this amounts to slide the extremal point $\tilde{p}$
for which observable~\eqref{obstheta} is maximised on the H\'enon attractor
(compare with~\eqref{levelsets}).  The horizontal \emph{plateau} in
\figref{xis-henon-theta}~(A), occurring approximately for $\theta$ between
$[0.5,0.75]$, corresponds to the extremal point $\tilde{p}$ belonging to the
leftmost tip-like portions of the H\'enon attractor: large variations in
$\theta$ in this range correspond to small variations in $\tilde{p}$.

For block lengths of $N_{blocklen}=10^4$, (\figref{xis-henon-theta}~(B)), the
estimates are more uniform across $\theta$. The same holds for
$N_{blocklen}=10^5$ and $10^6$ and we see a definite bias in the latter case,
which has the same sign and approximately the same value for all $\theta$.
Roughly speaking, choosing block lengths of at least $N_{blocklen}=10^4$
ensures that we only select block maxima in branches of the H\'enon attractor
which are close to its outer ``peel'', compare with
\figref{attractor}~(C1-C2). However, this does not necessarily guarantee
accurate estimation of the limit value of $\xi$, for the reason illustrated in
\figref{attractor}~(D1-D2). 

We argue that the same explanation holds for the variability of the estimates
in panel~(A) of \figref{solenoid-blocklens-alpha} and for the even poorer
convergence in \figref{henon-blocklens-alpha}.  Plots similar to
\figref{attractor} for the latter case suggest that as block length is
increased, the probability mass that is lost at the lower tail of the
empirical distribution of the block maxima is redistributed amongst other
attractor branches which lie closer to the point $p_M^0$.  To illustrate this
process we chose observable~\eqref{obstheta} for ease of visualisation.

Similar considerations hold for the Lozi map~\eqref{lozi}.
\figref{lozi-blocklens-alpha} shows the sensitivity of the numerical estimates
of $\xi$ with respect to the block length used to compute the maxima for
observable~\eqref{obsalpha4}. For the chosen parameter values, we obtain the
estimate $\dim_L(\Lambda)=1.185$, in good agreement with the bounds $1.176669
< \dim_H(\Lambda) < 1.247848$ on the Hausdorff dimension of the Lozi attractor
$\Lambda$ proved in~\cite{Ishii:97}. When the point $p_M$ is chosen in the
attractor of the Lozi map, the estimates display strong oscillations around
the value predicted by the theory (\figref{lozi-blocklens-alpha} panel~(A)),
as in \figref{henon-blocklens-alpha}. We then choose $p_M=(0.2,0.01)$: this
point does not lie on the attractor of the Lozi map, but the nearest point
$(\tilde{x},\tilde{y})$ on the attractor belongs to one of the straight
portions.  Also in this case we observe oscillations around the theoretically
expected value (\figref{lozi-blocklens-alpha} panel~(B)), like
in~\figref{henon-blocklens-theta}.

\section{The Lorenz63 and Lorenz84 flows}
\label{sec:lorenz}

The theoretical and numerical machinery developed in the previous sections is
now applied to two paradigmatic ordinary differential equations, both derived
and studied by Ed Lorenz.  We first of all consider the model
of~\cite{Lorenz:63}:
\begin{equation}
  \begin{aligned}
    \label{eq:lorenz63}
  \dot{x} &= \sigma(y-x),r\\
  \dot{y} &= x(\rho-z) - y,   \\
  \dot{z} &= xy-\beta z,
  \end{aligned}
\end{equation}
derived from the Rayleigh equations for convection in a fluid layer between
two plates. Here $\sigma$ is the Prandtl and $\rho$ the Rayleigh number.  We
refer to this as the Lorenz63 model and fix $\sigma=10$, $\beta=8/3$ and
$\rho=28$, which is a fairly common choice in the vast literature on the
Lorenz63 system, see e.g.~\cite{sparrow,Tucker1999,Araujo:09}. The statistics
of extremes has been previously analysed in~\cite{VHF:09}, who found
smooth-like variation of the GEV parameters with respect to changes in the
parameter $\rho$ within a suitable range.

We also study a three-dimensional system proposed by Lorenz in
1984~\cite{Lorenz:84}:
\begin{equation}
  \begin{aligned}
    \label{eq:lorenz84}
    \dot{x} & = -ax - y^2 - z^2 + aF, \\
    \dot{y} & = -y + xy - bxz + G, \\
    \dot{z} & = -z + bxy + xz.
  \end{aligned}
\end{equation}
This is derived by a Galerkin projection from an infinite dimensional model
for the atmospheric circulation at mid-latitudes in the Northern Hemisphere.
The variable $x$ is the strength of the symmetric, globally averaged westerly
wind current.  The variables $y$ and $z$ are the strength of cosine and sine
phases of a chain of superposed waves transporting heat poleward.  The terms
in $b$ represent displacement of the waves due to interaction with the
westerly wind.  The coefficient $a$, if less than one, allows the westerly
wind current to damp less rapidly than the waves.  The time scale of $t$
corresponds to about 5 days. The terms in $F$ and $G$ are thermal forcings:
$F$ represents the \emph{symmetric} cross-latitude heating contrast and $G$
accounts for the \emph{asymmetric} heating contrast between oceans and
continents. System~\eqref{lorenz84} has been used in both climatological
studies, for example by coupling it with a low-dimensional model for ocean
dynamics~\cite{vanVeen:01}. Several works have examined its bifurcations,
mainly in the $(F,G)$-parameter
plane~\cite{Shilnikov:95,Masoller:92,BSV1,vanVeen:03}. Almost nothing is known
theoretically regarding the structure of its strange attractors. As in the
above references, we fix $a=0.25$ and $b=4$ and consider the chaotic dynamics
occurring at $(F,G)=(8,1)$.

We analyse time series generated by observables computed along orbits of these
flows, sampled every $\Delta t$ time units. We fix $\Delta t=0.05$ time units
for the Lorenz63 and $\Delta t=0.1$ for the Lorenz84 model. We consider the
two observables
\begin{align}
  \label{eq:obsroot}
  \phi_{1}(x,y)&=-\dist(p,p_M),\quad\text{with $p=(x,y,z)\in\mR^3$,}\\
  \label{eq:obsflat}
  \phi_{2}(x,y,z)&=x.
\end{align}
Observable $\phi_2$ has a clear physical meaning in both models:
for~\eqref{lorenz63}, the variable $x$ represents the intensity of the
convection, whereas in~\eqref{lorenz84} the variable $x$ represents the
strength of the westerly wind current. As in the previous sections, we examine
the sensitivity of the numerical estimates of $\xi$ with respect to the block
length used to compute the maxima.

We first consider the Lorenz63 system~\eqref{lorenz63}. It will be useful to
recall some geometrical facts of the Poincar\'e map to $z=\textrm{constant}$
sections. Given the planar sections $\Sigma=\{(x,y,1):|x|,|y|\leq 1\}$, and
$\Sigma'=\{(1,y,z):|y|,|z|\leq 1\}$, the map $P:\Sigma\to\Sigma$ decomposes as
$P=P_2\circ P_1$, where $P_1:\Sigma\to\Sigma'$ and $P_2:\Sigma'\to\Sigma$.  To
describe the form of $P$, let $\beta=|\lambda_{s}|/\lambda_u$,
$\beta'=|\lambda_{ss}|/\lambda_u$, where $\lambda_s$, $\lambda_{ss}$ and
$\lambda_{u}$ are the eigenvalues of the linearised Lorenz63 flow at the
origin, with $\lambda_s=-8/3$, $\lambda_{ss}=-22.83$ and $\lambda_{u}=11.83$
for our choice of parameters.  Then it can be shown that
$P_1(x,y,1)=(1,x^{\beta'}y,x^{\beta})$, and $P_2$ is a diffeomorphism, see
\cite{Holland:07}. Thus the rectangle $\Sigma^{+}=\{(x,y,1):x> 0,|y|\leq 1\}$
gets mapped into a region $P_1(\Sigma^{+})$ with a cusp at $y=0$.  The cusp
boundary can be represented as the graph $|y|=z^{\beta'/\beta}\approx
z^{8}$. The flow has a strong stable foliation, and we form the quotient space
$\widehat\Sigma=\Sigma/\sim$ by defining an equivalence relation $p\sim q$ if
$p\in\gamma^s(q)$, for a stable leave $\gamma^s$. Hence the map
$P:\Sigma\to\Sigma$ can be reduced to a uniformly expanding one-dimensional
map $f:\widehat\Sigma\to\widehat\Sigma$, with a derivative singularity at
$x=0$. Here $\widehat\Sigma$ identified with $[-1,1]$, and $f'(x)\approx
|x|^{\beta-1}$ near $x=0$.

The Lorenz flow admits an SRB measure $\nu$ which can be written 
as $\nu=\nu_{P}\times\mathrm{Leb}$ (up to a normalisation constant).
The measure $\nu_P$ is the SRB measure associated to the Poincar\'e map $P$, and
the measure is exact dimensional, i.e. the local dimension is defined $\nu$-a.e.,
see \cite{galatolo:2009}. Using the existence of the stable foliation, and the
SRB property of $\nu$, we can write $\nu_P$ as the (local) product $\nu_{\gamma^u}
\times\nu_{\gamma^s}$ where $\nu_{\gamma^u}$ is the conditional measure
on unstable manifolds, and $\nu_{\gamma^s}$ is the conditional measure
on stable manifolds. We can identify each measure
$\nu_{\gamma^u}$ (via a holonomy map) with that
of the invariant measure $\nu_{f}$ associated to $f$. The measure
$\nu_{f}$ is absolutely continuous
with respect to Lebesgue measure, but it has zero density at the endpoints
of $\widehat\Sigma$, that is
\begin{equation}
  \label{eq:zerodens}
  \nu_f([1-\epsilon,1])\approx
  \epsilon^{1/\beta}\approx\epsilon^{4.4}\quad\text{ as $\epsilon\to 0$}.
\end{equation}
From this analysis we can now conjecture the values of $\xi$.  Following the
reasoning as applied in \secref{solenoid} the conjectural values of $\xi$ are
\begin{align}
  \label{eq:xis-lorenz63-root}
  -&\frac{1}{\xi}=\dim_H(\nu),&&\text{for observable~\eqref{obsroot},}\\
  \label{eq:xis-lorenz63-flat}
  -&\frac{1}{\xi}=\frac{1}{\beta}+\frac12+\tilde{d}_s && \text{where
    $\tilde{d}_s\ll 1$ for observable~\eqref{obsflat}.}
\end{align}
The constant $\tilde{d}_s$ comes from the dimension of $\nu_s$ which is
(numerically) seen to be small due to the strong stable foliation.  As in
\secref{henon}, we replace the Hausdorff dimension with the Lyapunov
dimension, which we numerically estimate at $\dim_L\Lambda\approx2.06$. We
take this value to be the estimate of the local dimension of $\nu$. In
contrast with the solenoid and H\'enon maps, the tail index associated to
observable ~\eqref{obsroot} comes from an estimate of the measure of
$\nu(L^{+}(\tilde{u}-\epsilon))$ which we assume scales as the product of the
three factors: $\sqrt\epsilon\cdot\epsilon^{d_u}\cdot\epsilon^{8d_s}$.  Here
the factor $\sqrt\epsilon$ comes from the measure $\nu$ conditioned on
$\Lambda\cap L^{+}(\tilde{u}-\epsilon)$ in the (central)-flow direction, while
the factor $\epsilon^{d_u}$ comes from the $\nu_P$-measure conditioned on
unstable manifolds that terminate at the cusp. In a generic case we would
expect $d_u=1$. However, since we are near the cusp (namely near the boundary
$\partial\widehat\Sigma$) we have $d_u=1/\beta=4.4$ due to the zero in the
density of $\nu_{f}$, see~\eqref{zerodens}. Finally we have a contributing
factor $\epsilon^{8d_s}$ that comes from the the strength of the cusp at
$P(\partial\Sigma)$, with $d_s$ the local dimension of $\nu_{\gamma^s}$.  We
would expect typically that $d_s\approx 0.06$, but it could be much smaller if
we are in the vicinity of the cusp.

For the numerical simulations we first consider observable~\eqref{obsroot},
where $p_M$ is a point chosen in the attractor by the same procedure used
before, namely selecting the final point of an orbit of length $10^3$ time
units. The estimates converge to the theoretically expected values of
$-1/\dim_L\Lambda\approx-0.5$, see \figref{63-blocklens}~(A).  Note that
convergence is attained here for block lengths of a few thousands, unlikely
what has been observed for the H\'enon and Lozi maps.  We also obtain
convergence to the conjectured value~\eqref{xis-lorenz63-flat} for
observable~\eqref{obsflat}, see \figref{63-blocklens}~(B).  Also in this case
the convergence is much faster than for the H\'enon and Lozi maps.

For the Lorenz84 system~\eqref{lorenz84}, Lorenz detected a H\'enon like
structure in a Poincar\'e section with the plane $y=0$, see~\cite[Figures 7
and 8]{Lorenz:84}. If this conjectural structure was correct, then the
attractor would coincide with the two-dimensional unstable manifold of a
saddle-like periodic orbit of the flow of~\eqref{lorenz84}.

Assuming that there is exists an SRB measure $\nu$ supported on this attractor,
and that there is a local product structure so that $\nu$ can be written
as $\nu_{\gamma^u}\times\nu_{\gamma^s}$ (as with Lorenz63), then following 
the reasoning of \secref{solenoid} the conjectural values of $\xi$ are
\begin{align}
  \label{eq:xis-lorenz84-root}
  -&\frac{1}{\xi}=\dim_H(\nu),&&\text{for observable~\eqref{obsroot}.}\\
  \label{eq:xis-lorenz84-flat}
  -&\frac{1}{\xi}=\frac{\dim_H(\nu_{\gamma^u})}{2}+\dim_H(\nu_{\gamma^s})&&
\text{for observable~\eqref{obsflat}, with $\dim(\gamma^u)=2$.}
\end{align}
In this conjecture, it is assumed that the level sets 
$L^{+}(\tilde{u}-\epsilon)$ meet the unstable manifolds via 
generic quadratic tangencies (unlike Lorenz63).
The estimates for observable~\eqref{obsroot} display oscillations around the
theoretical value~\eqref{xis-lorenz84-root}, see~\figref{84-blocklens}~(A).
This behaviour similar to what observed for the H\'enon and Lozi maps, see
Figs.~\ref{fig:henon-blocklens-alpha} and~\ref{fig:lozi-blocklens-alpha}~(A).
The estimates in \figref{84-blocklens}~(B) display oscillations around the
value~\eqref{xis-lorenz84-flat}: again this is similar to what was observed in
the H\'enon and Lozi maps, see Figs.~\ref{fig:henon-blocklens-theta}
and~\ref{fig:lozi-blocklens-alpha}~(B).

\section{Discussion and conclusions}
\label{sec:conclusions}

This paper has presented an extension of the currently available extreme value
theory for dynamical systems to types of observables which are more similar to
those found in applications. Namely, the observables considered here are not
(necessarily) functions of the distance from a point which is generic with
respect to the invariant measure of the chaotic
system. Formula~\eqref{tail-solenoid-abcout} and its
generalisation~\eqref{xis-lorenz84-flat} were derived under generic
assumptions on the geometry of the invariant manifolds underlying the strange
attractor. Current research by the Authors aims at formulating explicit
conditions under which such formulas hold, both for uniformly and
non-uniformly hyperbolic systems. Preliminary findings suggest the following.
Suppose we have a system with an attractor $\Lambda\subset\mathbb{R}^d$ that
supports a Sinai-Ruelle-Bowen (SRB) measure $\nu$. Moreover suppose that
$\Lambda$ admits a local product structure so that $\nu$ can be locally
regarded as the product measure $\nu_{\gamma^u}\times\nu_{\gamma^s}$, where
$\nu_{\gamma^u}$ (resp. $\nu_{\gamma^s}$) are the conditional measures on
unstable (resp. stable) manifolds.  Since $\nu$ is SRB, the measures
$\nu_{\gamma^u}$ are equivalent to the Riemannian measures on the unstable
manifolds, and their local dimension $d_u$ is an integer. The local dimension
of $\nu_{\gamma^s}$ is typically non-integer. For sufficiently smooth
observables $\phi:\mathbb{R}^d\to\mathbb{R}$ that have maxima off $\Lambda$,
we conjecture that the tail index $\xi$ is given by the formula:
\begin{equation}\label{eq:general-tail}
-\frac{1}{\xi}=\frac{d_u}{2}+d_s.
\end{equation}
The factor $\frac{d_u}{2}$ comes from assuming that the level sets meet the
unstable manifolds in generic (quadratic) tangencies. The factor $d_s$ is the
local dimension of $\nu_{\gamma^s}$. We believe that this dimension $d_s$ is
equal to the dimension of the tip set $\Gamma$ as defined by equation
(\ref{eq:tipset}). Most of our examples had $\mathrm{dim}_H(\Gamma)<1$, but in
general this could be larger than 1. If this is so, then the projection of
$\Gamma$ onto a line in the direction of $\nabla\phi(\tilde{p})$ would
typically have dimension equal to one, and the intersection of $\Gamma$ with
each level set would (typically) be an uncountable set of positive Hausdorff
dimension.  Thus in addition to studying regularity of unstable holonomies, a
careful analysis of the attractor's geometry would be required when estimating
the $\nu$-measure of the level regions nearby the extremal point $\tilde{p}$.

It is of interest to verify the above formula for maps where $d_u$ is larger
than one: such is the case for the so-called quasi-periodic H\'enon-like
attractors~\cite{BSV4,BSV2,BSV3c}, which are contained in the closure of the
2D unstable manifold of a saddle-like invariant circle. For flows, this
situation corresponds to $d_u=3$, see e.g.~\cite{BSV1}.  Also, the Lorenz63
example presented in \secref{lorenz} shows beyond doubt that the geometry of
the attractor can play a substantial role in determining the limit GEV
distribution. In that case the level sets of the observable do not meet the
attractor via quadratic tangencies: instead, the level sets meet the attractor
at cusps where the measure $\nu_{\gamma^u}$ has a zero.  Therefore relation
(\ref{eq:general-tail}) fails to hold and the alternative
formula~\eqref{xis-lorenz63-flat} is derived.  This situation bears
resemblance to the configuration \figref{thom-illustrationalpha}~(C) for
Thom's map, which leads to formula~\eqref{tail-thom-alphaout2} for the tail
index.  Similarly, a modified formula for $\xi$ is expected to hold for the
Lozi map under the observable $\phi(x,y)=x$, for which the extremal point
$\tilde{p}$ coincides with a cusp-like point in the attractor.

As far as applications are concerned, this paper both points at the further
development of useful methodologies and also raises a number of significant
questions. We envisage the development of estimation methods for the
parameters of the GEV distribution which take into account the information
provided by formulas such as~\eqref{general-tail}. Given a concrete system,
parameter estimation would be complemented by an analysis of the structure of
the attractor to determine appropriate values for $d_u$ and $d_s$.
Specifically, $d_u$ could be estimated by examining Poincar\'e sections of the
attractor and/or finite time Lyapunov exponents. Calculation of Lyapunov
exponents would then yield $d_s$ through the relation
$\dim_L(\Lambda)=d_u+d_s$, which follows from the local product structure of
the invariant measure. Such analysis would also aim to ascertain whether a
formula like~\eqref{general-tail} or appropriate modifications
like~\eqref{xis-lorenz63-flat} should be used. This information could be fed
into the parameter estimation procedure in an appropriate Bayesian setting.

In the presence of parameter-dependent systems, these formulas provide an
explanation for the smooth-like dependence of extreme value statistics with
respect to changes in the control parameters. This phenomenon was first
observed in~\cite{Felicietal2007a,Felicietal2007b} and the implications for
parameter estimation in non-stationary systems were discussed
in~\cite{VHF:09}. This phenomenon critically depends on the structure of the
observables: indeed, for observables like~\eqref{energy-vorticity} we expect
formulas like~\eqref{general-tail} or~\eqref{xis-lorenz63-flat}, which could
depend smoothly on control parameters through smooth-like dependence of the
attractor dimension on control parameters. On the other hand such smooth-like
dependence is rather unlikely to occur for the observables considered so far
in the theoretical work, which are of the form~\eqref{obs-dist}.  Indeed, SRB
measures in geophysical systems are usually singular with the Lebesgue measure
in phase space, due to dissipation. Therefore, even if the point $p_M$ is
generic for the SRB measure for a given value of the control parameters, this
situation \emph{is typically not stable} under parameter variation.  If $p_M$
is fixed, then one would expect jumps in the value of $\xi$ whenever $p_M$
``drops off'' or ``drops into'' the attractor, see the discussion for
Figures~\ref{fig:xis-thom-b} and~\ref{fig:solenoid-blocklens-alpha}.

As far as the questions are concerned, the main one appears to be the
extremely slow convergence displayed by the H\'enon-like attractors considered
here (see
e.g. Figures~\ref{fig:henon-blocklens-alpha},~\ref{fig:lozi-blocklens-alpha}~\ref{fig:84-blocklens}). Such
a slow convergence has been previously observed in more complex atmospheric
models, see~\cite{vannitsem:07}. Does such a slow convergence take place in
state-of-the-art global climate models? This might pose a very serious
methodological problem for those studies aiming at quantifying climatic change
in extremes, for example changes in the behaviour of hurricanes, wind storms
and extreme rainfall.

These problems even raise the following provocative question: how relevant are
limit laws for extreme behaviour, if it takes too long for the limit to be
attained for any practical purpose? This question may have different answers.
One possibility is that novel modelling approaches could be developed to
provide more reliable estimates of extreme behaviour, not necessarily
restricted to the standard limit laws such as the GEV or the Generalised
Pareto distributions~\cite{Coles2001}. Alternatively, novel parameter
estimation procedures might be developed, that incorporate corrections or
modifications to account for the phenomena illustrated for H\'enon-like
attractors, also see \figref{attractor}. The results of this paper seem to
suggest that whatever the final answer(s), the methods will have to take into
account the geometry and the fractal nature of the strange attractors
underlying the dynamics.  We believe that these questions and problems will be
the subject of significant research efforts in the near future.

\appendix
\section{Parameter estimation for the GEV distribution}
\label{app:numerical}

We now describe the procedure which we have used to estimate the parameters
$\mu,\sigma,\xi$ of the GEV distribution~\eqref{GEV}.  Consider $N_{bmax}$
values $x_1,\ldots,x_{N_{bmax}}$ which we assume to form a random sample
from~\eqref{GEV}. For the systems under consideration, it often turns out that
the theoretically expected value of $\xi$ is smaller than $-0.5$. In such
cases, the standard maximum likelihood approach cannot be used, because the
maximum likelihood estimator is not regular~\cite{Coles2001}.  We therefore
resort to the method of L-moments~\cite{Hosking1990}.  For the GEV
distribution, the L-moments estimation equations are
\begin{align}
  \label{eq:Lmommu}
  \lambda_1&=\mu-\frac\sigma\xi(1-\Gamma(1-\xi)),\\
  \label{eq:Lmomsi}
  \lambda_2&=-\frac\sigma\xi(1-2^\xi)\Gamma(1-\xi),\\
  \label{eq:Lmomxi}
  \frac{\lambda_3}{\lambda_2}&=2\frac{1-3^\xi}{1-2^\xi}-3,
\end{align}
see Table 1 in~\cite{Hosking1990}.  Given the sample
$x_1,\ldots,x_{N_{bmax}}$, we use the \texttt{R} package \texttt{Lmoments}
(\texttt{http://cran.r-project.org/}) to estimate the first three L-moments
$\lambda_i,i=1,2,3$.  Eq.~\eqref{Lmomxi} is then solved for $\xi$ by a Newton
method, starting from the initial estimate $\hat{\xi}=7.859z+2.9554z^2$, with
$z=2/(3+\frac{\lambda_3}{\lambda_2})-\log2/\log3$, see Table 2
in~\cite{Hosking1990}. Once an estimate of $\xi$ is obtained, this is plugged
into~\eqref{Lmomsi}, which is solved for $\sigma$.  Lastly~\eqref{Lmommu} is
solved for $\mu$.

For the numerical computations, which also include quantifying the estimation
uncertainty, we adopt the following procedure. Positive integers $N_{bmax}$,
$N_{blocklen}$ and $N_{samp}$ are fixed: here $N_{bmax}$ is the total number
of block maxima to be computed and $N_{blocklen}$ is the length of the data
blocks over which each maximum has to be extracted.  We first discard a
transient of $10^5$ iterates with the map under consideration (e.g. the
H\'enon or the solenoid map). Then a total of $N_{bmax}\cdot N_{blocklen}$
iterates with the map is computed and the $N_{bmax}$ block maxima are
extracted. This sample is divided into $N_{samp}$ sub-samples, each containing
$N_{bmax}/N_{samp}$ values.  We then apply the above L-moment estimation
procedure to each sub-sample, thereby obtaining $N_{samp}$ distinct parameter
estimates
\begin{equation}
  \label{eq:subsample}
  \{(\mu_s,\sigma_s,\xi_s)\,|\,s=1,\ldots,N_{samp}\}.
\end{equation}
The sample means of the estimates~\eqref{subsample}:
\begin{equation}
  \label{eq:finalestim}
  \hat{\mu}=\frac1{N_{samp}}\sum_{s=1}^{N_{samp}}\mu_s,\quad
  \hat{\sigma}=\frac1{N_{samp}}\sum_{s=1}^{N_{samp}}\sigma_s,\quad
  \hat{\xi}=\frac1{N_{samp}}\sum_{s=1}^{N_{samp}}\xi_s
\end{equation}
are taken as the final GEV parameter estimates and the  standard deviations
\begin{equation}
  \label{eq:finaluncer}
  s_{\mu}^2=\frac1{N_{samp}}\sum_{s=1}^{N_{samp}}(\mu_s-\hat{\mu})^2,\quad
  s_{\sigma}^2=\frac1{N_{samp}}\sum_{s=1}^{N_{samp}}(\sigma_s-\hat{\sigma})^2,\quad
  s_{\xi}^2=\frac1{N_{samp}}\sum_{s=1}^{N_{samp}}(\xi_s-\hat{\xi})^2
\end{equation}
are taken as estimates of uncertainty for the final values~\eqref{finalestim}.

\section*{Acknowledgments}
This research has been carried out within the project ``PREDEX: PREdictability
of EXtremes weather events'', funded by the Complexity-NET:
\texttt{www.complexitynet.eu/}.  The Authors gratefully acknowledge support by
the UK and Dutch funding agencies involved in the Complexity-NET: the EPSRC
and the NWO.  P. R. has been partially supported by the MEC grant
MTM2009-09723.  The Authors are also indebted to their respective Institutes
for kind hospitality.

\bibliography{predex}

\def\cprime{$'$}
\begin{thebibliography}{10}

\bibitem{Araujo:09}
V.~Araujo, M.~J. Pacifico, E.~R. Pujals, and M.~Viana.
\newblock Singular-hyperbolic attractors are chaotic.
\newblock {\em Trans. Amer. Math. Soc.}, 361(5):2431--2485, 2009.

\bibitem{Beirlantetal2004}
Jan Beirlant, Yuri Goegebeur, Jozef Teugels, and Johan Segers.
\newblock {\em Statistics of Extremes: Theory and Applications}.
\newblock John Wiley and Sons, Berlin, 2004.

\bibitem{benedicks-young}
M.~Benedicks and L.-S. Young.
\newblock Markov extensions and decay of correlations for certain {H}\'enon
  maps.
\newblock {\em Asterisque}, 261:13--56, 2000.

\bibitem{BSV1}
Henk Broer, Carles Sim{\'o}, and Renato Vitolo.
\newblock Bifurcations and strange attractors in the {L}orenz-84 climate model
  with seasonal forcing.
\newblock {\em Nonlinearity}, 15(4):1205--1267, 2002.

\bibitem{BSV2}
Henk Broer, Carles Sim{\'o}, and Renato Vitolo.
\newblock Chaos and quasi-periodicity in diffeomorphisms of the solid torus.
\newblock {\em Discrete Contin. Dyn. Syst. B}, 14(3):871--905, 2010.

\bibitem{BroerTakens2011}
Henk~W. Broer and Floris Takens.
\newblock {\em Dynamical systems and chaos}, volume 172 of {\em Applied
  Mathematical Sciences}.
\newblock Springer, New York, 2011.

\bibitem{BDSSV2010}
H.W. Broer, H.A Dijkstra, C.~Sim\'o, A.E. Sterk, and R.~Vitolo.
\newblock The dynamics of a low-order model for the {A}tlantic {M}ultidecadal
  {O}scillation.
\newblock {\em Discrete Contin. Dyn. Syst. B}, 16(1):73--107, 2011.

\bibitem{Castillo}
Enrique Castillo.
\newblock {\em Extreme value theory in engineering}.
\newblock Statistical Modeling and Decision Science. Academic Press Inc.,
  Boston, MA, 1988.

\bibitem{Chazottes:10}
J.~. {Chazottes} and P.~{Collet}.
\newblock {Poisson approximation for the number of visits to balls in
  nonuniformly hyperbolic dynamical systems}.
\newblock {\em ArXiv e-prints}, July 2010.

\bibitem{Coles2001}
Stuart Coles.
\newblock {\em An Introduction to Statistical Modeling of Extreme Values}.
\newblock Springer Series in Statistics. Springer, New York, 2001.

\bibitem{Collet}
P.~Collet.
\newblock Statistics of closest return for some non-uniformly hyperbolic
  systems.
\newblock {\em Ergodic Theory and Dynamical Systems}, 21:401--420, 2001.

\bibitem{Collet:84}
P.~Collet and Y.~Levy.
\newblock Ergodic properties of the {L}ozi mappings.
\newblock {\em Comm. Math. Phys.}, 93(4):461--481, 1984.

\bibitem{Embrechts}
Paul Embrechts, Claudia Kl{\"u}ppelberg, and Thomas Mikosch.
\newblock {\em Modelling extremal events}, volume~33 of {\em Applications of
  Mathematics (New York)}.
\newblock Springer-Verlag, Berlin, 1997.
\newblock For insurance and finance.

\bibitem{Falconer}
Kenneth Falconer.
\newblock {\em Fractal geometry}.
\newblock John Wiley \& Sons Inc., Hoboken, NJ, second edition, 2003.
\newblock Mathematical foundations and applications.

\bibitem{Faranda2011}
D.~Faranda, V.~Lucarini, G.~Turchetti, and S.~Vaienti.
\newblock Numerical convergence of the block-maxima approach to the generalized
  extreme value distribution.
\newblock Preprint: arXiv:1103.0889v1, 2011.

\bibitem{Felicietal2007a}
Mara Felici, Valerio Lucarini, Antonio Speranza, and Renato Vitolo.
\newblock Extreme value statistics of the total energy in an
  intermediate-complexity model of the midlatitude atmospheric jet. part {I}:
  Stationary case.
\newblock {\em J. Atmos. Sci.}, 64(7):2137--2158, July 2007.

\bibitem{Felicietal2007b}
Mara Felici, Valerio Lucarini, Antonio Speranza, and Renato Vitolo.
\newblock Extreme value statistics of the total energy in an
  intermediate-complexity model of the midlatitude atmospheric jet. part {II}:
  Trend detection and assessment.
\newblock {\em J. Atmos. Sci.}, 64(7):2159--2175, July 2007.

\bibitem{Fre09}
A.~C.~M. Freitas.
\newblock Statistics of the maximum for the tent map.
\newblock {\em Chaos Solitons Fractals}, 42(1):604--608, 2009.

\bibitem{FF08}
A.~C.~M. Freitas and J.~M. Freitas.
\newblock On the link between dependence and independence in extreme value
  theory for dynamical systems.
\newblock {\em Statist. Probab. Lett.}, 78(9):1088--1093, 2008.

\bibitem{FFT10c}
A.~C.~M. {Freitas}, J.~M. {Freitas}, and M.~{Todd}.
\newblock {Extremal Index, Hitting Time Statistics and periodicity}.
\newblock {\em ArXiv e-prints}, August 2010.

\bibitem{FFT10b}
A.~C.~M. {Freitas}, J.~M. {Freitas}, and M.~{Todd}.
\newblock {Extreme Value Laws in Dynamical Systems for Non-smooth
  Observations}.
\newblock {\em ArXiv e-prints}, June 2010.

\bibitem{FFT10}
A.~C.~M. Freitas, J.~M. Freitas, and M.~Todd.
\newblock Hitting time statistics and extreme value theory.
\newblock {\em Probab. Theory Related Fields}, 147(3-4):675--710, 2010.

\bibitem{FreitasandFreitas2008a}
Ana Cristina~Moreira Freitas and Jorge~Milhazes Freitas.
\newblock Extreme values for benedicks-carleson quadratic maps.
\newblock {\em Ergodic Theory Dynam. Systems}, 28(4):1117--1133, 2008.

\bibitem{Gal78}
J.~Galambos.
\newblock {\em The asymptotic theory of extreme order statistics}.
\newblock John Wiley \& Sons, New York-Chichester-Brisbane, 1978.
\newblock Wiley Series in Probability and Mathematical Statistics.

\bibitem{galatolo:2009}
S.~Galatolo and M.~J. Pacifico.
\newblock Lorenz-like flows: exponential decay of correlations for the
  poincar\'e map, logarithm law, quantitative recurrence.
\newblock {\em Ergodic Theory and Dynamical Systems}, 30:1703--1737, 2009.

\bibitem{GP1983}
Peter Grassberger and Itamar Procaccia.
\newblock Measuring the strangeness of strange attractors.
\newblock {\em Phys. D}, 9(1-2):189--208, 1983.

\bibitem{Gupta}
C.~Gupta.
\newblock Extreme value distributions for some classes of non-uniformly
  partially hyperbolic dynamical systems.
\newblock Preprint, 2010.

\bibitem{Guptaetal2008}
Chinmaya Gupta, Mark Holland, and Matthew Nicol.
\newblock Extreme value theory for a class of dynamical systems modeled by
  young towers.
\newblock {\em preprint}, 2009.

\bibitem{Hai03}
G.~Haiman.
\newblock Extreme values of the tent map process.
\newblock {\em Statist. Probab. Lett.}, 65(4):451--456, 2003.

\bibitem{Hasselblatt:04}
Boris Hasselblatt and J{\"o}rg Schmeling.
\newblock Dimension product structure of hyperbolic sets.
\newblock In {\em Modern dynamical systems and applications}, pages 331--345.
  Cambridge Univ. Press, Cambridge, 2004.

\bibitem{HNT}
M.~P. Holland, M.~Nicol, and A.~T{\"o}r{\"o}k.
\newblock Extreme value distributions for non-uniformly hyperbolic dynamical
  systems.
\newblock To appear Transactions AMS, 2010.

\bibitem{Holland:07}
Mark Holland and Ian Melbourne.
\newblock Central limit theorems and invariance principles for {L}orenz
  attractors.
\newblock {\em J. Lond. Math. Soc. (2)}, 76(2):345--364, 2007.

\bibitem{Hosking1990}
J.~R.~M. Hosking.
\newblock {$L$}-moments: analysis and estimation of distributions using linear
  combinations of order statistics.
\newblock {\em J. Roy. Statist. Soc. Ser. B}, 52(1):105--124, 1990.

\bibitem{Ishii:97}
Yutaka Ishii.
\newblock Towards a kneading theory for lozi mappings. ii: Monotonicity of the
  topological entropy and hausdorff dimension of attractors.
\newblock {\em Communications in Mathematical Physics}, 190:375--394, 1997.
\newblock 10.1007/s002200050245.

\bibitem{Kap1984}
James~L. Kaplan, John Mallet-Paret, and James~A. Yorke.
\newblock The {L}yapunov dimension of a nowhere differentiable attracting
  torus.
\newblock {\em Ergodic Theory Dynam. Systems}, 4(2):261--281, 1984.

\bibitem{Leadbetter1983}
M.~R. Leadbetter.
\newblock Extremes and local dependence in stationary sequences.
\newblock {\em Z. Wahrsch. Verw. Gebiete}, 65(2):291--306, 1983.

\bibitem{LLR83}
M.~R. Leadbetter, G.~Lindgren, and H.~Rootz{\'e}n.
\newblock {\em Extremes and related properties of random sequences and
  processes}.
\newblock Springer Series in Statistics. Springer-Verlag, New York, 1983.

\bibitem{Lorenz:63}
E.~N. Lorenz.
\newblock Deterministic nonperiodic flow.
\newblock {\em J. Atmos. Sci.}, 20:130--141, 1963.

\bibitem{Lorenz:84}
Edward~N. Lorenz.
\newblock Irregularity: a fundamental property of the atmosphere*.
\newblock {\em Tellus A}, 36A(2):98--110, 1984.

\bibitem{Lucarinietal2007}
Valerio Lucarini, Antonio Speranza, and Renato Vitolo.
\newblock Parametric smoothness and self-scaling of the statistical properties
  of a minimal climate model: what beyond the mean field theories?
\newblock {\em Phys. D}, 234(2):105--123, 2007.

\bibitem{Masoller:92}
C.~Masoller, A.C.Sicardi Schifino, and Lilia Romanelli.
\newblock Regular and chaotic behavior in the new lorenz system.
\newblock {\em Physics Letters A}, 167(2):185 -- 190, 1992.

\bibitem{NBN06}
C.~Nicolis, V.~Balakrishnan, and G.~Nicolis.
\newblock Extreme events in deterministic dynamical systems.
\newblock {\em Physical Review Letters}, 97:210602, 2006.

\bibitem{Res87}
S.~I. Resnick.
\newblock {\em Extreme values, regular variation, and point processes},
  volume~4 of {\em Applied Probability. A Series of the Applied Probability
  Trust}.
\newblock Springer-Verlag, New York, 1987.

\bibitem{Shilnikov:95}
A.~Shil{\cprime}nikov, G.~Nicolis, and C.~Nicolis.
\newblock Bifurcation and predictability analysis of a low-order atmospheric
  circulation model.
\newblock {\em Internat. J. Bifur. Chaos Appl. Sci. Engrg.}, 5(6):1701--1711,
  1995.

\bibitem{Sim97}
K.~Simon.
\newblock The {H}ausdorff dimension of the {S}male-{W}illiams solenoid with
  different contraction coefficients.
\newblock {\em Proc. Amer. Math. Soc.}, 125(4):1221--1228, 1997.

\bibitem{sparrow}
Colin Sparrow.
\newblock An introduction to the {L}orenz equations.
\newblock {\em IEEE Trans. Circuits and Systems}, 30(8):533--542, 1983.

\bibitem{Sterk:10}
A.E. Sterk, R.~Vitolo, H.W. Broer, C.~Sim\'{o}, and H.A. Dijkstra.
\newblock New nonlinear mechanisms of midlatitude atmospheric low-frequency
  variability.
\newblock {\em Physica D}, 239:702--718, 2010.

\bibitem{Tucker1999}
Warwick Tucker.
\newblock The {L}orenz attractor exists.
\newblock {\em C. R. Acad. Sci. Paris S\'er. I Math.}, 328(12):1197--1202,
  1999.

\bibitem{vanVeen:03}
Lennaert van Veen.
\newblock Baroclinic flow and the {L}orenz-84 model.
\newblock {\em Internat. J. Bifur. Chaos Appl. Sci. Engrg.}, 13(8):2117--2139,
  2003.

\bibitem{vanVeen:01}
Lennaert Van~Veen, Theo Opsteegh, and Ferdinand Verhulst.
\newblock Active and passive ocean regimes in a low-order climate model.
\newblock {\em Tellus A}, 53(5):616--628, 2001.

\bibitem{vannitsem:07}
S.~Vannitsem.
\newblock Statistical properties of the temperature maxima in an intermediate
  order quasi-geostrophic model.
\newblock {\em Tellus A}, 59(1):80--95, 2007.

\bibitem{VHF:09}
R.~Vitolo, M.~P. Holland, and C.~A.~T. Ferro.
\newblock Robust extremes in chaotic deterministic systems.
\newblock {\em Chaos}, 19:043127, 2009.

\bibitem{Vitoloetal2008}
Renato Vitolo, Paolo Ruti, Alessandro dell'Aquila, Mara Felici, Valerio
  Lucarini, and Antonio Speranza.
\newblock Accessing extremes of mid-latitudinal wave activity: methodology and
  application.
\newblock {\em Tellus A}, 61:35--49, 2009.

\bibitem{BSV3c}
Renato Vitolo, Carles Sim{\'o}, and Henk Broer.
\newblock {Routes to chaos in the {H}opf-saddle-node bifurcation for fixed
  points of 3D-diffeomorphisms}.
\newblock {\em Nonlinearity}, 23:1919--1947, 2010.

\bibitem{BSV4}
Renato Vitolo, Carles Sim{\'o}, and Henk Broer.
\newblock Quasi-periodic bifurcations of invariant circles in low-dimensional
  dissipative dynamical systems.
\newblock {\em Regul. Chaotic Dyn.}, 16(1-2):154--184, 2011.

\bibitem{VitoloandSperanza2011}
Renato Vitolo and Antonio Speranza.
\newblock {Vortex statistics in a simple quasi-geostrophic model}.
\newblock {\em Preprint}, 2011.

\bibitem{WangYoung}
Qiudong Wang and Lai-Sang Young.
\newblock Toward a theory of rank one attractors.
\newblock {\em Ann. of Math. (2)}, 167(2):349--480, 2008.

\bibitem{Young:85}
Lai-Sang Young.
\newblock Bowen-{R}uelle measures for certain piecewise hyperbolic maps.
\newblock {\em Trans. Amer. Math. Soc.}, 287(1):41--48, 1985.

\bibitem{Young2002}
Lai-Sang Young.
\newblock What are {SRB} measures, and which dynamical systems have them?
\newblock {\em J. Statist. Phys.}, 108(5-6):733--754, 2002.

\end{thebibliography}

\begin{figure}[p]
  \centering
  \psfrag{Attractor}{$\Lambda$}
  \psfrag{An}{$A_n$}
  \psfrag{pM}{$p_M$}
  \psfrag{L(u)}{$L(u)$}
  \includegraphics[height=0.23\textwidth]{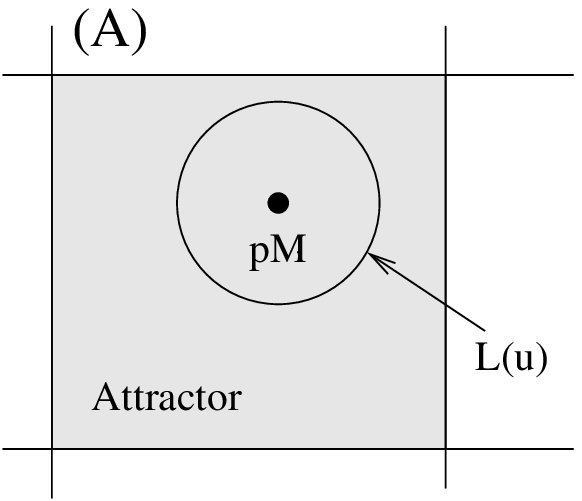}\quad
  \includegraphics[height=0.23\textwidth]{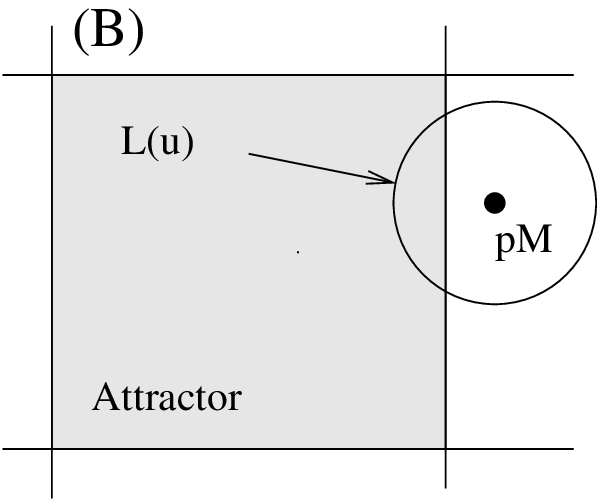}
  \includegraphics[height=0.27\textwidth]{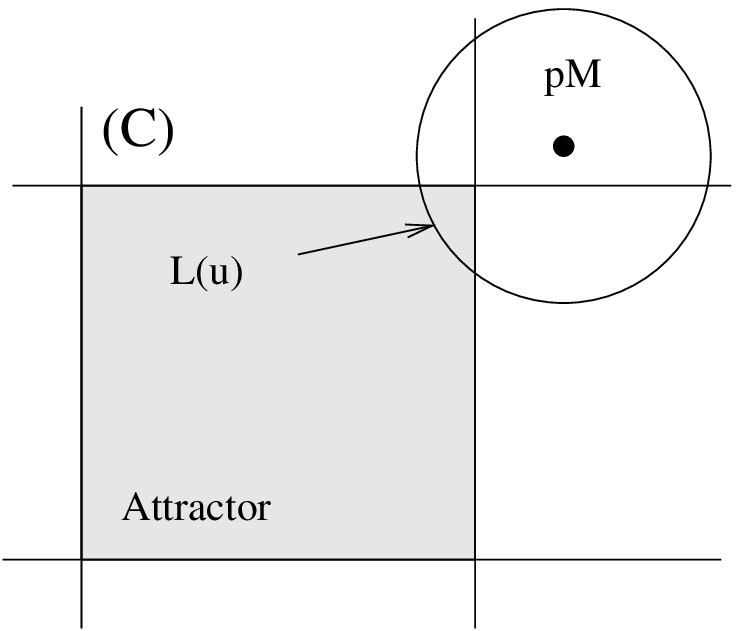}
  \caption{%
    Sketch of the three situations considered in \thmref{thom-theoryalpha} for
    the level sets $L(u)$ (defined in~\eqref{levelsets}) for observable
    $\phi_\alpha$~\eqref{obsalpha}. }
  \label{fig:thom-illustrationalpha}
\end{figure}

\begin{figure}[p]
  \centering
  \psfrag{Attractor}{$\Lambda$}
  \psfrag{An}{$A_n$}
  \psfrag{pM}{$p_M$}
  \psfrag{L(u)}{$L(u)$}
  \includegraphics[height=0.23\textwidth]{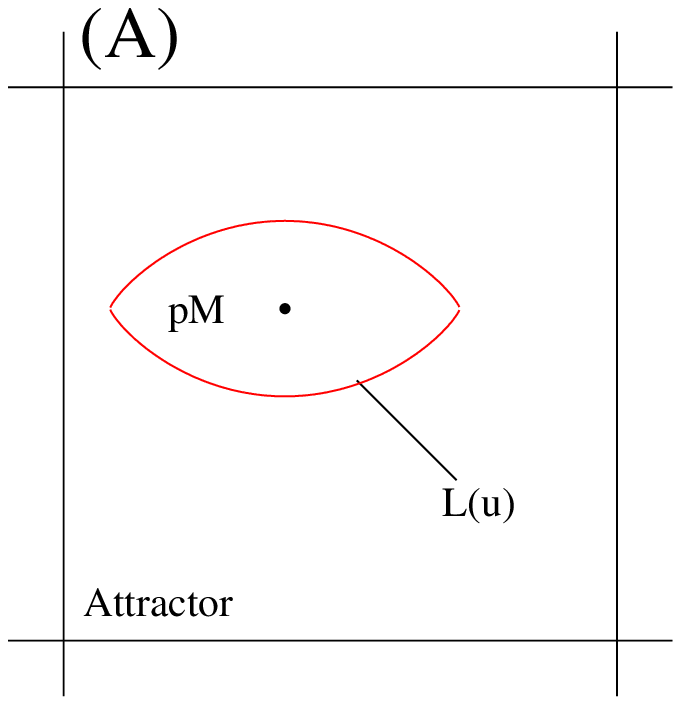}\quad
  \includegraphics[height=0.23\textwidth]{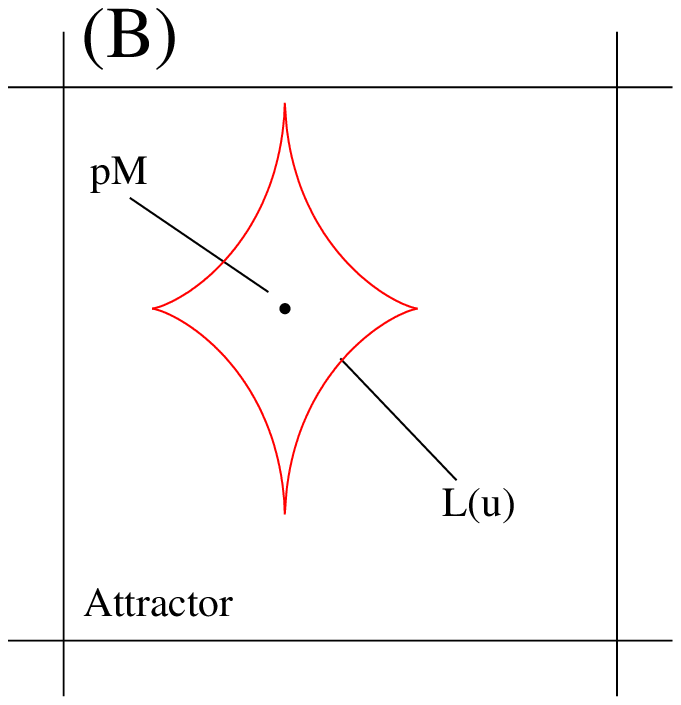}
  \includegraphics[height=0.23\textwidth]{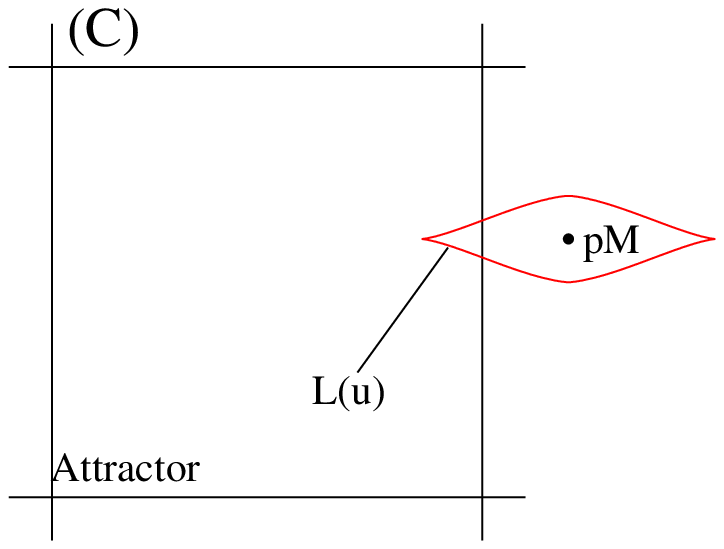}
  \caption{%
    Sketch of a few possible configurations for the level sets $L(u)$
    (defined in~\eqref{levelsets}) for observable $\phi_{ab}$~\eqref{obsab}.
    (A) $(a,b)=(2,1.25)$; (B) $(a,b)=(0.5,0.75)$; 
    (C) $(a,b)=(1.5,0.7)$.
  }
  \label{fig:thom-illustrationab}
\end{figure}

\begin{figure}[p]
  \centering 
  \psfrag{xi}{$\xi$}
  \psfrag{Nlength}{$N_{blocklen}$}
  \includegraphics[width=0.49\textwidth]{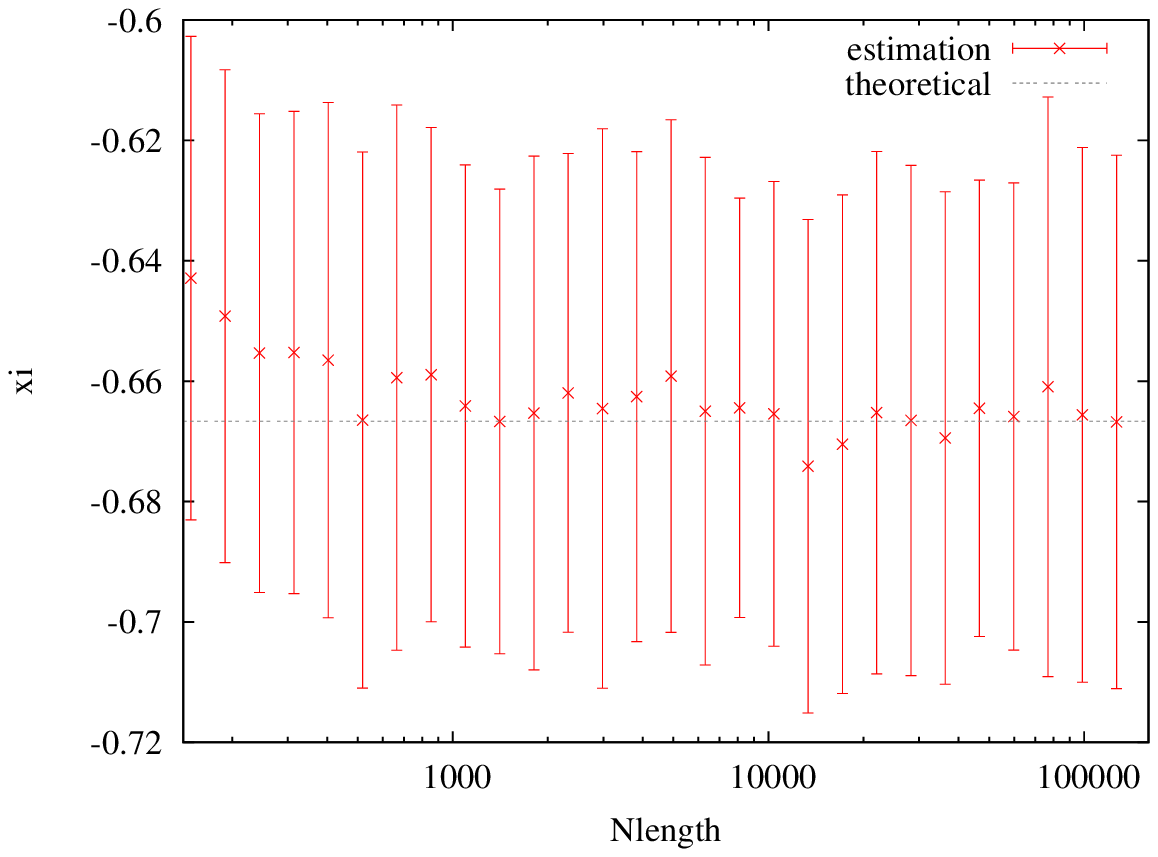}
  \includegraphics[width=0.49\textwidth]{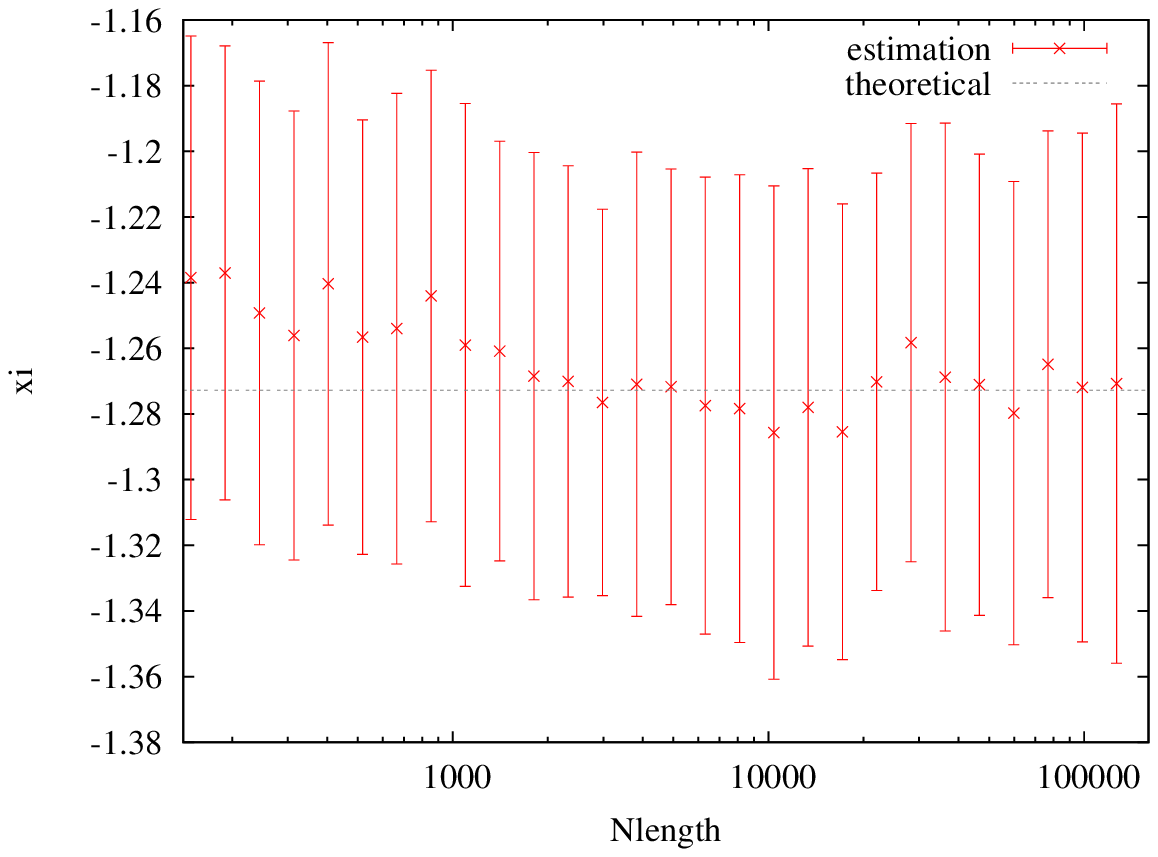}
  \caption{%
    Point estimates (crosses) and estimation uncertainty (vertical bars) of
    the tail index $\xi$ versus block length $N_{blocklen}$ for Thom's
    map~\eqref{thom} under the observable~\eqref{obsab} with $a=2$ and
    $p_M=(0.510001,0.5090001)$ fixed, where $b/2=0.5$ (left) and $b/2=1.75$
    (right). The horizontal dashed lines represent theoretically expected
    values according to~\eqref{tail-thom-abin}.  Crosses and vertical bars
    are the mean and $\pm$ one standard deviation of
    a sample of $N_{samp}=100$ individual estimates along a single orbit.
    Individual estimates are obtained by the method of L-moments
    with sequences of $N_{bmax}= 50000$ block maxima over blocks of length
    $N_{blocklen}$, as described in Appendix~\ref{app:numerical}.
  }
  \label{fig:thom-blocklens-p}
\end{figure}

\begin{figure}[p]
  \centering
  \psfrag{xi}{$\xi$}
  \psfrag{b}{$b$}
  \includegraphics[width=0.5\textwidth]{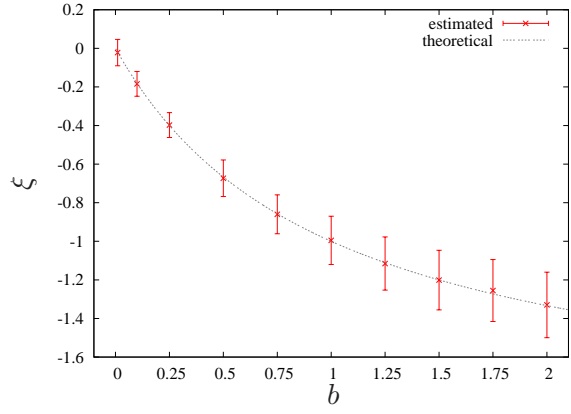}
  \caption{%
    Point estimates (crosses) and estimation uncertainty (vertical bars) of
    the tail index $\xi$ versus parameter $b$ for Thom's map~\eqref{thom}
    under the observable~\eqref{obsab} with $a=2$ and
    $p_M=(0.510001,0.5090001)$ fixed and varying
    $b/2=0.01,0.1,0.25,0.5,0.75,1,1.25,1.5,1.75,2$.  The dashed line
    represents theoretically expected values according
    to~\eqref{tail-thom-abin}.  Point and interval estimates are obtained by
    the method of L-moments as for \figref{xis-thom-p}, with $N_{bmax}=
    10000$, $N_{blocklen}=10000$ and $N_{samp}=100$, see
    Appendix~\ref{app:numerical}.
    }
  \label{fig:xis-thom-p}
\end{figure}

\begin{figure}[p]
  \centering 
  \psfrag{xi}{$\xi$}
  \psfrag{Nlength}{$N_{blocklen}$}
  \includegraphics[width=0.47\textwidth]{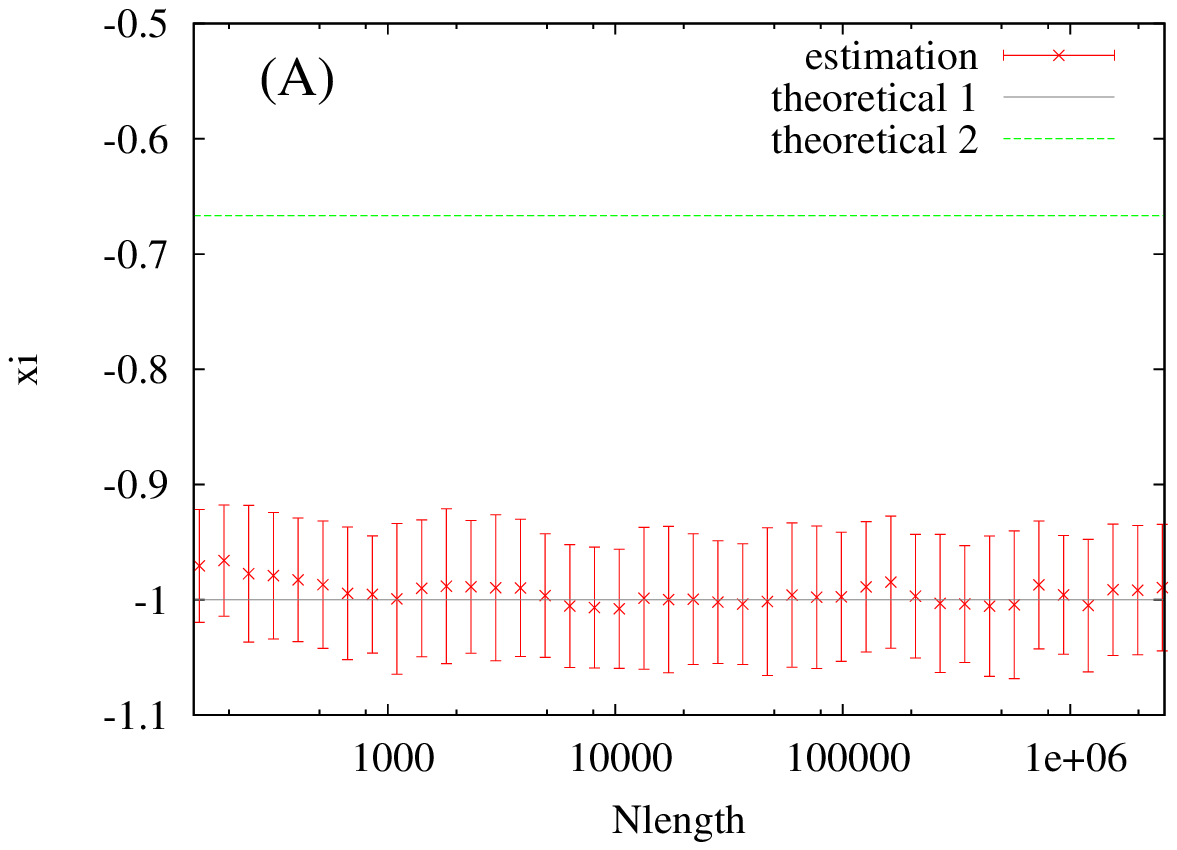}
  \includegraphics[width=0.47\textwidth]{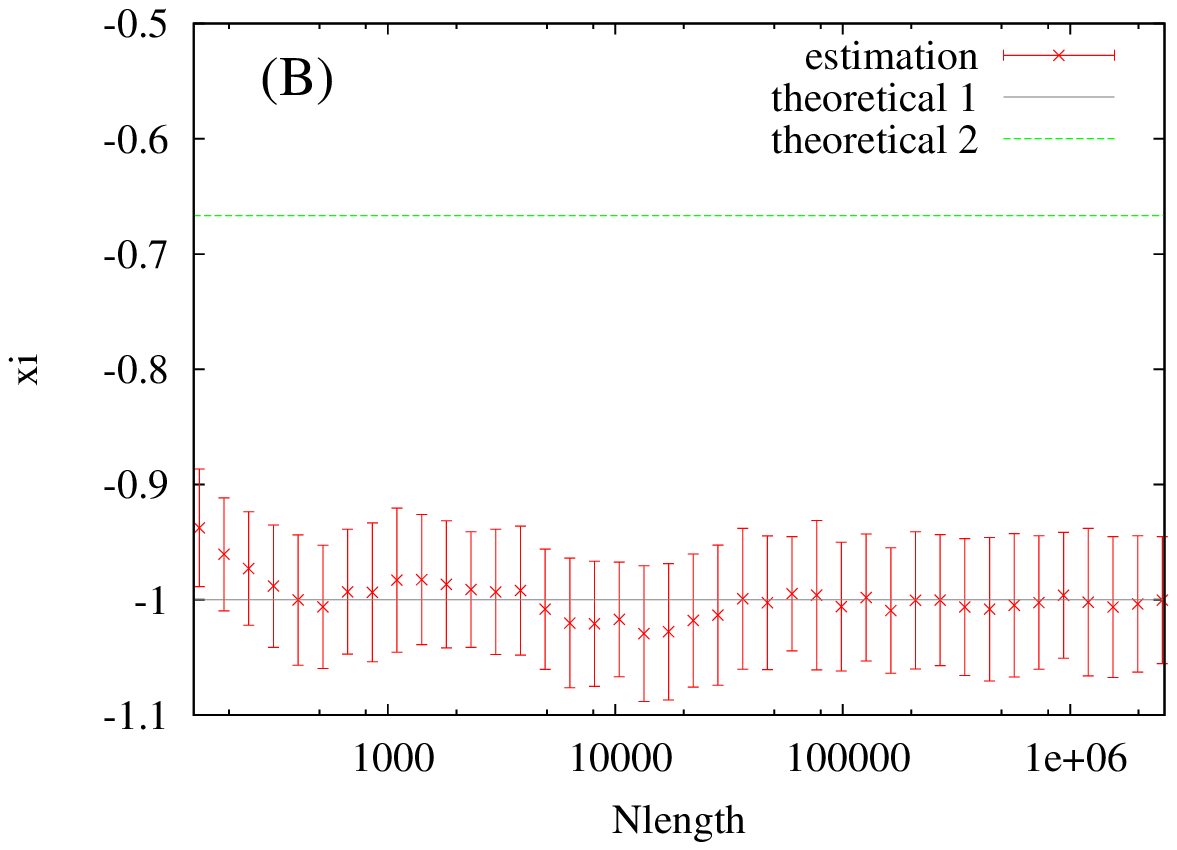}
  \includegraphics[width=0.47\textwidth]{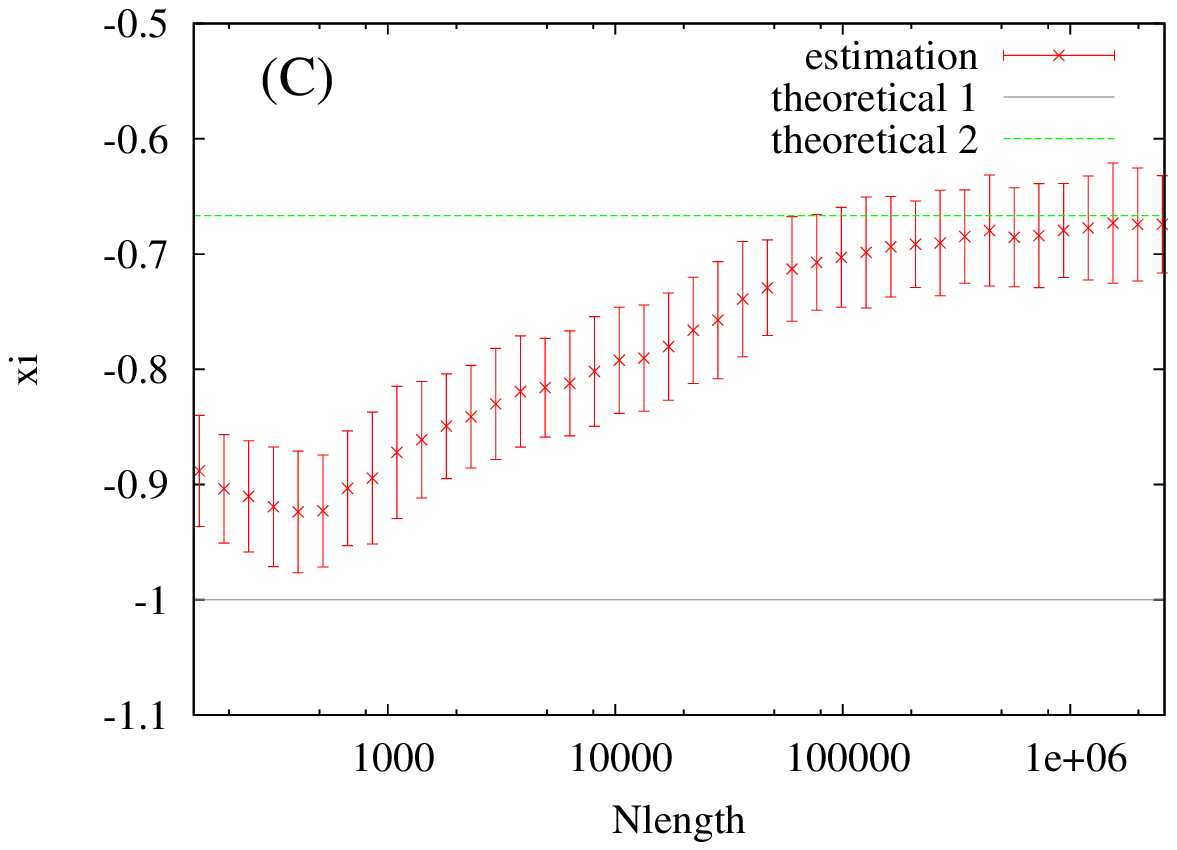}
  \includegraphics[width=0.47\textwidth]{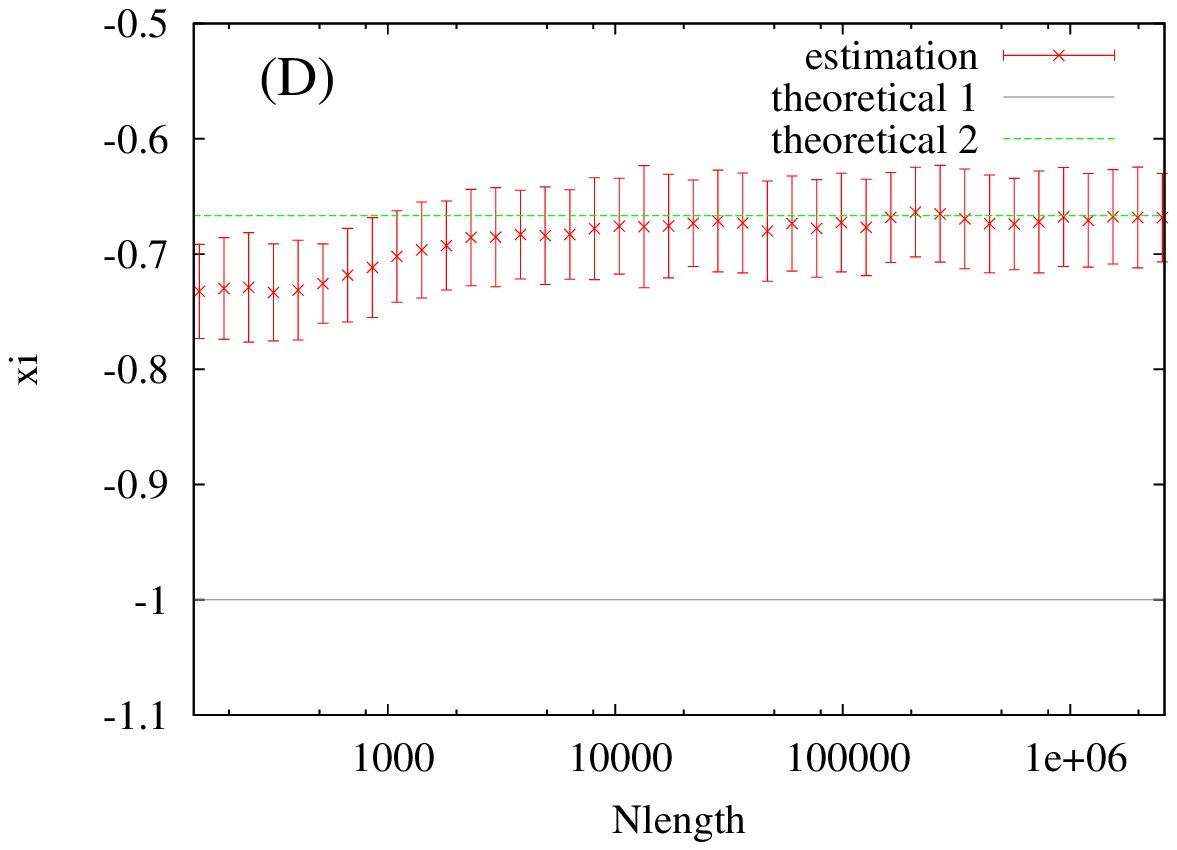}
  \caption{%
    Point estimates (crosses) and estimation uncertainty (vertical bars) of
    the tail index $\xi$ versus block length $N_{blocklen}$ for Thom's
    map~\eqref{thom} under the observable~\eqref{obsalpha} with $\alpha=1$
    and $p_M=(x_M,y_M)$ with $y_M=0.510001$ fixed and (A)
    $x_M= 0.9$, (B) $x_M=1.0$, (C) $x_M=1.01$, (D) $x_M=1.1$.
    Horizontal lines labelled by 1, 2 represent theoretical values according
    to~\eqref{tail-thom-alphain} and~\eqref{tail-thom-alphaout1},
    respectively.
    Estimates are obtained by the method of L-moments as for
    \figref{xis-thom-p}, with
    $N_{bmax}= 50000$ and $N_{samp}=100$, see
    Appendix~\ref{app:numerical}.  
  }
  \label{fig:thom-blocklens-b}
\end{figure}

\begin{figure}[p]
  \centering
  \psfrag{xM}{$x_M$}
  \psfrag{xi}{$\xi$}
  \includegraphics[width=0.495\textwidth]{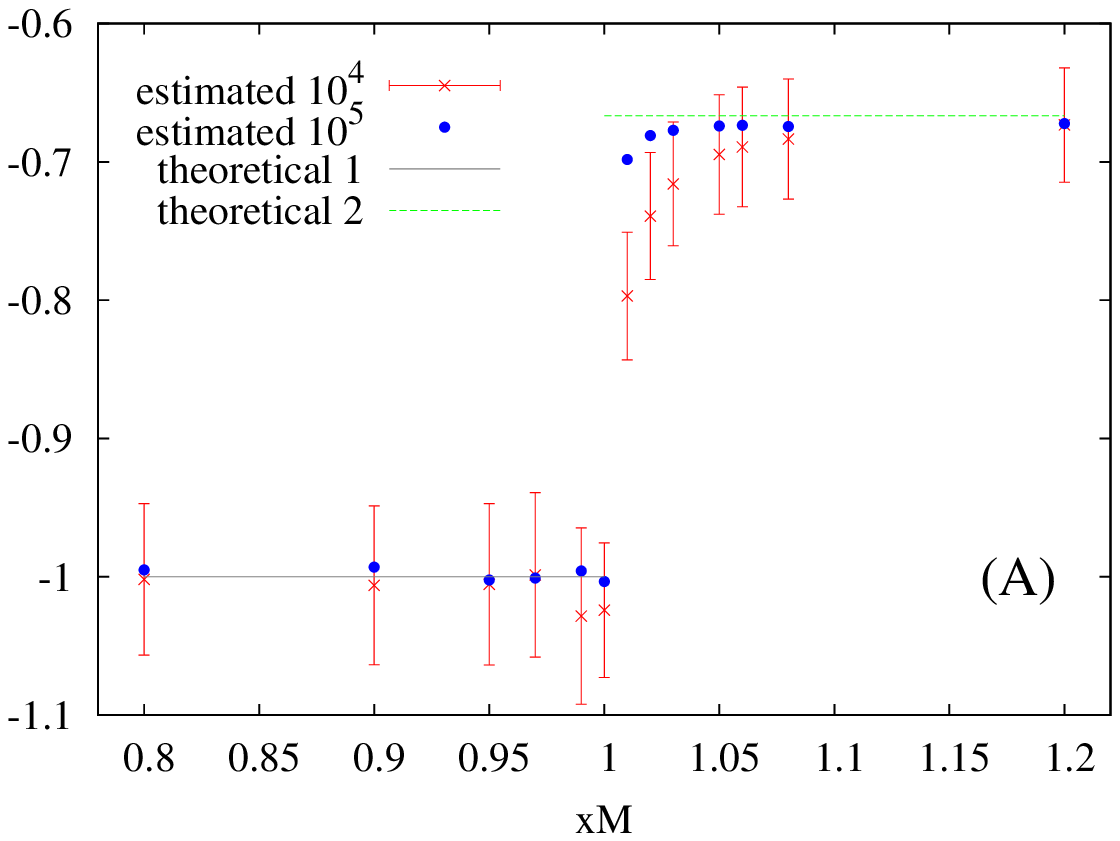}
  \psfrag{yM}{$y_M$}
  \includegraphics[width=0.495\textwidth]{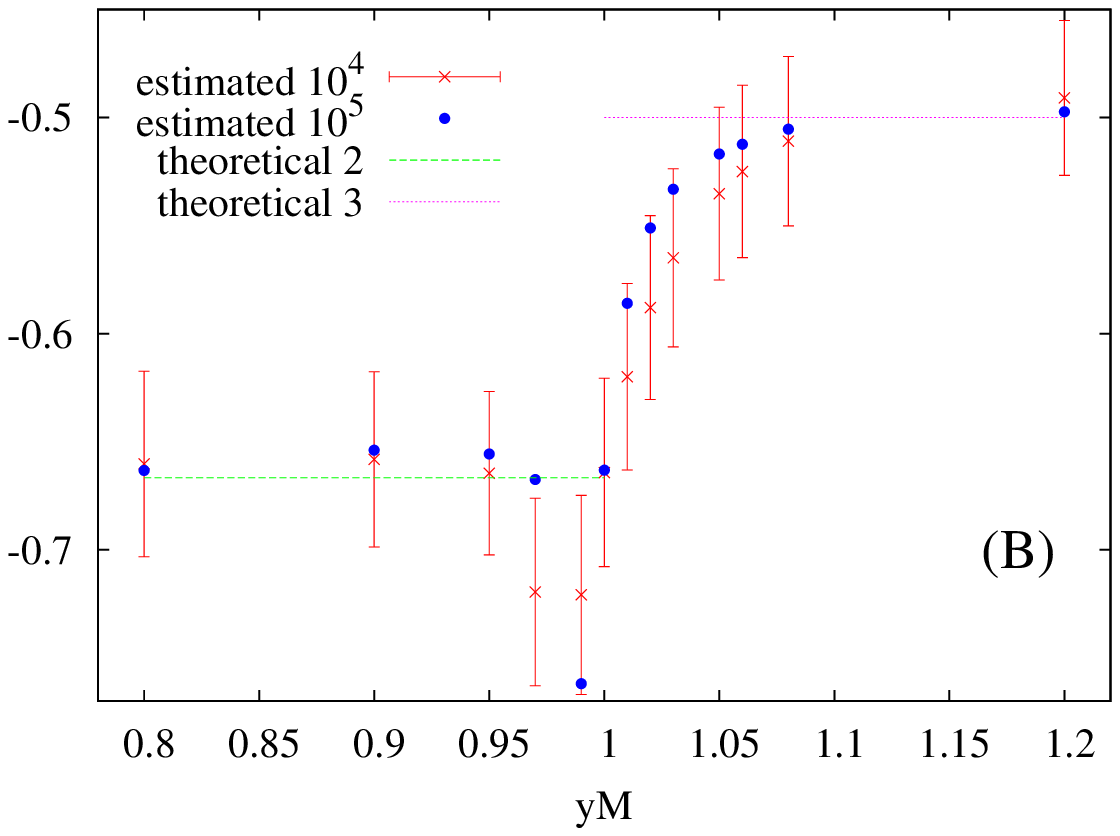}
  \caption{%
    Point estimates (crosses) and estimation uncertainty (vertical bars) of
    the tail index $\xi$ for Thom's map~\eqref{thom} under the
    observable~\eqref{obsalpha} with $\alpha=2$ and
    (A)~$y_M=0.5090001$ and $x_M$ varying, and (B)~$x_M=1.2$ and $y_M$ varying.
    Horizontal lines labelled by 1, 2, 3 represent theoretical values according
    to~\eqref{tail-thom-alphain},~\eqref{tail-thom-alphaout1}
    and~\eqref{tail-thom-alphaout2}, respectively.
    The method of L-moments was used as
    described in Appendix~\ref{app:numerical} with $N_{bmax}= 10000$ and
    $N_{samp}=100$ fixed, with $N_{blocklen}=10^4$ (red) and $N_{blocklen}
    =10^5$ (blue).
  }
  \label{fig:xis-thom-b}
\end{figure}

\begin{figure}[p]
  \centering
  \psfrag{xi}{$\xi$}
  \psfrag{Nlength}{$N_{blocklen}$}
  \includegraphics[width=0.495\textwidth]{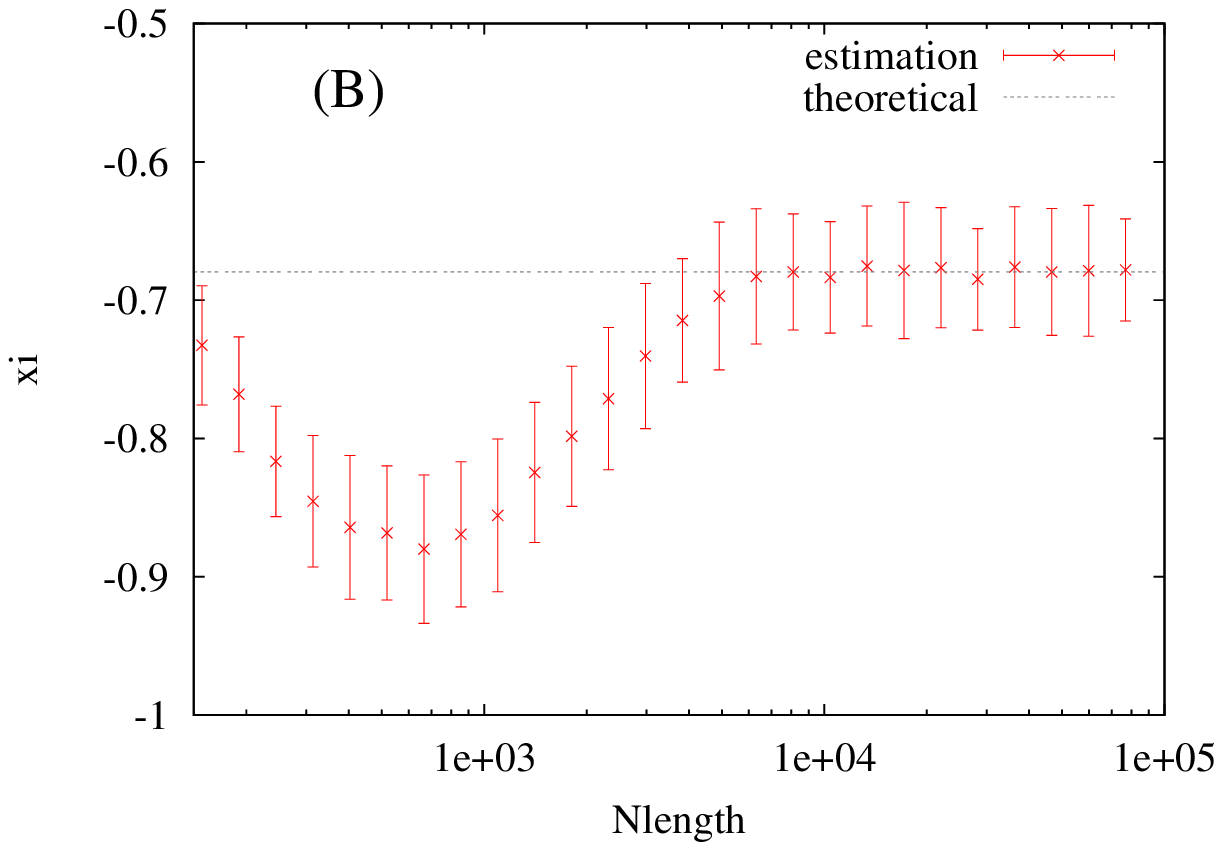}
  \includegraphics[width=0.495\textwidth]{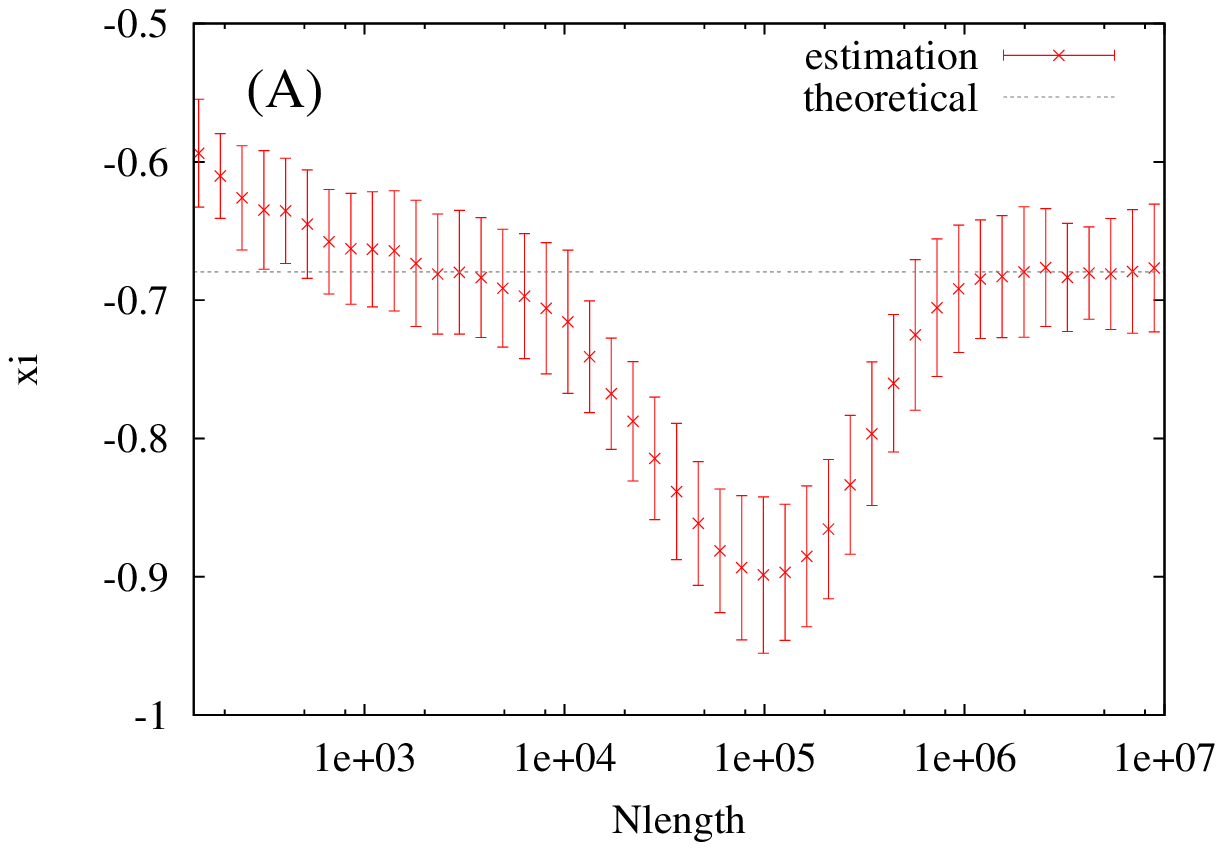}
  \caption{%
    Point estimates (crosses) and estimation uncertainty (vertical bars) of
    the tail index $\xi$ versus block length $N_{blocklen}$ for the Solenoid
    map~\eqref{solenoid} under the observable~\eqref{obsthetaxy} (left)
    and~\eqref{obsthetaxz} (right) with $\theta=0.5$.
    The horizontal dashed lines represent theoretically
    expected values according to~\eqref{tail-solenoid-abcout}
    Estimates are obtained by the method of L-moments as for
    \figref{xis-thom-p}, with
    $N_{bmax}= 10000$ and $N_{samp}=100$, see
    Appendix~\ref{app:numerical}.
  }
  \label{fig:blocklens-soli-theta}
\end{figure}

\begin{figure}[p]
  \centering
  \psfrag{xi}{$\xi$}
  \psfrag{theta}{$N_{blocklen}$}
  \includegraphics[width=0.32\textwidth]{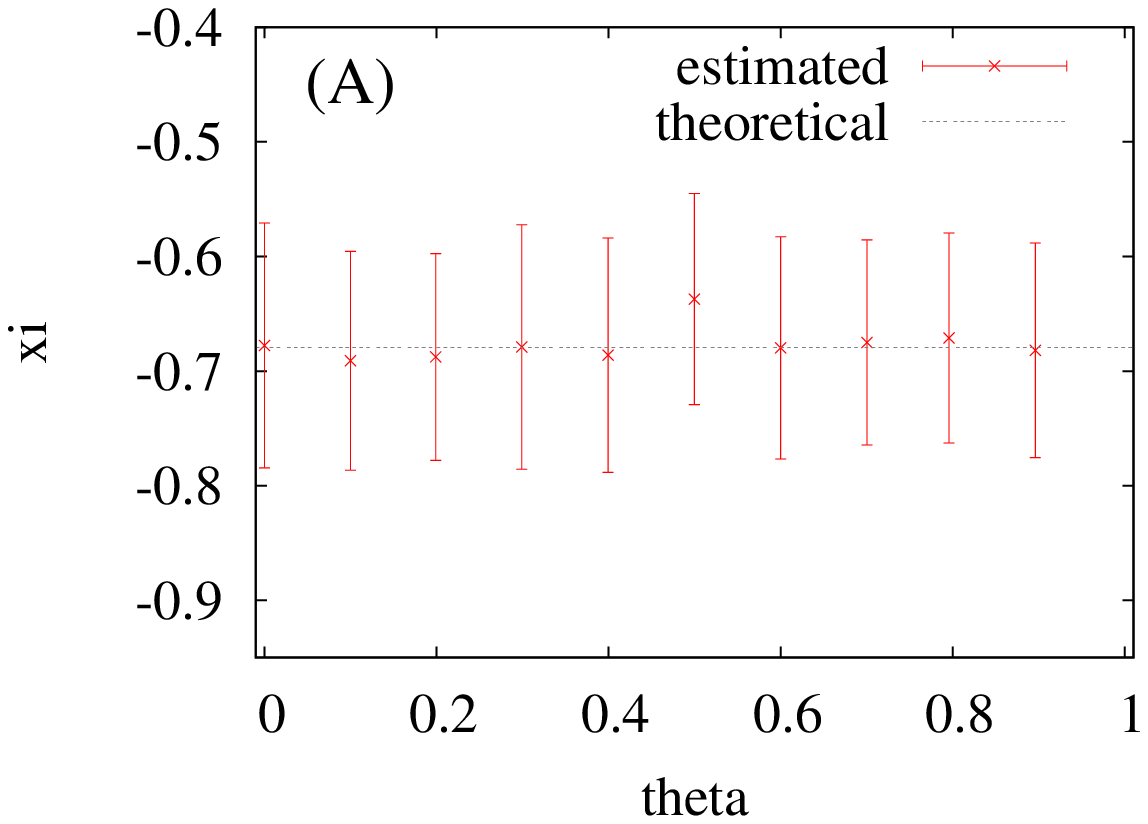}
  \includegraphics[width=0.32\textwidth]{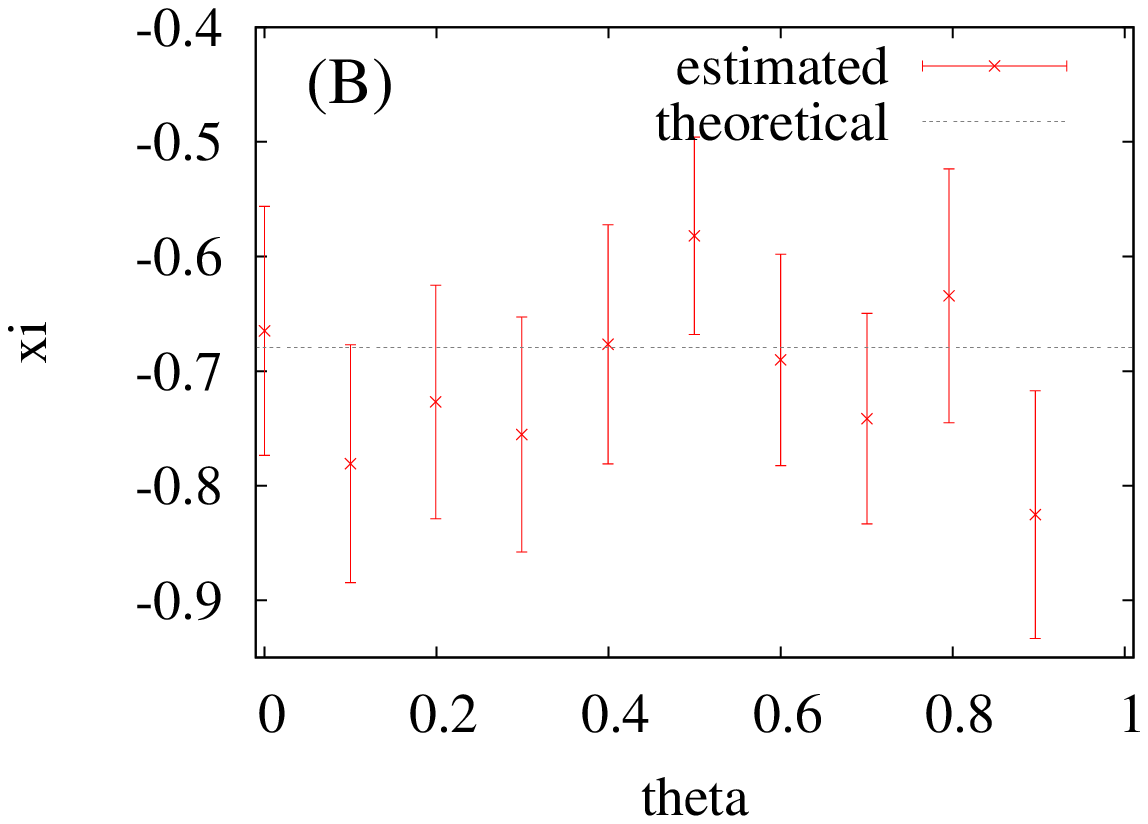}
  \includegraphics[width=0.32\textwidth]{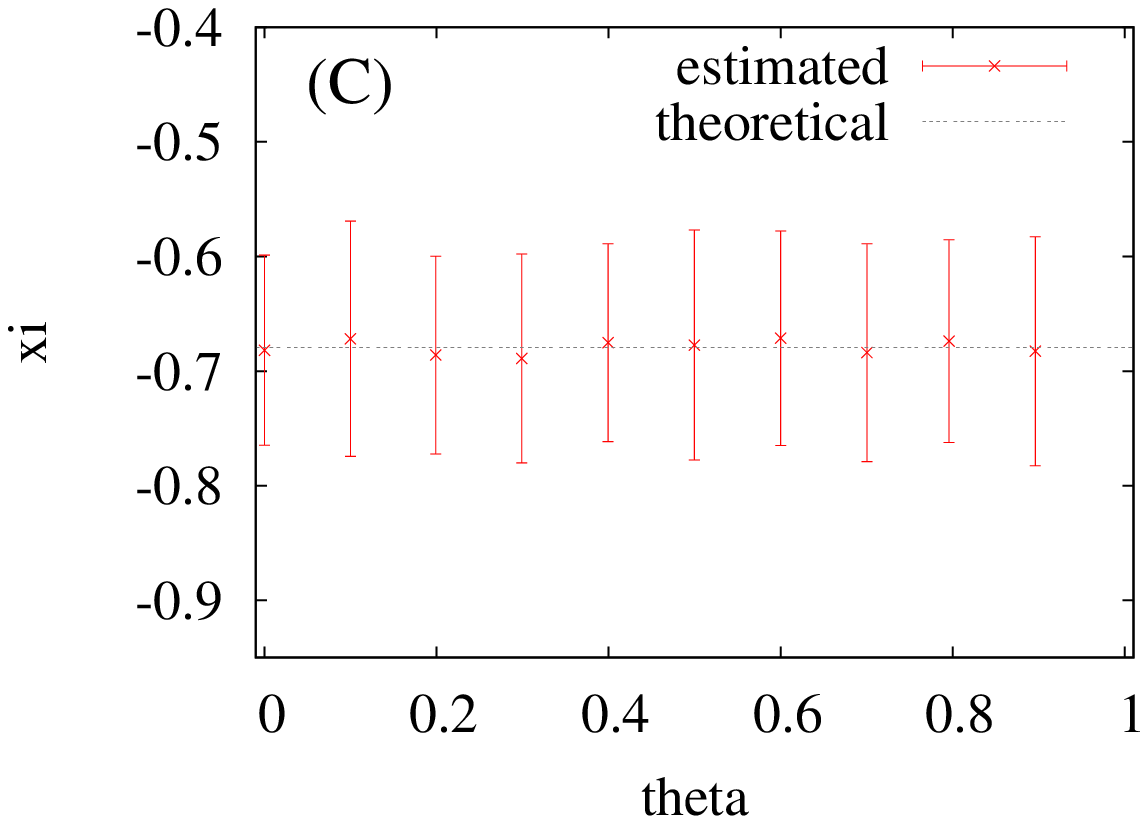}
  \caption{%
    Point estimates (crosses) and estimation uncertainty (vertical bars) of
    the tail index $\xi$ versus parameter $\theta$ for the solenoid
    map~\eqref{solenoid} under the observable~\eqref{obsthetaxy}~(A)
    and~\eqref{obsthetaxz} (B,C) with $x_0=y_0=z_0=3$, where
    $\theta=\frac{i}{10}$ for $i= 0, \dots, 9$.
    The dashed line represents theoretically expected
    values according to~\eqref{tail-solenoid-abcout}.  Estimates are obtained
    by the method of L-moments as for \figref{xis-thom-p}, with $N_{bmax}=
    10000$ and $N_{samp}=100$, see Appendix~\ref{app:numerical}, where
    $N_{blocklen}=10000$ (A,C) and $N_{blocklen}=10^6$ (B).
  }
  \label{fig:xis-solenoid-theta1}
\end{figure}

\begin{figure}[p]
  \centering
  \psfrag{xi}{$\xi$}
  \psfrag{lambda}{$\lambda$}
  \includegraphics[width=0.495\textwidth]{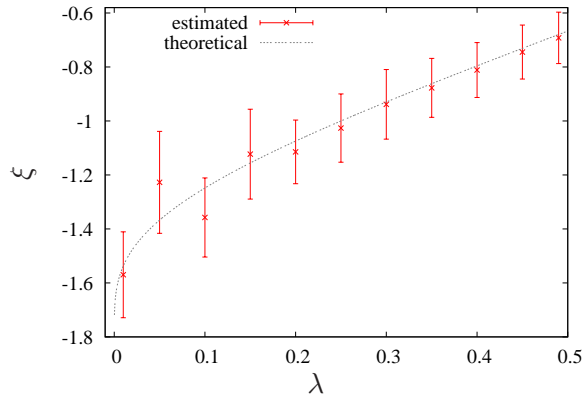}
  \caption{%
    Point estimates (crosses) and estimation uncertainty (vertical bars) of
    the tail index $\xi$ versus parameter $\lambda$ for the solenoid
    map~\eqref{solenoid} under the observable~\eqref{obsthetaxy} with
    $\theta=x_0=y_0=0$. The dashed line represents theoretically expected
    values according to~\eqref{tail-solenoid-abcout}.
    Estimates are obtained  by the method of L-moments as for
    \figref{xis-thom-p}, with $N_{bmax}=50000$, $N_{samp}=100$ and  $N_{blocklen}=10000$,
    see Appendix~\ref{app:numerical}.
  }
  \label{fig:xis-solenoid-lambda}
\end{figure}

\begin{figure}[p]
  \centering
  \psfrag{xi}{$\xi$}
  \psfrag{Nlength}{$N_{blocklen}$}
  \includegraphics[width=0.47\textwidth]{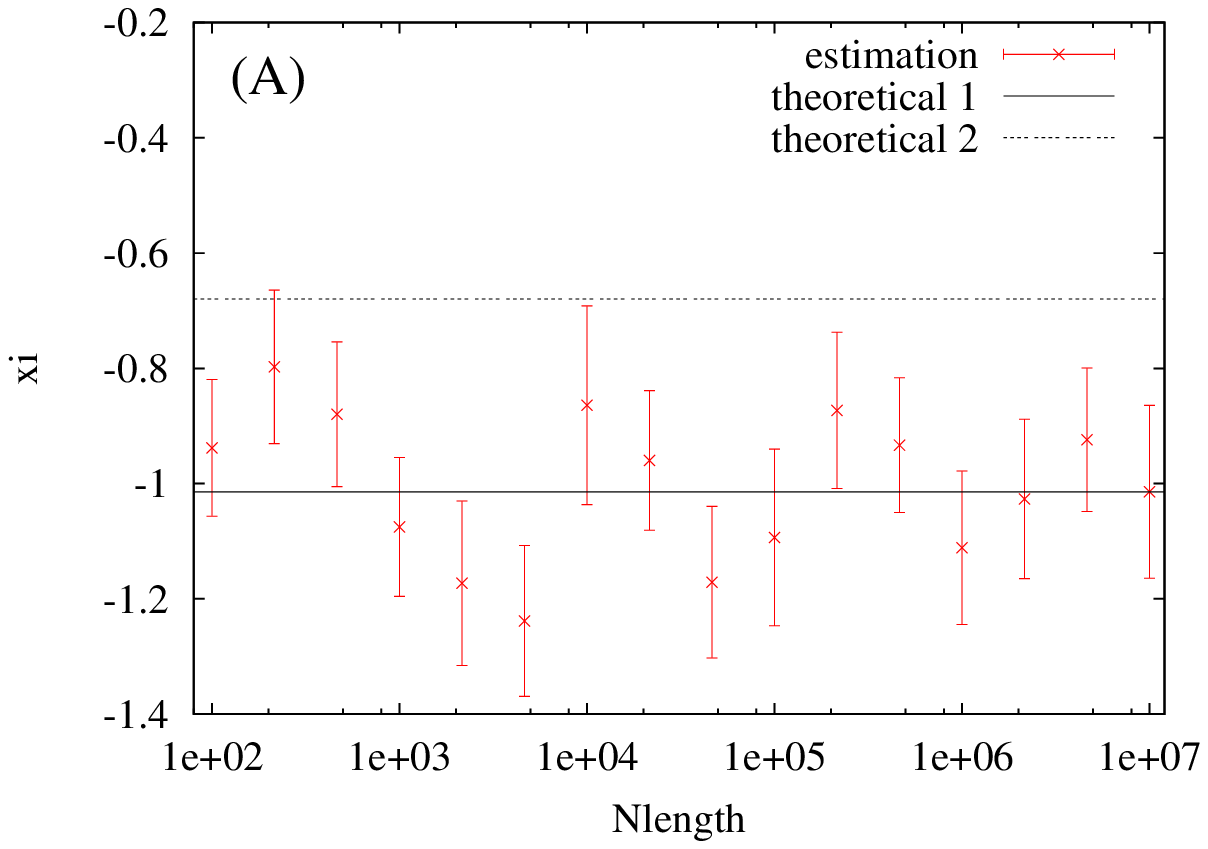}
  \includegraphics[width=0.47\textwidth]{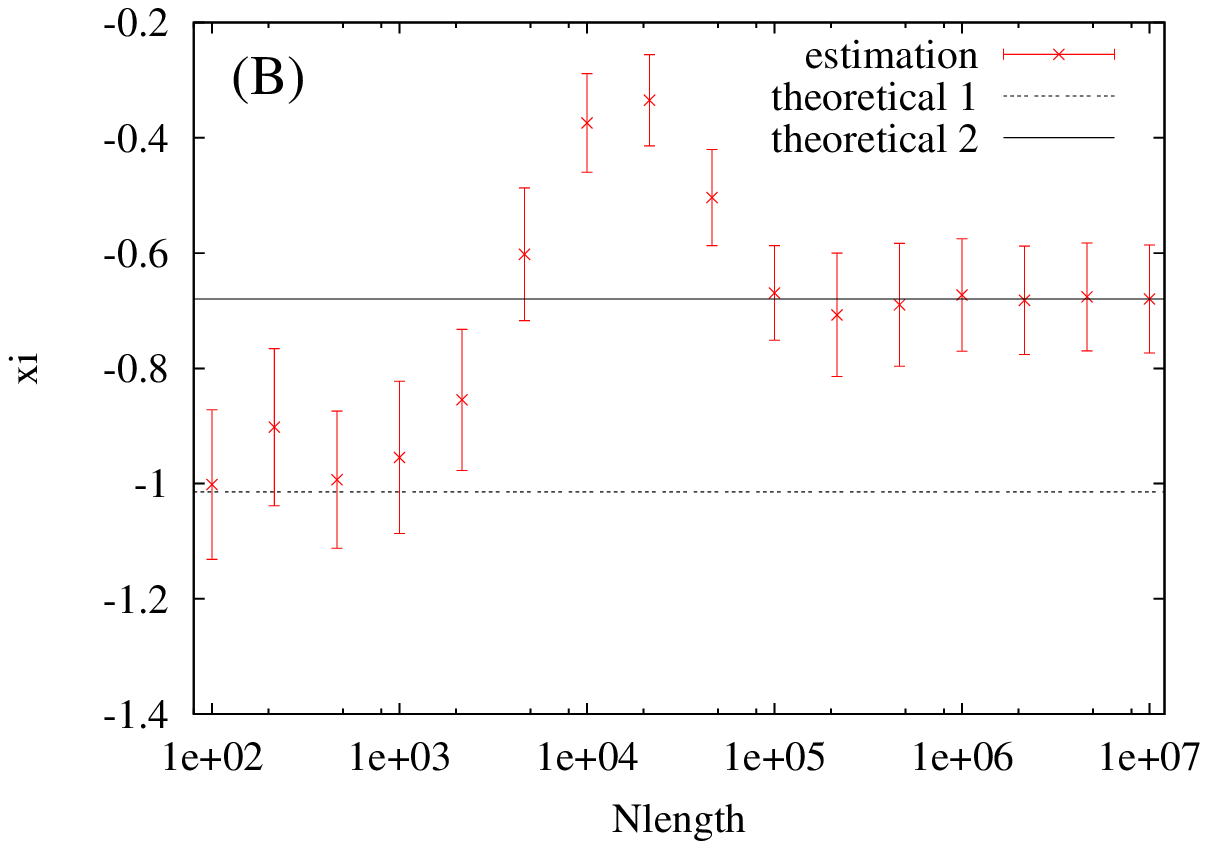}
  \includegraphics[width=0.47\textwidth]{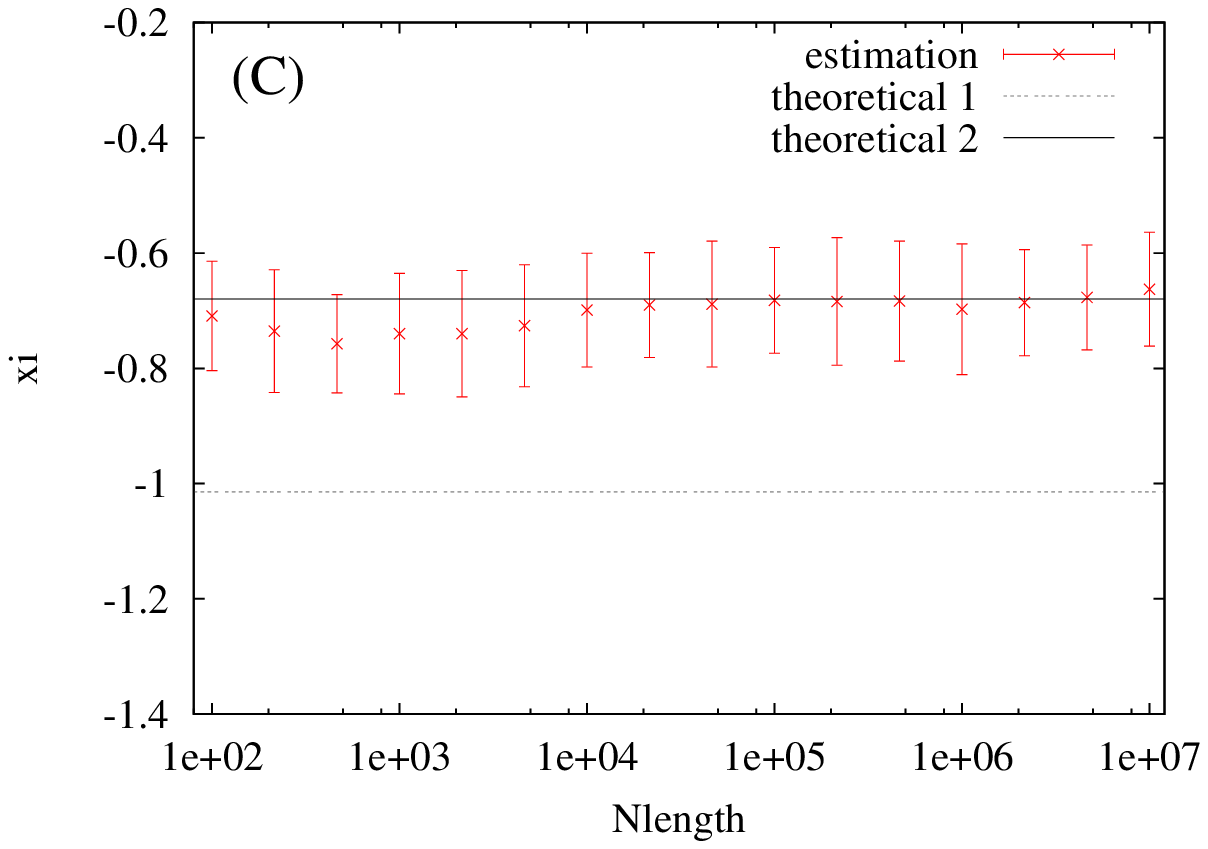}
  \caption{%
    Point estimates (crosses) and estimation uncertainty (vertical bars) of
    the tail index $\xi$ versus block length $N_{blocklen}$ for
    the solenoid map~\eqref{solenoid} under the
    observable~\eqref{obsalpha3} with $\alpha=0.3$.
    (A) $p_M$ is chosen as a point $p_M^0\in\Lambda$ as described
    in the text. (B) $p_M=p_M^t=(1+t)p_M^0$ with $t=0.1$.
    (C) $p_M=p_M^t$ with $t=1$.
    The dashed lines represent theoretically
    expected values according to~\eqref{tail-thom-alphain}
    Estimates are  obtained by the method of L-moments as for
    \figref{xis-thom-p},
    with $N_{bmax}= 10000$ and $N_{samp}=100$, see
    Appendix~\ref{app:numerical}.
  }
  \label{fig:solenoid-blocklens-alpha}
\end{figure}

\begin{figure}[p]
  \centering
  \psfrag{xi}{$\xi$}
  \psfrag{Nlength}{$N_{blocklen}$}
  \includegraphics[width=0.495\textwidth]{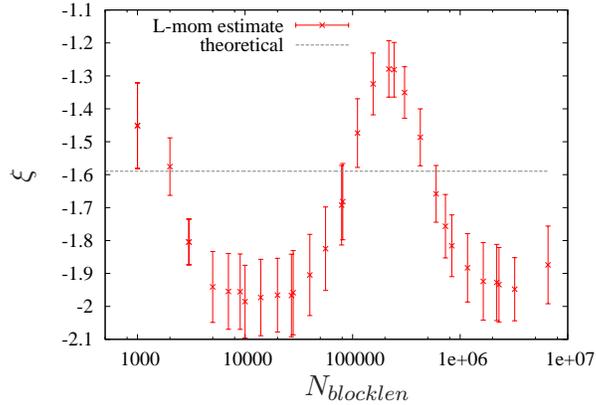}
  \caption{%
    Point estimates (crosses) and estimation uncertainty (vertical bars)
    of the tail index $\xi$ versus block length $N_{blocklen}$ for the
    H\'enon map~\eqref{henon} under the observable~\eqref{obsalpha4}
    with $\alpha=2$.
    The horizontal dashed line represents theoretically
    expected values according to~\eqref{tail-henon-alphain},
    with the Lyapunov dimension replacing the Hausdorff dimension, see text.
    Estimates are obtained by the method of L-moments as for
    \figref{xis-thom-p}, with $N_{bmax}=50000$ and $N_{samp}=100$, see
    Appendix~\ref{app:numerical}.
  }
  \label{fig:henon-blocklens-alpha}
\end{figure}

\begin{figure}[p]
  \centering
  \psfrag{xi}{$\xi$}
  \psfrag{Nlength}{$N_{blocklen}$}
  \includegraphics[width=0.495\textwidth]{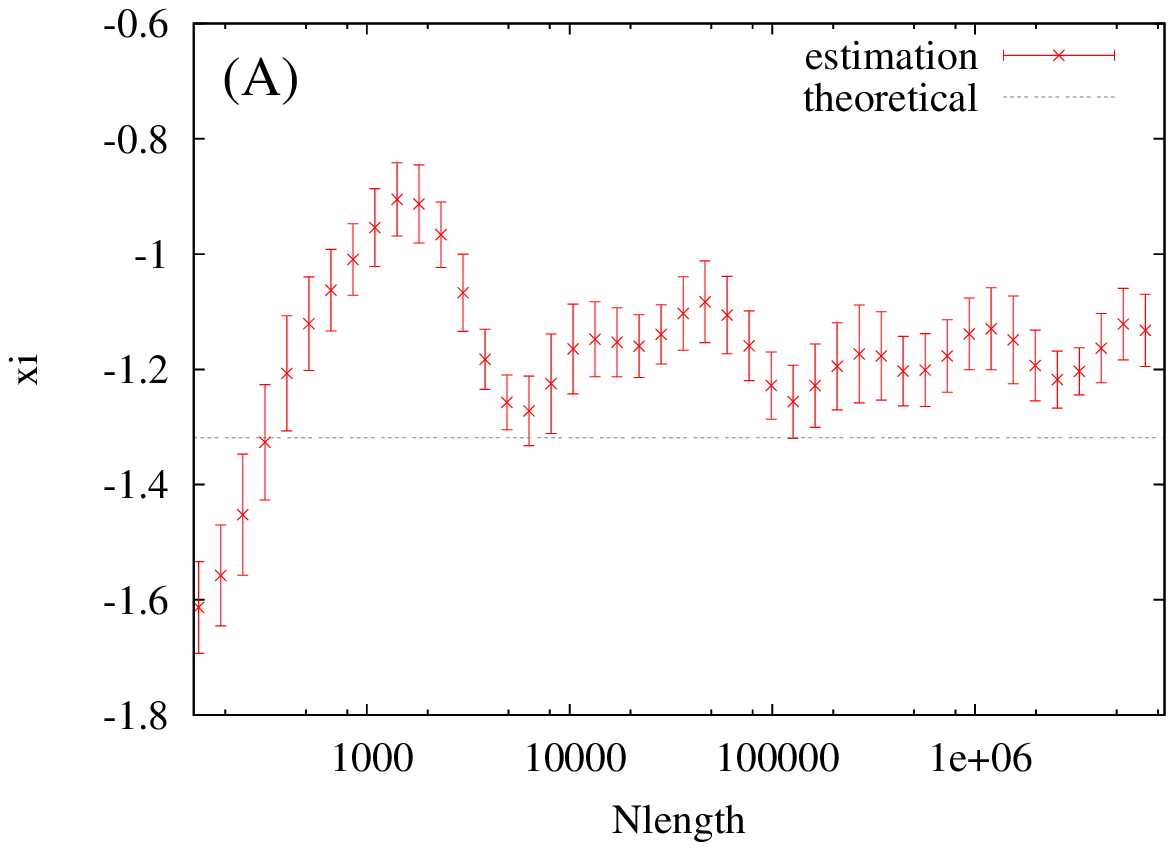}
  \includegraphics[width=0.495\textwidth]{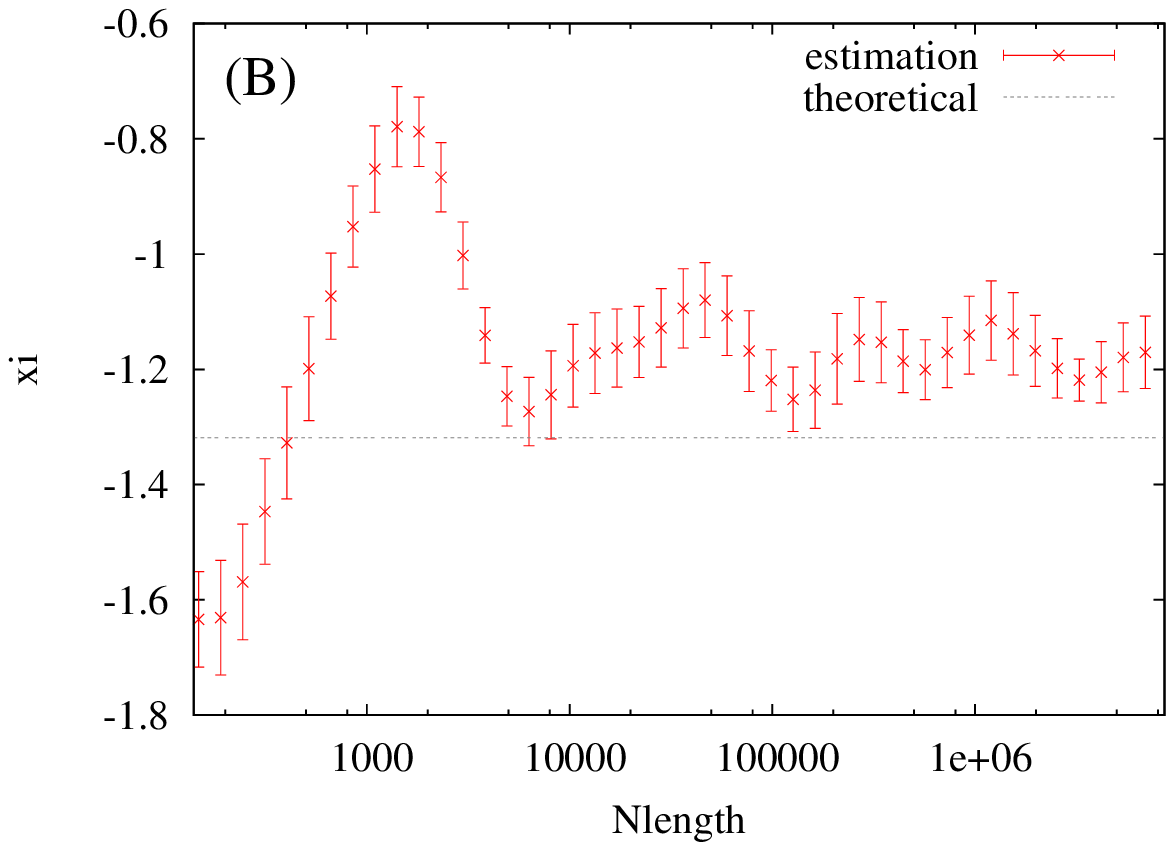}
  \caption{%
    Point estimates (crosses) and estimation uncertainty (vertical bars) of
    the tail index $\xi$ versus block length $N_{blocklen}$ for 
    the H\'enon map~\eqref{henon} under the observable~\eqref{obstheta} with
    $\theta=0$ (left) and $\theta=0.5$ (right).
    The horizontal dashed lines represent theoretically
    expected values according to~\eqref{tail-henon-thetaout}.
    Estimates are
    obtained by the method of L-moments as for \figref{xis-thom-p}, with
    $N_{bmax}= 10000$, and $N_{samp}=100$, see
    Appendix~\ref{app:numerical}.
    }
  \label{fig:henon-blocklens-theta}
\end{figure}

\begin{figure}[p]
  \centering
  \psfrag{xi}{}
  \psfrag{theta}{}
  \includegraphics[width=0.8\textwidth]{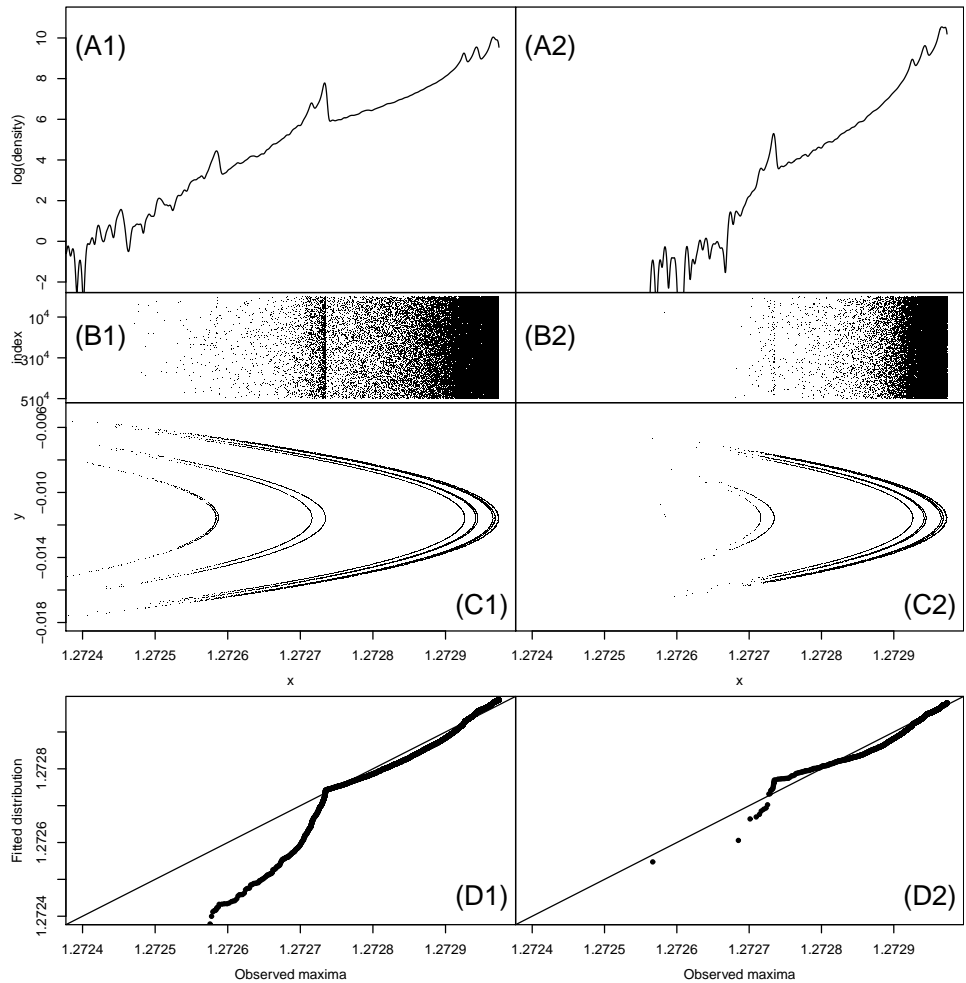}
  \caption{%
    Diagnostics of the GEV distribution fit for the H\'enon map~\eqref{henon}
    under observable~\eqref{obstheta} with $\theta=0$, with
    $N_{bmax}=5\cdot10^5$ block maxima computed over blocks of length
    $N_{blocklen}=5\cdot10^4$ (left column, A1-D1) and
    $N_{blocklen}=1.2\cdot10^5$ (right column A2-D2).
    (A1,2) Non-parametric log-densities of the block maxima, obtained by
    Gaussian kernel smoothing with bandwidth $0.000002$.
    (B1,2) Time series of the $5\cdot10^4$ block maxima (with
    the block sequential index on the vertical axis).
    (C1,2) Points on the H\'enon attractor corresponding to the block
    maxima used in (A1) and (A2), respectively.
    (D1,2) Quantile-quantile plot of the empirical distribution of
    the block maxima (horizontally) versus the fitted GEV distribution.
    }
  \label{fig:attractor}
\end{figure}
\begin{figure}[p]
  \centering
  \psfrag{xi}{$\xi$}
  \psfrag{theta}{$\theta$}
  \includegraphics[width=0.495\textwidth]{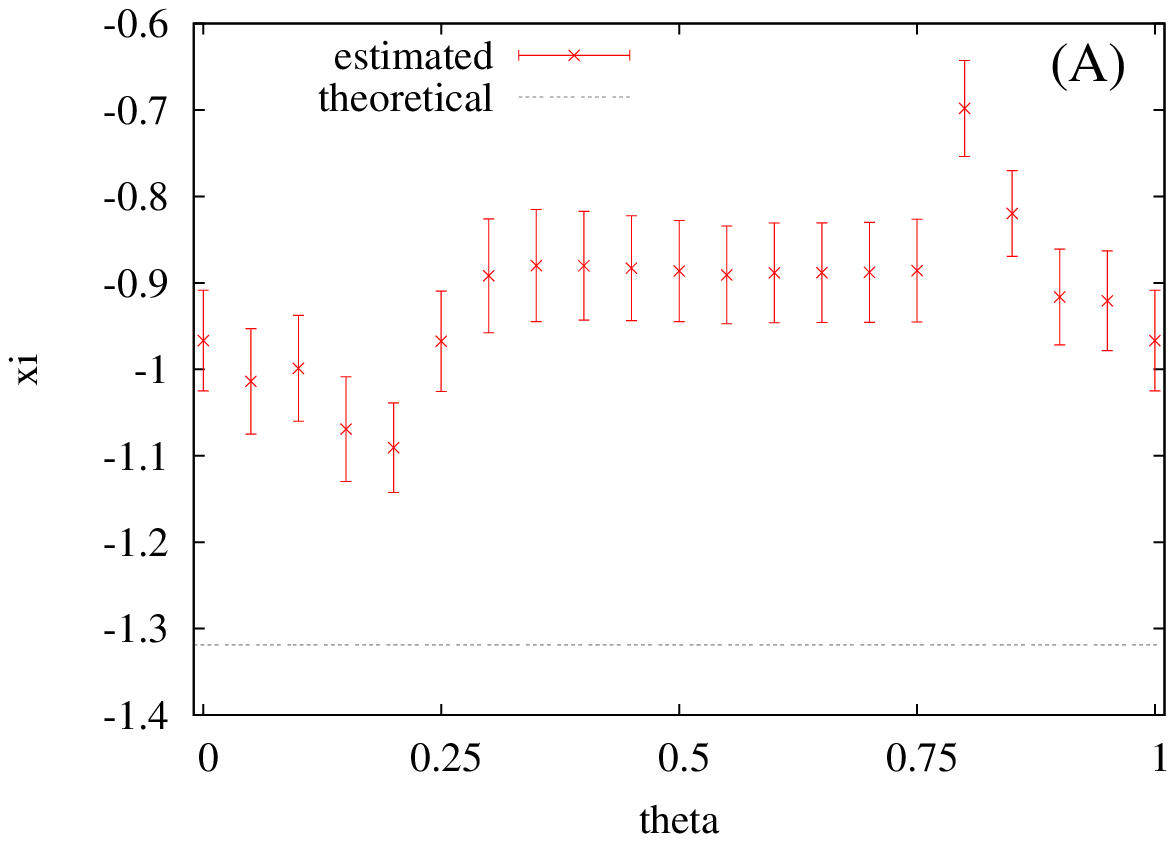}
  \includegraphics[width=0.495\textwidth]{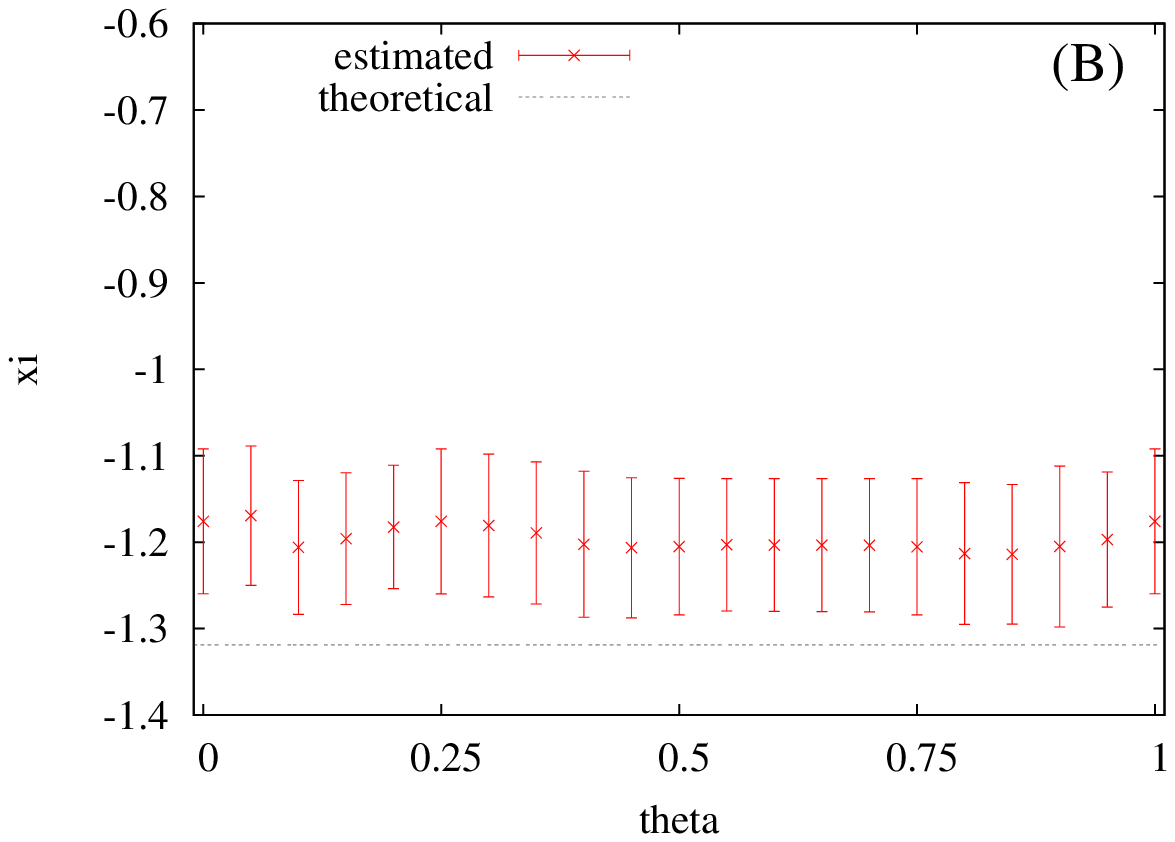}
  \includegraphics[width=0.495\textwidth]{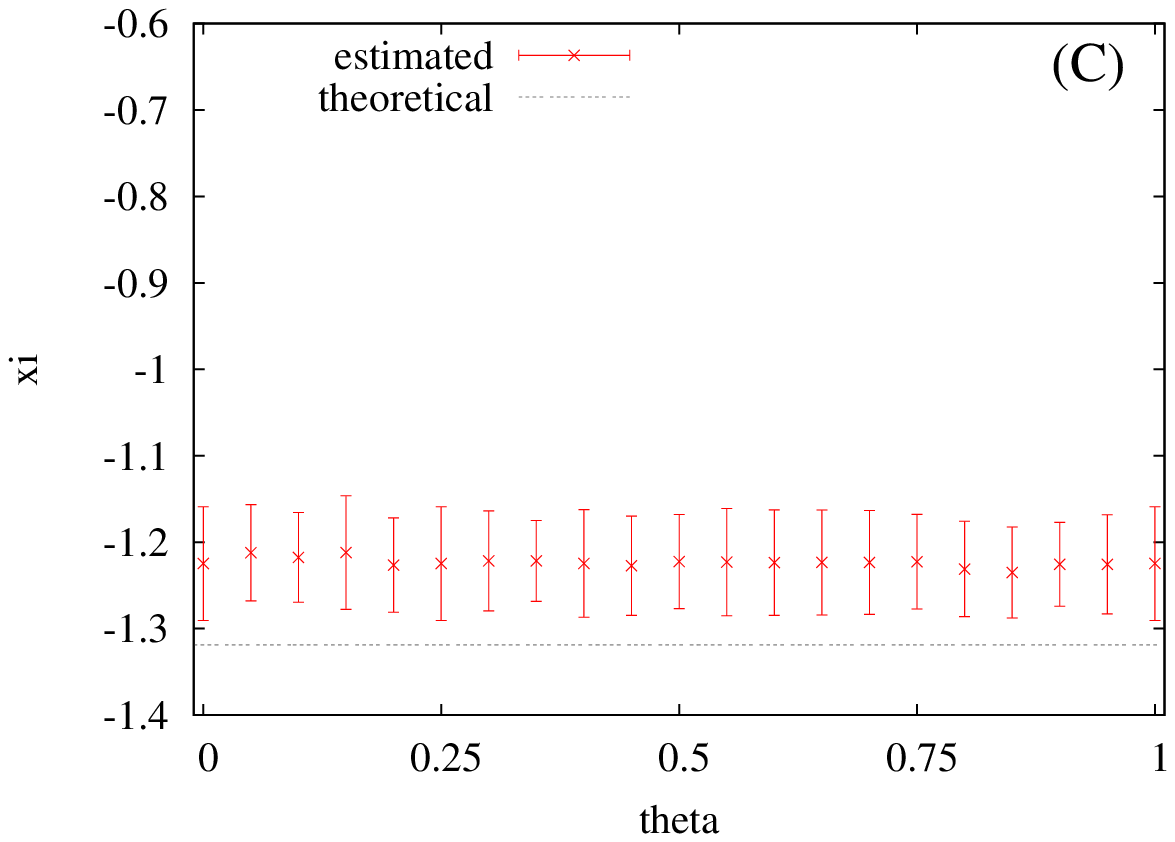}
  \includegraphics[width=0.495\textwidth]{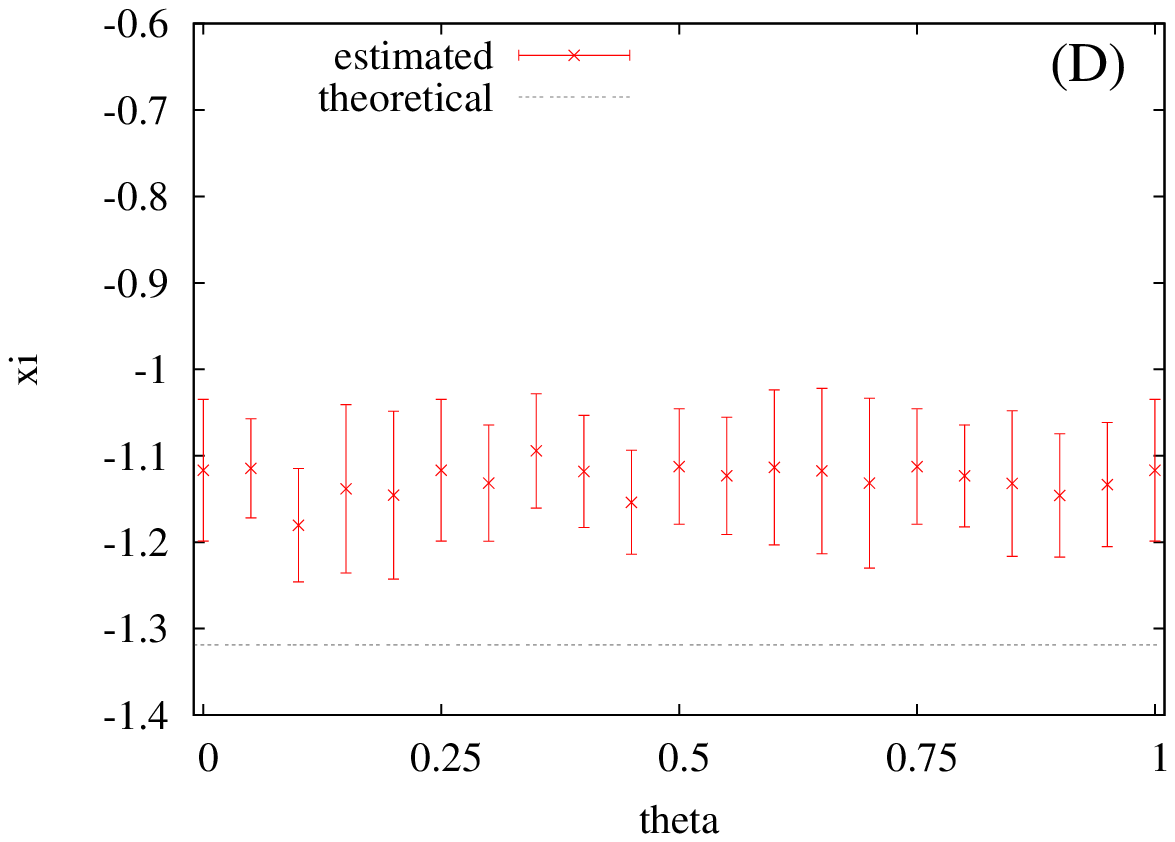}
  \caption{%
    Point estimates (crosses) and estimation uncertainty (vertical bars) of
    the tail index $\xi$ versus $\theta$ for 
    the H\'enon map~\eqref{henon} under the observable~\eqref{obstheta}
    for block lengths of (A) $N_{blocklen}=10^3$, (B) $N_{blocklen}=10^4$, 
    (C) $N_{blocklen}=10^5$, (D) $N_{blocklen}=10^6$.
    The horizontal dashed lines represent theoretically
    expected values according to~\eqref{tail-henon-thetaout}.
    Estimates are
    obtained by the method of L-moments as for \figref{xis-thom-p}, with
    $N_{bmax}= 10^5$, and $N_{samp}=100$, see
    Appendix~\ref{app:numerical}.
  }
  \label{fig:xis-henon-theta}
\end{figure}

\begin{figure}[p]
  \centering
  \psfrag{xi}{$\xi$}
  \psfrag{Nlength}{$N_{blocklen}$}
  \includegraphics[width=0.495\textwidth]{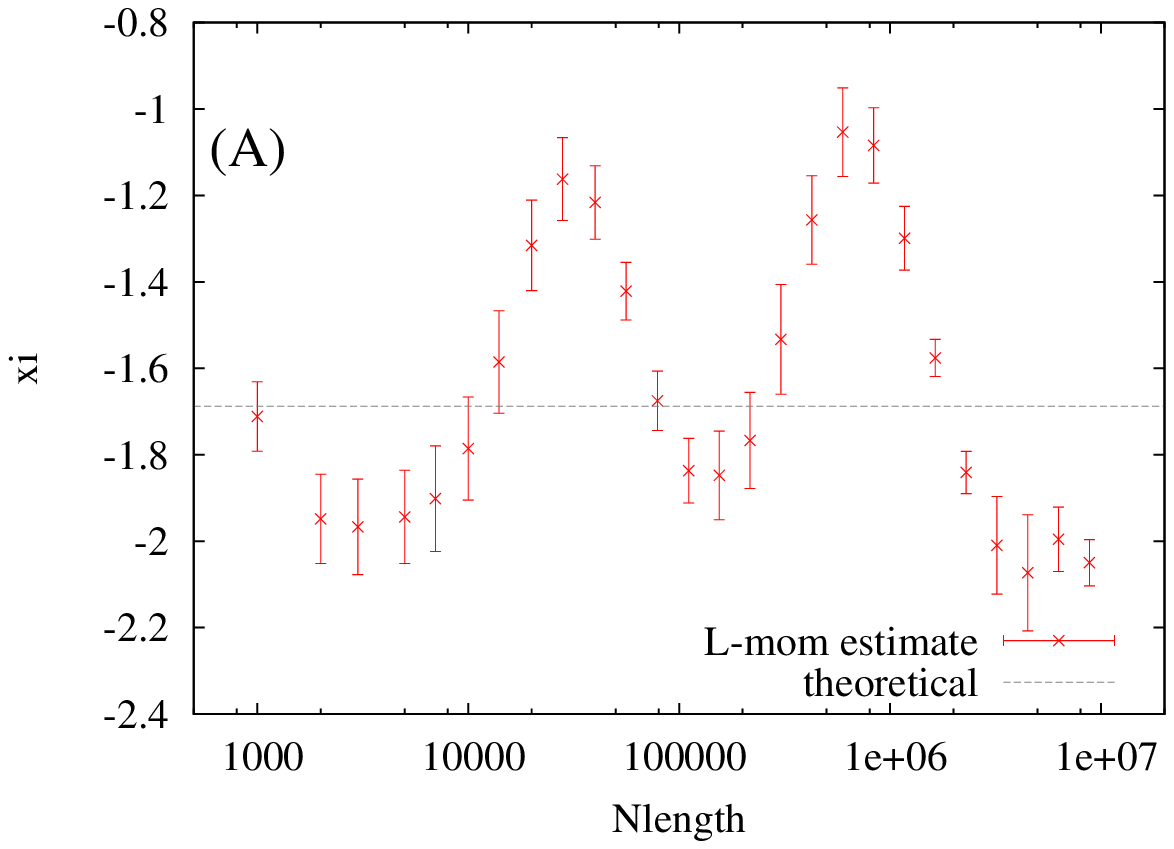}
  \includegraphics[width=0.495\textwidth]{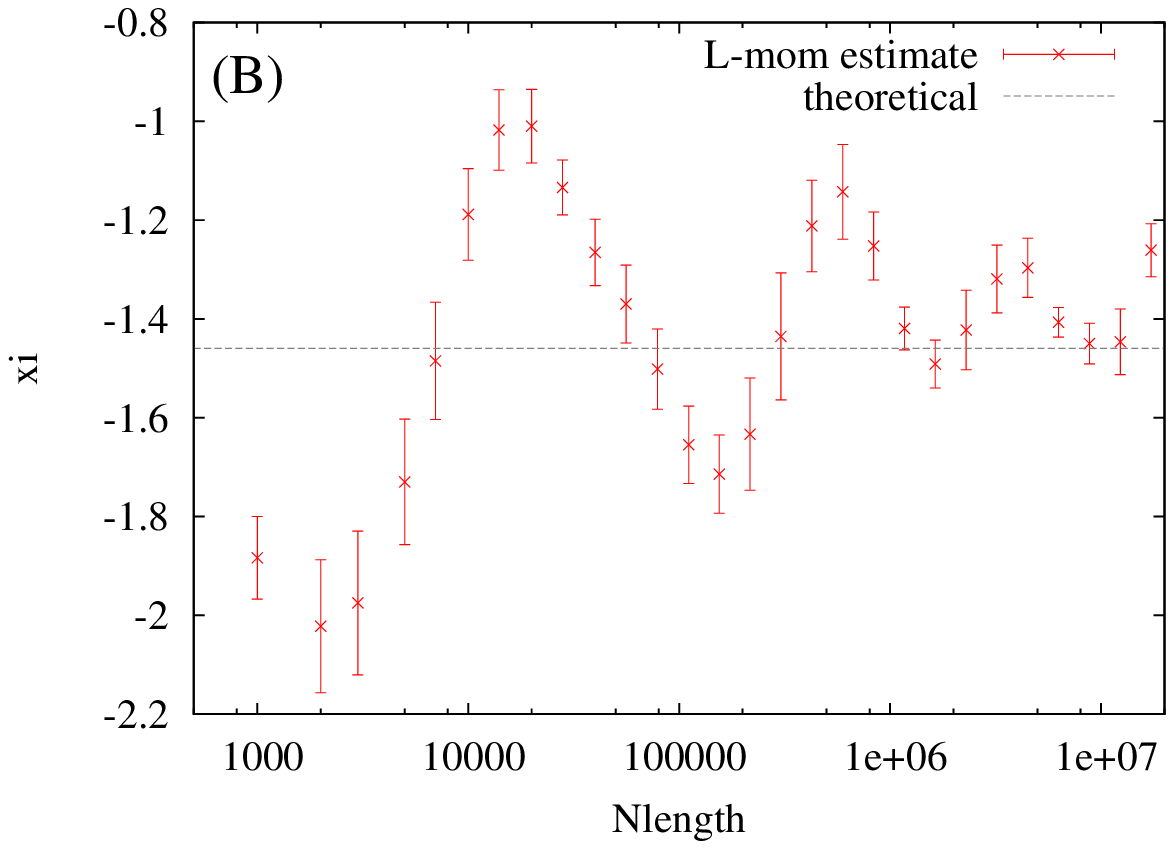}
  \caption{%
    Point estimates (crosses) and estimation uncertainty (vertical bars)
    of the tail index $\xi$ versus block length $N_{blocklen}$ for the
    Lozi map~\eqref{henon} under the observable~\eqref{obsalpha4}
    with $\alpha=2$ and (A) for a point $p_M$ belonging to the attractor;
    (B) for $p_M=(0.2,0.01)$ 
    The horizontal dashed line represents theoretically
    expected values according to~\eqref{tail-henon-alphain},
    with the Lyapunov dimension replacing the Hausdorff dimension, see text.
    Estimates are obtained by the method of L-moments as for
    \figref{xis-thom-p}, with $N_{bmax}=50000$ and $N_{samp}=100$, see
    Appendix~\ref{app:numerical}.
  }
  \label{fig:lozi-blocklens-alpha}
\end{figure}

\begin{figure}[p]
  \centering 
  \psfrag{xi}{$\xi$}
  \psfrag{Nlength}{$N_{blocklen}$}
  \includegraphics[width=0.47\textwidth]{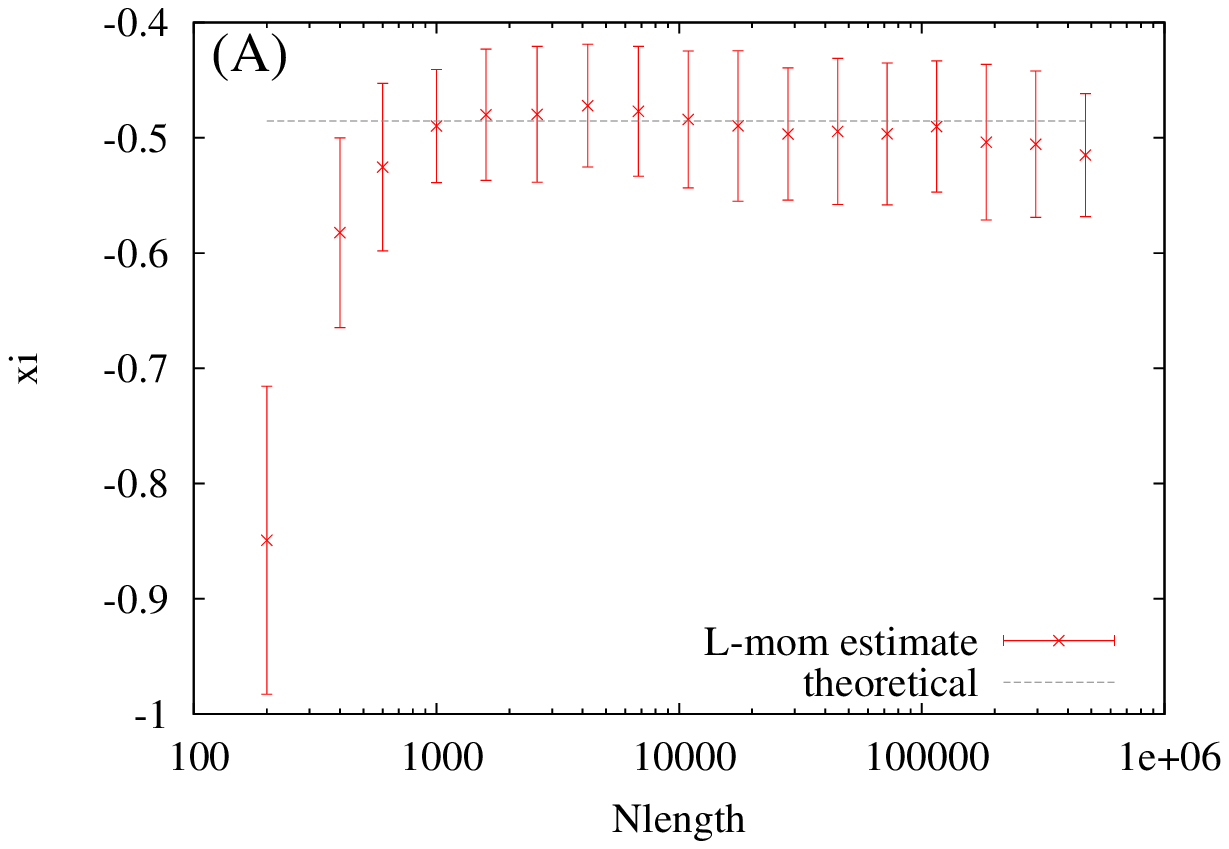}
  \includegraphics[width=0.47\textwidth]{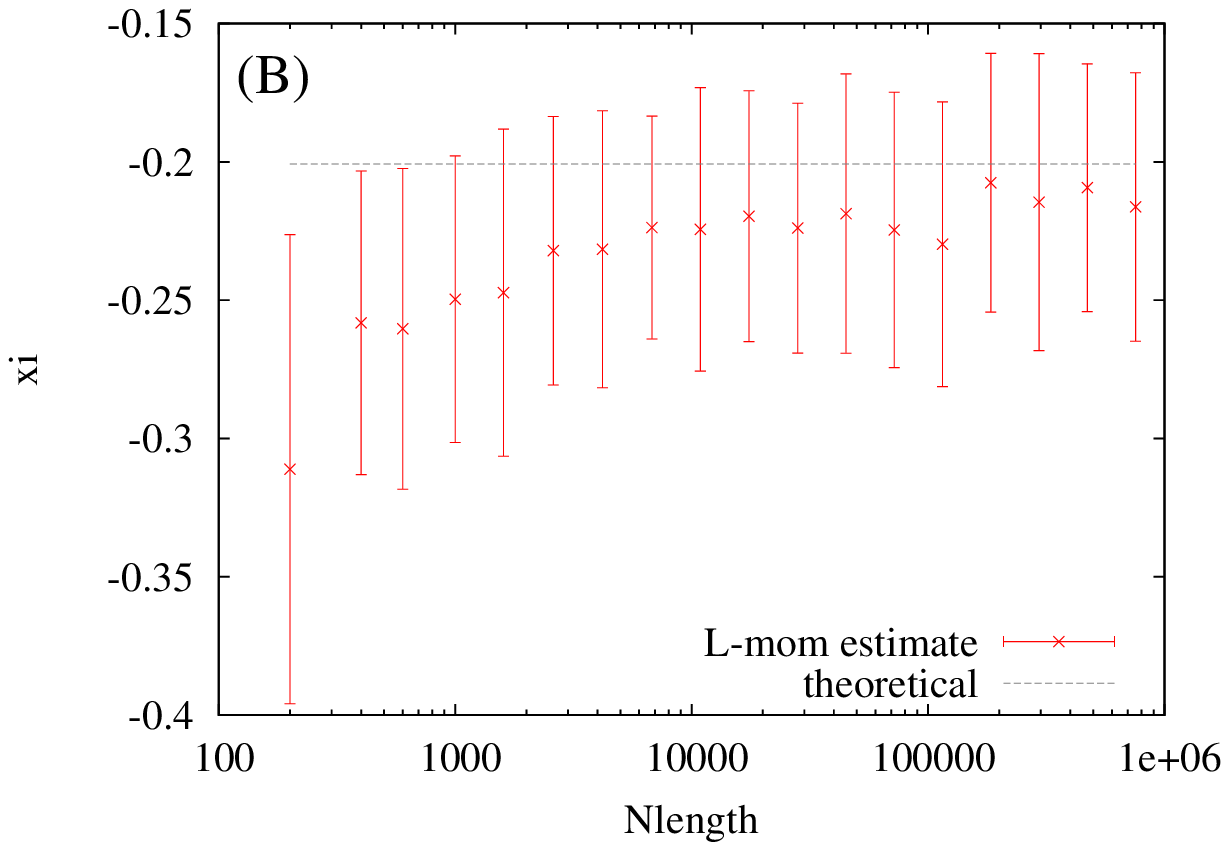}
  \caption{%
    Point estimates (crosses) and estimation uncertainty (vertical bars) of
    the tail index $\xi$ versus block length $N_{blocklen}$
    for he Lorenz63 flow~\eqref{lorenz63}
    (A) under the observable~\eqref{obsroot} where $p_M$ is chosen
    as the final point of an orbit of length $10^3$ time units starting
    from an arbitrary point, and (B) under observable~\eqref{obsflat}.
    Horizontal lines labelled by represent theoretical values.
    Estimates are obtained by the method of L-moments as for
    \figref{xis-thom-p}, with
    $N_{bmax}= 20000$ and $N_{samp}=100$, see
    Appendix~\ref{app:numerical}.
  }
  \label{fig:63-blocklens}
\end{figure}

\begin{figure}[p]
  \centering 
  \psfrag{xi}{$\xi$}
  \psfrag{Nlength}{$N_{blocklen}$}
  \includegraphics[width=0.47\textwidth]{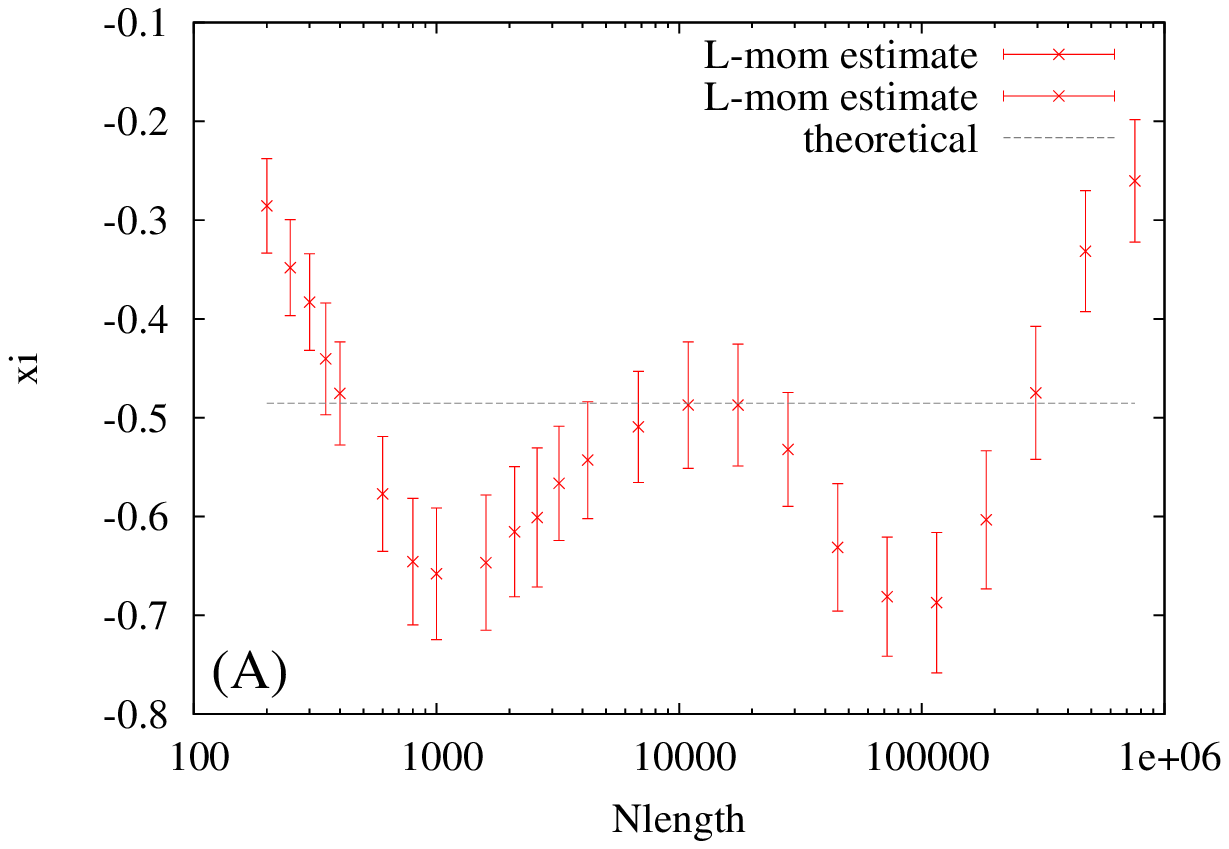}
  \includegraphics[width=0.47\textwidth]{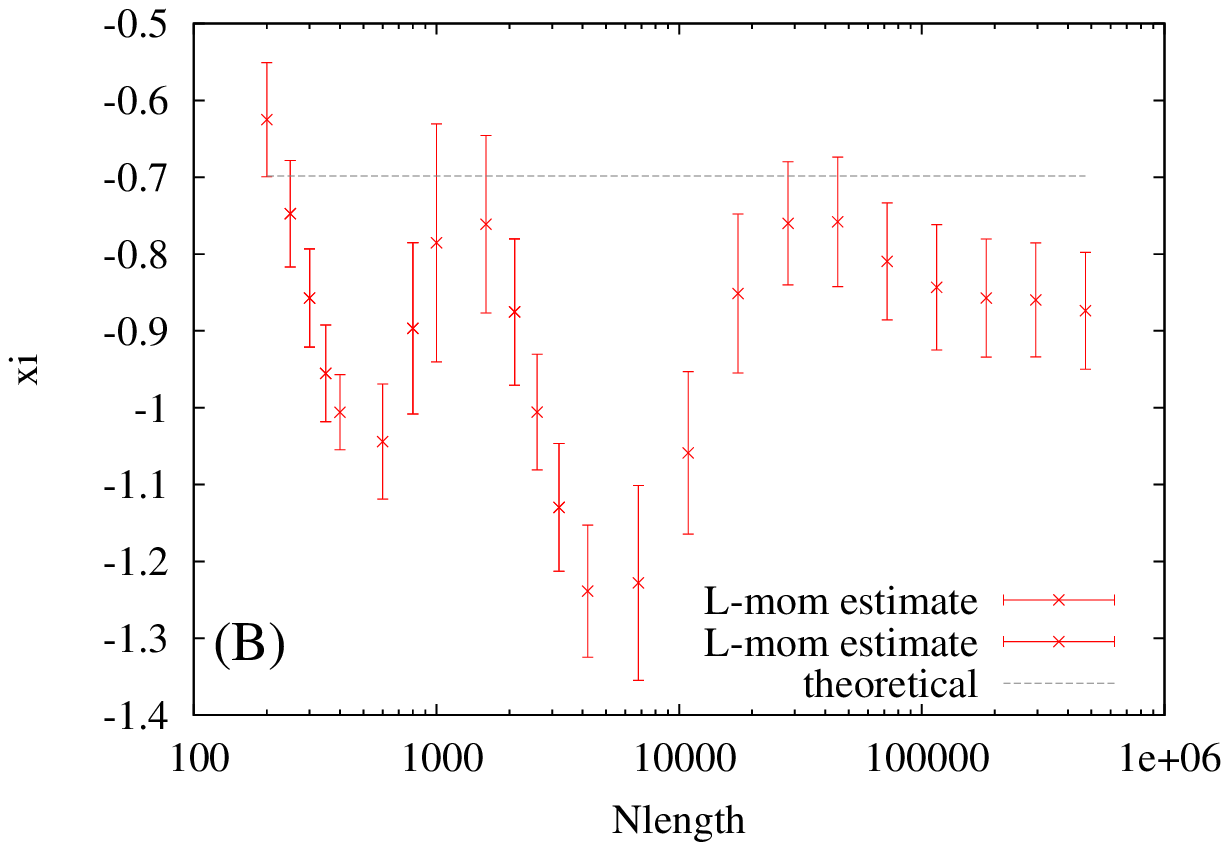}
  \caption{%
    Point estimates (crosses) and estimation uncertainty (vertical bars) of
    the tail index $\xi$ versus block length $N_{blocklen}$
    for he Lorenz63 flow~\eqref{lorenz63}
    (A) under the observable~\eqref{obsroot} where $p_M$ is chosen
    as the final point of an orbit of length $10^3$ time units starting
    from an arbitrary point, and (B) under observable~\eqref{obsflat}.
    Horizontal lines labelled by represent theoretical values
    according to~\eqref{xis-lorenz84-root} (A) and~\eqref{xis-lorenz84-flat}
    (B).
    Estimates are obtained by the method of L-moments as for
    \figref{xis-thom-p}, with
    $N_{bmax}= 20000$ and $N_{samp}=100$, see
    Appendix~\ref{app:numerical}.
  }
  \label{fig:84-blocklens}
\end{figure}

\bibliographystyle{plain}

\end{document}